\documentclass[10pt,reqno]{amsart}
\usepackage{amssymb,mathrsfs,color}
\usepackage{enumerate}
\usepackage{graphicx}
\usepackage{xcolor} 
\usepackage{geometry}
\usepackage{amsfonts,amssymb,amsmath}
\usepackage{slashed}
\usepackage{ mathrsfs }
\usepackage[titletoc,title]{appendix}
\usepackage{wrapfig}
\usepackage[multiple]{footmisc}

\usepackage{cancel}

\usepackage{cite}

\usepackage{nth}

\usepackage{hyperref}

\usepackage{marginnote}

\usepackage{pgfplots}

\usepackage{slashed}

\definecolor{green}{rgb}{0,0.8,0} 

\definecolor{deepgreen}{cmyk}{1,0,1,0.5}

\newcommand{\Del}[1]{}

\numberwithin{equation}{section}

\newtheorem{theorem}{Theorem}[section]
\newtheorem{corollary}[theorem]{Corollary}
\newtheorem{lemma}[theorem]{Lemma}
\newtheorem{proposition}[theorem]{Proposition}
\newtheorem{remark}[theorem]{Remark}

\newcommand{\mand}{{\ \ \text{and} \ \  }}

\newcommand{\angles}[2]{\langle #1,#2\rangle}

\newcommand{\ep}{\varepsilon}

\newcommand{\cross}{\times}

\newcommand{\tr}{\textrm{tr}}

\newcommand{\bfn}{{\bf n}}

\newcommand{\bara}{{\overline a}}

\newcommand{\bary}{{\overline y}}

\newcommand{\barA}{{\overline A}}

\newcommand{\barD}{{\overline D}}

\newcommand{\barT}{{\overline T}}

\newcommand{\barV}{{\overline V}}

\newcommand{\barxi}{\overline{\xi}}

\newcommand{\barsigma}{\overline{\sigma}}

\renewcommand{\hbar}{{\underline h}}

\newcommand{\nbar}{{\underline n}}

\newcommand{\Ebar}{{\underline E}}

\newcommand{\bbR}{\mathbb R}

\newcommand{\calB}{\mathcal B}
\newcommand{\calC}{\mathcal C}
\newcommand{\calD}{\mathcal D}
\newcommand{\calE}{\mathcal E}
\newcommand{\calF}{\mathcal F}

\newcommand{\calH}{\mathcal H}

\newcommand{\calL}{\mathcal L}

\newcommand{\calT}{\mathcal T}

\newcommand{\tila}{{\tilde{a}}}

\newcommand{\tilc}{{\tilde{c}}}

\newcommand{\tilf}{{\tilde{f}}}

\newcommand{\tiln}{{\tilde{n}}}

\newcommand{\tilp}{{\tilde{p}}}

\newcommand{\tilF}{{\tilde{F}}}

\newcommand{\tilG}{{\tilde{G}}}

\newcommand{\tilP}{{\tilde{P}}}

\newcommand{\tilV}{{\tilde{V}}}

\newcommand{\scE}{{\mathscr{E}}}

\newcommand{\vece}{{\vec e}}

\newcommand{\mth}{m^{\mathrm{th}}}
\newcommand{\tilcalC}{{\tilde{\calC}}}
\newcommand{\tiltilcalC}{{\tilde{\tilde{\calC}}}}

\newcommand{\snabla}{{\slashed{\nabla}}}

\newcommand{\ud}{\mathrm{d}}

\newcommand{\tiltheta}{{\tilde{\theta}}}
\newcommand{\fy}{\varphi}

\newcommand{\Thetam}{\Theta^{(m)}}
\newcommand{\Lambdam}{\Lambda^{(m)}}
\newcommand{\Sigmam}{\Sigma^{(m)}}
\newcommand{\Xm}{X^{(m)}}
\newcommand{\gm}{g_{(m)}}
\newcommand{\dgm}{g^{(m)}}
\newcommand{\Thetamp}{\Theta^{(m+1)}}
\newcommand{\Lambdamp}{\Lambda^{(m+1)}}
\newcommand{\Sigmamp}{\Sigma^{(m+1)}}

\newcommand{\Thetamm}{\Theta^{(m-1)}}
\newcommand{\Lambdamm}{\Lambda^{(m-1)}}
\newcommand{\Sigmamm}{\Sigma^{(m-1)}}

\newcommand{\dgmm}{g^{(m-1)}}
\newcommand{\calLm}{\calL^{(m)}}
\newcommand{\fm}{f^{(m)}}
\newcommand{\Bm}{B^{(m)}}
\newcommand{\Cm}{C^{(m)}}
\newcommand{\Dm}{D^{(m)}}
\newcommand{\Em}{E^{(m)}}
\newcommand{\Fm}{F^{(m)}}
\newcommand{\gammam}{\gamma_{(m)}}

\newcommand{\Phim}{\Phi_{(m)}}

\newcommand{\tilOmega}{{\tilde{\Omega}}}

\newcommand{\sdiv}{\slashed{\mathrm{div}}}
\makeatletter
\newsavebox{\@brx}
\newcommand{\llangle}[1][]{\savebox{\@brx}{\(\m@th{#1\langle}\)}%
  \mathopen{\copy\@brx\kern-0.5\wd\@brx\usebox{\@brx}}}
\newcommand{\rrangle}[1][]{\savebox{\@brx}{\(\m@th{#1\rangle}\)}%
  \mathclose{\copy\@brx\kern-0.5\wd\@brx\usebox{\@brx}}}
\makeatother
\newcommand{\bangles}[2]{\llangle #1,#2\rrangle}

\usepackage{nth}

\title{Well-posedness of the free boundary hard phase fluids  in Minkowski background and its Newtonian limit}
\author{Shuang Miao
\and  Sohrab Shahshahani
\and Sijue Wu}
\thanks{S. Miao was supported by the Fundamental Research Funds for the Central Universities in China.  S. Shahshahani was supported by the Simons Foundation grant 639284. S.~Wu was supported in part by NSF grant  DMS-1764112. }
\begin{document}

\begin{abstract}
The hard phase  model describes a relativistic barotropic irrotational fluid with sound speed equal to the  speed of light. In this paper, we prove the local well-posedness for this model in the Minkowski background with free boundary. Moreover, we show that as the  speed of light tends to infinity, the solution of this model converges to the solution of the corresponding Newtonian free boundary problem for incompressible fluids. In the appendix we explain how to extend our proof to the general barotropic fluid free boundary problem.
\end{abstract}
\maketitle
\section{Introduction}
 Let $(\mathbb R^{1+3}, m)$ be the Minkowski space-time  with  metric components $m_{\mu\nu}, \mu,\nu=0,1,2,3$,
\begin{align}\label{eq:Minkowskimetric}
	m_{00}=-1,\quad m_{11}=m_{22}=m_{33}=1, \quad \textrm{and} \quad m_{\mu\nu}=0, \quad \textrm{if}\quad \mu\slashed{=}\nu.
\end{align}
The motion of a relativistic perfect fluid occupying a domain $\Omega\subset \mathbb R^{1+3}$  in Minkowski background $(\mathbb R^{1+3}, m)$  is governed by the conservation laws
\begin{align}
&\nabla_\mu T^{\mu\nu}=0,\label{eq:con2}\\
&\nabla_\mu I^\mu =0,\label{eq:con1}
\end{align}
where 
\begin{align}\label{def EM tensor}
	T^{\mu\nu}=(\rho+p)u^{\mu}u^{\nu}+p(m^{-1})^{\mu\nu}
\end{align}
is 
the \emph{energy-momentum} tensor, and 
\begin{align}\label{def par current}
	I^{\mu}=\frak nu^{\mu}
\end{align}
is the \emph{particle current}. Here $u$ is the \emph{fluid velocity}, which is a  dimensionless future-directed unit timelike $4$-vector, so its components satisfy 
\begin{align}\label{u normalization}
\begin{split}
m_{\alpha\beta}u^\alpha u^\beta = -1,\qquad u^0>0;
\end{split}
\end{align}
and $\rho$ is the \emph{energy density}, $p$ is the \emph{pressure}, $\frak n$ is the \emph{number density} of particles, and  $\nabla$ is the covariant derivative associated to $m$.
 Let $s$ be the \emph{entropy per particle}, and $\theta$ be the \emph{temperature}. 
The laws of thermodynamics state that $\rho$ and $p$ are functions of $\frak n$ and $s$; $\rho\ge 0$, $p\ge 0$, and 
 \begin{align}\label{thermo general}
	p=\frak n\frac{\partial \rho}{\partial \frak n}-\rho,\quad \theta=\frac{1}{\frak n}\frac{\partial\rho}{\partial s}.
\end{align}
The \emph{sound speed} $\eta$ is defined by
\begin{align}\label{def sound speed}
	\eta^{2}:=\left(\frac{\partial p}{\partial\rho}\right)_{s}
\end{align}
and  is assumed to satisfy $0\le \eta\le c$, where $c$ is the  speed of light. \footnote{We have chosen not to normalize units leading to $c=1$, because as a biproduct of our well-posedness result and a priori estimates (to be discussed below) we are able to rigorously justify the Newtonian limit of the problem as $c\to\infty$. The exact powers of $c$ that appear in the equations below will be explained in the discussion following the statement of Theorem~\ref{thm:main}. Note also that even though the speed of light is taken to be $c$ rather than $1$, the coordinates $(x^0,\dots,x^3)$ on $(\bbR^{1+3},m)$ are chosen so that the metric components are \eqref{eq:Minkowskimetric}.}

Assume that the perfect fluid is  \emph{barotropic}, that is,  the pressure is a function of the energy density:
$$p=f(\rho).$$
Then \eqref{eq:con1} decouples from \eqref{eq:con2}, which by themselves form a closed system. In this case, 
both $p$ and $\rho$ are functions of a single variable $\sigma$,  defined by
\begin{align}\label{def enthalpy square}
	\rho+p=\sigma\frac{d\rho}{d\sigma}.
\end{align}
Assume that the function $f$ is strictly increasing and the integral 
$$\int_0^p\frac{dp}{\rho+p}=F(p)$$
exists. Let
$$V=\|V\|u$$
where 
$$\|V\|:=e^F$$
and $$G=\frac{\rho+p}{\|V\|^2}.$$
Then \eqref{eq:con2} can be reduced to the following equations in the fluid domain $\Omega$:
\begin{align}
&V^\nu\nabla_\nu V^\mu+\frac12 \nabla^\mu(\|V\|^2)=0\qquad\text{in }\Omega,\label{isentropic-barotropic1}\\
&\nabla_\mu \big(G(\|V\|)V^\mu\big)=0\qquad\text{in }\Omega. \label{isentropic-barotropic2}
\end{align}
See \cite{Ch-hp1} for the derivation of \eqref{isentropic-barotropic1}-\eqref{isentropic-barotropic2}.

Assume further that the fluid is \emph{irrotational}, that is
\begin{align}
V^\mu=\nabla^\mu \phi
\end{align}
for some scalar function $\phi$, and the sound speed $\eta$ equals to the  speed of light $c$. Then we arrive at the \emph{hard phase} model \footnote{This is also referred to as a \emph{stiff} or \emph{incompressible} fluid in the relativistic fluid literature.}. During the gravitational collapse of the degenerate core of a massive star, when the mass-energy density exceeds the nuclear saturation density, the sound speed is thought to approach the speed of light, cf. \cite{Ch-hp1, Lich-book, Rez-book, F-P, Zeldovich, Walecka}.  The hard phase model is an idealized model for this physical situation, cf. \cite{Ch-hp1}.  See also \cite{Ch-video}.
As derived in \cite{Ch-hp1}, the equation of state relating $p$ and $\sigma$ for the hard phase model  is 
\begin{align}\label{eq:state}
\begin{split}
p = \frac{1}{2}(\sigma^2-c^4),
\end{split}
\end{align}
the energy density satisfies 
\begin{align}\label{rho sigma rela}
	\rho=\frac12(\sigma^{2}+c^4),
\end{align}
and 
\begin{align}\|V\|=\sigma,\qquad G=1.\end{align}
For the hard phase fluid, the variable $\sigma$ is the enthalpy, which satisfies $\sigma^2> c^4$ in the fluid domain $\Omega$.  Equations
\eqref{isentropic-barotropic1}-\eqref{isentropic-barotropic2} in this case reduce to
\begin{align}\label{eq:Voriginalintro}
\begin{split}
-V_\mu V^\mu= \sigma^2, \qquad \nabla_\mu V^\mu=0, \qquad\text{in }\Omega.
\end{split}
\end{align}
 In terms of the potential function $\phi$, where $V^\mu=\nabla^\mu \phi$,  equations \eqref{eq:Voriginalintro} can be equivalently written as:
\begin{align}\label{hard-potential} -\nabla_\mu\phi \nabla^\mu\phi=\sigma^2,\qquad \Box\phi=0, \qquad\text{in }\Omega.\end{align}
Observe that the first equation in \eqref{eq:Voriginalintro} or \eqref{hard-potential}, from which \eqref{isentropic-barotropic1} follows by taking a covariant derivative, is a direct consequence of \eqref{u normalization}.
The energy-momentum tensor $T$ and the particle current $I$ for the hard phase model are
\begin{align*}
\begin{split}
T^{\mu\nu}=V^\mu V^\nu-\frac{1}{2}(m^{-1})^{\mu\nu}(V_\alpha V^\alpha+c^{4}),\quad\qquad I^\mu=V^\mu.
\end{split}
\end{align*}

 In this paper we study the motion of a hard phase fluid with free boundary, surrounded by vacuum. Let $(x^0, x^1, x^2, x^3)$ be the rectangular coordinates for a point in Minkowski spacetime ($\mathbb R^{1+3}, m)$. We  also use $t$ for $x^0$,  and $x=(x^1, x^2, x^3)$, and $\partial_t$ for $\nabla_0$, $\partial_x$ for $(\nabla_1, \nabla_2,\nabla_3)$. Let $\Omega=\{t\ge 0\}\cap\Omega$ be the fluid domain, 
$\Omega_t:=\{x^0=t\}\cap \Omega$,  and $\partial\Omega_t$ be the boundary of $\Omega_t$. Let $\partial\Omega=\bigcup_{ t\ge 0}\partial\Omega_t$ and $\calT\partial\Omega$ be the tangent space of $\partial\Omega$.
 Besides equations \eqref{eq:Voriginalintro}, we assume that on the boundary $\partial\Omega$,
\begin{align}
&\sigma^2=c^4 \qquad \text{on }\partial\Omega\label{boundary1}\\
&V\vert_{\partial\Omega}\in \calT\partial\Omega\label{boundary2}.
\end{align}
The first condition \eqref{boundary1} is equivalent to $p=0$ on $\partial\Omega$, and the second states that the fluid particle on the boundary $\partial\Omega$ will remain on $\partial\Omega$ at later times. 

To summarize, we study the Cauchy problem for the following system of equations 
\begin{align}\label{eq:Voriginalintro2}
\begin{split}
\begin{cases}
-V_\mu V^\mu= \sigma^2,\qquad &\mathrm{in~}\Omega  \\
\nabla_\mu V^\mu=0,\quad dV=0,\qquad &\mathrm{in~}\Omega\\
\sigma^2=c^4,\qquad &\mathrm{on~}\partial\Omega\\
V\vert_{\partial\Omega}\in \calT\partial\Omega
\end{cases},
\end{split}
\end{align}
where $V$, $\sigma^2$ and $\partial\Omega$ are the unknowns. Here $dV$ is the exterior derivative of the 1-form $V$ (we will abuse notation to write $V$ for both the vectorfield and the corresponding 1-form $V^\flat$), so $dV=0$ means that $V$ is irrotational.  We assume that the initial data $(V_0, \sigma_0)$ and  $\Omega_0$ are given and satisfy
\begin{align}\label{eq:data}
\begin{split}
\begin{cases}
-(V_0)_\mu (V_0)^\mu= \sigma_0^2\ge c^4, \quad V^0>0, \qquad &\mathrm{in~}\Omega_0  \\
\nabla_\mu (V_0)^\mu=0,\quad dV_0=0,\qquad &\mathrm{in~}\Omega_0\\
\sigma_0^2=c^4,\qquad &\mathrm{on~}\partial\Omega_0\\
\nabla_\mu \sigma_0^2\nabla^\mu \sigma_0^2\ge c_0^2 c^4>0\qquad&\text{on }\partial\Omega_{0}
\end{cases},
\end{split}
\end{align} 
where $c_0$ is a nonzero constant.\footnote{The last assumption in \eqref{eq:data} is the relativistic Taylor sign condition, which we will explain next.} 
We show that if the domain $\Omega_0$ and  $(V_0, \sigma_0)$ are sufficiently smooth and satisfy \eqref{eq:data},  then \eqref{eq:Voriginalintro2} is uniquely solvable in a time interval $[0, T_0]$, with $T_0>0$ depends only on the initial data, 
and the solution has the same regularity as the initial data. Furthermore we show that for suitably given initial data relative to the  speed of light $c$, the life span of the solution is $T_0=cT_1$, with $T_1>0$ independent of $c$;  and as $c\to\infty$, the solution of \eqref{eq:Voriginalintro2}-\eqref{eq:data} converges to the solution of the corresponding free boundary problem of the Newtonian incompressible fluid.

The free boundary problem \eqref{eq:Voriginalintro2} is a fully nonlinear system defined on a free domain. The key to solving \eqref{eq:Voriginalintro2} is to reduce it to a quasilinear system. 

Consider the Newtonian counterpart of our problem, the water wave problem, which concerns the motion of an incompressible, irrotational ideal fluid in  free domains, neglecting surface tension. Let $\tilV$ be the velocity and $\tilp$ the pressure. Assume that the fluid occupies the domain $\tilOmega_t$ at time $t$, with boundary $\partial\tilOmega_t$,  and let $\tiln$ be the unit outward normal to $\partial\tilOmega_t$.  An important condition for the well-posedness of the water wave system is the Taylor sign condition\footnote{It is known that the failure of this condition leads to instability; cf. \cite{GTay1}.}:
\begin{align}\label{taylor-ww}
-\frac{\partial\tilp}{\partial\tiln}\ge \tilc_0>0\qquad \text{on }\partial\tilOmega_t.
\end{align}
It was shown in \cite{Wu97, Wu99} that for the water wave system,  taking one material derivative to the Euler equation gives rise to a quasilinear  equation.  This quasilinear equation is\footnote{Let $\bf g$ be the gravity. The motion of water waves is described by 
\begin{align*}
\begin{split}
\begin{cases}
D_t \tilV+\nabla \tilp={\bf g},\qquad &\mathrm{in~}\tilOmega_t  \\
\text{div}\, \tilV=0,\quad \text{curl}\, \tilV=0,\qquad &\mathrm{in~}\tilOmega_t
\end{cases},\qquad
\begin{cases}
\tilp=0,\qquad &\mathrm{on~}\partial\tilOmega_t\\
(1,\tilV)\in \calT\partial\tilOmega,\qquad &\mathrm{on~}\partial\tilOmega
\end{cases},
\end{split}
\end{align*}
Taking a $D_t$ derivative to the Euler equation and computing the commutator $[D_t, \nabla]\tilp= -\nabla \tilp\cdot\nabla\tilV$ yields 
\begin{align}\label{ww-quasi}
D_t^2 \tilV-\nabla\tilp\cdot\nabla\tilV=-\nabla D_t\tilp.
\end{align}
Since $\tilp=0$ on $\partial\tilOmega_t$,  $-\nabla\tilp=a\tiln$ on the free boundary $\partial\tilOmega_t$, with $a=-\frac{\partial \tilp}{\partial\tiln}$. Also $\Delta \tilV=0$. Restricting \eqref{ww-quasi} on $\partial\tilOmega_t$ gives \eqref{eq:tilV0}.}
\begin{align}\label{eq:tilV0}
\begin{cases}
(D_t^2+\tila\nabla_\tiln)\tilV=-\nabla D_t\tilp,\qquad& \mathrm{on~}\partial\tilOmega_t\\
\Delta \tilV =0,\qquad &\mathrm{in~}\tilOmega_t
\end{cases},
\end{align}
where $D_t=\partial_t+\tilV\cdot\nabla$ is the material derivative, $\tila=-\frac{\partial \tilp}{\partial\tiln}$, and $\nabla_\tiln$ is the Dirichlet-Neumann operator. Observe that by Green's identity,  the Dirichlet-Neumann operator $\nabla_\tiln$ is a positive operator:
$$\int_{\partial\tilOmega_t} v \nabla_\tiln v \,ds=\int_{\tilOmega_t}|\nabla v|^2\,dx>0$$
for $v$ harmonic. In \cite{Wu97, Wu99} boundary integrals were used to express the quantities $\tila$ and $-\nabla D_t\tilp$, and the first equation in \eqref{eq:tilV0} was shown to be a quasilinear equation of hyperbolic type, with the left-hand side consisting of principal terms, and a local-wellposedness result was obtained. 
In \cite{Ch-Lin1} a similar quasilinear equation as the first in \eqref{eq:tilV0} was used to study the more general case that allows for non-zero vorticity (see also \cite{ZZ}).  Instead of boundary integrals, elliptic regularity estimates and equations
\begin{align}\label{ww-p}
\begin{cases}
-\Delta\tilp=\partial_i \tilV^j\partial_j\tilV^i,\quad\text{in }{\tilOmega_t}, \qquad\qquad \qquad \tilp=0 \quad\text{on }\partial\tilOmega_t\\
-\Delta D_t\tilp=\partial_i \tilp\Delta\tilV^i+G(\partial \tilV, \partial^2\tilp),\quad\text{in }{\tilOmega_t}, \qquad D_t\tilp=0 \quad\text{on }\partial\tilOmega_t
\end{cases}
\end{align}
  were used in \cite{Ch-Lin1} to control the regularity of $\nabla\nabla\tilp$ and $\nabla\nabla D_t\tilp$ via the regularity of $\Delta \tilp$ and $\Delta D_t\tilp$, and an a priori estimate was obtained under the assumption that the Taylor sign condition \eqref{taylor-ww} holds.

This motivates us to take a $D_V:=V^{\mu}\nabla_{\mu}$ derivative of the equation \eqref{isentropic-barotropic1}  
to obtain\footnote{We know $\|V\|=\sigma$ for the hard phase model. In the rest of this paper we will not work with the fluid velocity $u$ again, and will refer to $V$ simply as the velocity.}
\begin{align}\label{eq:DVinside}
D_V^2V^\nu-\frac12\nabla_\mu\sigma^2\nabla^\mu V^\nu=-\frac12\nabla^\nu D_V \sigma^2.
\end{align}
Since $\sigma^2\equiv c^4$ on $\partial\Omega$ by assumption, $\nabla\sigma^2$ is normal (with respect to $m$) to $\partial\Omega$. 
Let  $n$ be the unit outward pointing (spacetime) normal to $\partial\Omega$.  Assume that the relativistic Taylor sign condition
\begin{align}\label{rel-taylor}
\nabla_\mu \sigma^2\nabla^\mu \sigma^2 >0\qquad\text{on }\partial\Omega
\end{align}
holds.\footnote{This is consistent with the fact that $\sigma^2>c^{4}$ in the fluid domain $\Omega$. Assume that the relativistic Taylor sign condition $\nabla^{\mu}\sigma_{0}^{2}\nabla_{\mu}\sigma_{0}^{2}\geq c_{0}^{2}c^{4}>0$ holds initially and that the solution exists. Then by continuity, \eqref{rel-taylor} will remain to hold for a short period of time.  During this time $\nabla\sigma^{2}$ is space-like, and hence $\partial\Omega$ is timelike.}
Then we 
 can write 
\begin{align*}
\begin{split}
\nabla \sigma^2  = -an,\qquad \mathrm{on~}\partial\Omega,
\end{split}
\end{align*}
where $a>0$ is given by
\begin{align}\label{eq:aintro}
\begin{split}
a=\sqrt{\nabla_\mu\sigma^2 \nabla^\mu\sigma^2}.
\end{split}
\end{align}
 Observe that  the second equation in \eqref{eq:Voriginalintro2} gives 
\begin{align}\label{box v}
\Box V=0,
\end{align}
where $\Box$ is the D'Alembert operator.   Going back to \eqref{eq:DVinside} and restricting it to the boundary we get
\begin{align}\label{eq:V1}
\begin{cases}(D_V^2+\frac{1}{2}a\nabla_n) V^\nu =-\frac{1}{2}\nabla^\nu D_V\sigma^2,\qquad &\mathrm{on~}\partial\Omega\\
\Box V^\nu=0,\qquad &\mathrm{in~}\Omega
\end{cases}.
\end{align}
Here $\nabla_n$ can be thought of as the \emph{hyperbolic Dirichlet-Neumann map}. That is, $\nabla_n \theta$ is the normal derivative on $\partial\Omega$ of the solution $\Theta$  for the wave equation 
$$\begin{cases}\Box\Theta=0\qquad\text{in }\Omega\\
\Theta=\theta \qquad\text{on }\partial\Omega\end{cases},$$  
 provided the initial data for $\Theta$ are given. Applying $\nabla_\nu$ to \eqref{isentropic-barotropic1} and \eqref{eq:DVinside} and summing over $\nu$ yields
\begin{align}\label{eq:sigma1a-intro}
\begin{cases}
\Box\sigma^2=(-2\nabla^\mu V^\nu) (\nabla_\mu V_\nu)\quad\text{in }\Omega,\qquad\sigma^2=c^4\quad \text{on }\partial\Omega\\
\Box D_V\sigma^2=4(\nabla^\mu V^\nu)\nabla_\mu \nabla_\nu \sigma^2+4(\nabla^\lambda V^\nu)(\nabla_\lambda V^\mu)(\nabla_\nu V_\mu)\quad\text{in }\Omega, \qquad D_V\sigma^2=0\quad \text{on }\partial\Omega.
\end{cases}
\end{align}
Here $D_V\sigma^2=0$ on $\partial\Omega$ follows from the boundary conditions \eqref{boundary1}-\eqref{boundary2}.

Although equations \eqref{eq:V1}-\eqref{eq:sigma1a-intro} take  similar forms as the quasilinear system \eqref{eq:tilV0}-\eqref{ww-p} of the water waves, there are fundamental differences. Most notably the Laplacian  $\Delta$ for water waves is replaced by the D'Alembertian $\Box$ for our problem. 
It is not clear whether the  hyperbolic Dirichlet-Neumann map $\nabla_n$ is still positive, and if \eqref{eq:V1}-\eqref{eq:sigma1a-intro} is a quasilinear system with a lower order right-hand side. Most importantly, it is not clear in which functional analytic settings the Cauchy problem for \eqref{eq:V1}-\eqref{eq:sigma1a-intro} can be solved.

 In \S\ref{general estimates} we  will develop the necessary analytic tools to resolve these issues. In particular in Lemma 
 ~\ref{lem:energy}, we will show that  in the energy functional for equation
 \begin{align}\label{eq:theta}
\begin{cases}(D_V^2+\frac{1}{2}a\nabla_n) \Theta =f,\qquad &\mathrm{on~}\partial\Omega\\
\Box \Theta=g,\qquad &\mathrm{in~}\Omega
\end{cases},
\end{align}
the   hyperbolic Dirichlet-Neumann map $\nabla_n$ controls $\int_{\Omega_T} |\nabla_{t,x} \Theta|^2\,dx$,  minus some integrals involving lower order terms and the initial data, provided $V$ is timelike; and the energy functional for \eqref{eq:theta} controls 
$$\int_{\partial\Omega_T}|D_V \Theta|^2\,dS+ \int_{\Omega_T}  |\nabla_{t,x} \Theta|^2\,dx.$$
Using the tools in \S\ref{general estimates}, we will  show that the quantities $a$ and $-\frac{1}{2}\nabla^\nu D_V\sigma^2$ in \eqref{eq:V1}-\eqref{eq:sigma1a-intro} are of lower order. We will first solve the Cauchy problem for the quasilinear system \eqref{eq:V1}-\eqref{eq:sigma1a-intro}, and then prove that a solution of \eqref{eq:V1}-\eqref{eq:sigma1a-intro} is also a solution of the Cauchy problem \eqref{eq:Voriginalintro2}-\eqref{eq:data}, provided the data for the Cauchy problem of the system \eqref{eq:V1}-\eqref{eq:sigma1a-intro} are derived from those of equation \eqref{eq:data}. 

To solve the system \eqref{eq:V1}-\eqref{eq:sigma1a-intro}, 
we will first construct an energy functional by applying the basic energy estimate in Lemma~\ref{lem:energy} to $\Theta=D_V^j V$ for  integers $0\le j \le k$, and prove an a priori estimate.  We find it advantageous to work with $D^j_V V$ instead of other types of derivatives of $V$, since we know $D_V^j \sigma^2=0$ on $\partial\Omega$, for $j\ge 0$ integers, thanks to the boundary conditions \eqref{boundary1}-\eqref{boundary2}. Using equation \eqref{eq:V1}, we can control the Sobolev norms   involving derivatives in all directions by showing $D_V^2\approx \nabla$ on $\Omega$. We will also include in our energy functional a quantity involving some $L^2$ integrals of $D_V^j\sigma^2$, see \eqref{energy}, to control the quantities $a$ and $-\frac{1}{2}\nabla^\nu D_V\sigma^2$ in \eqref{eq:V1}. To prove the existence of solutions of the system \eqref{eq:V1}-\eqref{eq:sigma1a-intro}, we will use the Galerkin method, discretizing the system  \eqref{eq:V1}-\eqref{eq:sigma1a-intro} into a system of finite dimensional ODEs. This appears to be one of the most natural methods to construct approximate systems for \eqref{eq:V1}-\eqref{eq:sigma1a-intro}, since it allows us to almost effortlessly extend our proof for the a priori estimate for  \eqref{eq:V1}-\eqref{eq:sigma1a-intro} to the discretized system. In Section \ref{subsubsec:apriori} we will give an extended outline of the approach in this paper.  

We now state our results. Let $(V_0,\sigma_0)$ and $\Omega_0$ be given and satisfy \eqref{eq:data}.
Observe that we can use the second equation in \eqref{eq:data} to compute the covariant derivative $\nabla^0 V_0$, hence 
$D_{V_0} V_0$, and subsequently the higher order derivatives $D_{V_0}^k V_0$ and $D_{V_0}^k\sigma^2_0$. 

Assume that there is a diffeomorphism $Y: \Omega_0\to B$ with $B$ the unit ball in $\bbR^3$, and assume  that the following regularity and compatibility conditions are satisfied by the data:
\begin{align}\label{eq:data-reg-Euler-intro}
\begin{split}
\partial_x^aY\in L^2(\Omega_0),\qquad &2a\leq K+2,\\
\partial_{x}^{a}D_{V_0}^kV_0, ~\partial_{x}^{a}D_{V_0}^{k+1}\sigma_0^2\in L^2(\Omega_0),\qquad &k\leq K+1,\quad 2a+k\leq K+2,\\
D_{V_0}^{K+1}V_{0}\in L^2(\partial \Omega_0),\qquad&\\
D_{V_0}^k\sigma_0^2\in H^1_0(\Omega_0),\qquad&k\leq K+1.
\end{split}
\end{align}
\begin{theorem}\label{thm:main}
Let $K$ be sufficiently large. Then for initial data \eqref{eq:data} satisfying the regularity and compatibility conditions \eqref{eq:data-reg-Euler-intro}, 
there exists $T_0>0$,  a unique domain $\Omega=\cup_{t\in[0,T_0]}\Omega_t$, and a unique solution $(V, \sigma)$ to \eqref{eq:Voriginalintro2} for $t\in [0, T_0]$. Moreover, $(V,\sigma)$ and $\Omega$ satisfy the same regularity properties as their initial data. 
\end{theorem}
\begin{remark}
By examining the proof of Theorem~\ref{thm:main} one can find a numerical value for $K$ (for instance $K=20$ is sufficient), but since achieving the optimal regularity is not our concern in this work, we have stated the theorem without specifying the optimal value of $K$ that can be derived from our proof. 
\end{remark}
\begin{remark}
The life span $T_0$ depends only on the norm of the initial data,  the constant $c_0$ and the  speed of light $c$. For suitably set up data relative to $c$,  $T_0=cT_1$, with $T_1$ independent of $c$, see Theorem~\ref{thm:limit} for the precise statement.
 \end{remark}
\begin{remark}
For the top order $D_{V_0}$ derivative of $V_0$ we use the weak formulation to define $D_{V_0}^{\ell+1}V_0$ on the boundary.  See for instance Lemma \ref{lem:Lambdanormal}. 
\end{remark}

To treat the Newtonian limit, we need to introduce some notation.
We will use $t'$ and $x'$ to denote the time and space variables in the non-relativistic setting. We define the flow map of $V$ by  $\frac{\ud\Phi^{i}(t,\cdot)}{\ud t}=\left(\frac{V^{i}}{V^{0}}\right)(t,\Phi(t,\cdot))$, so that $\Phi(t,\cdot)$ maps $\Omega_0$ to $\Omega_t$, and let $\Psi_c(t',\cdot)=\Phi(ct',\cdot)$.
The non-relativistic fluid velocity components are defined by
\begin{align}\label{eq:Newvdef}
\begin{split}
v^i(t',x')=c\frac{u^i(ct',x')}{u^0(ct',x')},
\end{split}
\end{align}
where $x'=\Psi_c(t',x_0')$, $x_0'\in\Omega_0$, so 
\begin{align*}
\begin{split}
u^0(ct',x')=\frac{1}{\sqrt{1-|v(t',x')|^2/c^2}};\qquad u^i(ct',x')=\frac{v^i(t',x')/c}{\sqrt{1-|v(t',x')|^2/c^2}},\quad\text{for } i=1,2,3.
\end{split}
\end{align*}
Unlike for the fluid velocity, for some of the thermodynamical variables one has to subtract the contribution of the rest mass to arrive at a quantity which has a limit as $c\to\infty$ (see for instance  Section 1.1 of \cite{Ch-sd}). In particular, the non-relativistic enthalpy, $h$, is defined by (again with $x'=\Psi_c(t',x_0'), x_0'\in\Omega_0$)
\begin{align}\label{eq:Newhdef}
\begin{split}
h(t',x')=\sigma(ct',x')-c^2,\qquad 
\text{so }\quad \sigma^2(ct',x')= c^4+2c^2h(t',x')+h^2(t',x').
\end{split}
\end{align}
The robust a priori estimates we establish in this paper allow us to conclude that in the limit $c\to\infty$, $(h,v)$ converge to a solution $(h',v')$ of the Newtonian problem 
\begin{align}\label{eq:Newtonian-problem}
\begin{cases}
\partial_{t'}v'+v'\cdot\nabla_{x'} v'=-\nabla_{x'}h',\qquad &\mathrm{in~}\calD_{t'}\\
\nabla_{x'}\cdot v'=0,\qquad \nabla_{x'}\cross v' =0, \qquad &\mathrm{in~}\calD_{t'}\\
h'=0,\qquad &\mathrm{on~}\partial\calD_{t'}\\
(1,v')\in\calT \left(\cup_{t'}(t',\partial\calD_{t'})\right)
\end{cases}.
\end{align}
Here $\calD_{t'}=\Psi(t',\Omega_{0})$ with $\Psi(t',\cdot)$ defined by 
\begin{align}\label{def Psi}
	\frac{\ud\Psi(t',\cdot)}{\ud t'}=v'(t',\Psi(t',\cdot)).
\end{align}
To state our result precisely, we introduce the notation
\begin{align}\label{eq:barVdefintro}
\begin{split}
\barV(t,x):=c^{-1}V(t,x),\qquad \barsigma^2(t,x)=c^{-2}\sigma^2(t,x)-c^2,
\end{split}
\end{align}
and define the renormalized initial energy norm
\begin{align*}
\begin{split}
\overline{\calE}^K_0:=\sum_{k\leq K}\big(\|\partial_{t,x}D_{\barV}^k\barV(0)\|_{L^2(\Omega_0)}^2+\|D_{\barV}^{k+1}\barV(0)\|_{L^2(\partial\Omega_0)}^2+\|\partial_{t,x}D_{\barV}^{k+1}\barsigma^2(0)\|_{L^2(\Omega_0)}^2\big).
\end{split}
\end{align*}
\begin{theorem}\label{thm:limit}
Let $K$ be sufficiently large, and $v$ and $h$ be defined by~\eqref{eq:Newvdef} and~\eqref{eq:Newhdef}. Suppose the initial data for \eqref{eq:Voriginalintro2} are chosen so that the components $\barV^{i}(0, \cdot)$, $i=1,2,3$, are independent of $c$ and  the component $|\barV^{0}(0,\cdot)-c|$ and the energy $\overline{\calE^K_0}$ are bounded with bounds independent of $c$. Then the solution from Theorem~\ref{thm:main} can be extended to time $T_0=c T_1$, with $T_1>0$ independent of $c$. Moreover as $c\to\infty$, $v(t',\Psi_c(t',\cdot))$, $v_{t'}(t',\Psi_c(t',\cdot))$, and $h(t',\Psi_c(t',\cdot))$ converge strongly in $H^5(\Omega_0)$ to $(v'(t',\Psi(t',\cdot))$, $v'_{t'}(t',\Psi(t',\cdot))$, and $h'(t',\Psi(t',\cdot)))$, respectively, with $(v(t',\cdot),v_{t'}(t',\cdot),h(t',\cdot))$ a solution of \eqref{eq:Newtonian-problem}.
\end{theorem}
\begin{remark}
Theorem~\ref{thm:limit} gives a different proof of existence for \eqref{eq:Newtonian-problem}, originally considered in \cite{Ch-Lin1,LinAnn}.
\end{remark}
\begin{remark}
The choice of $H^5$ for the convergence norm is arbitrary and the convergence can be made as strong as we wish by taking $K$ large.
\end{remark}

The method in this paper works for the more general free boundary relativistic barotropic fluid model \eqref{isentropic-barotropic1}-\eqref{isentropic-barotropic2}-\eqref{boundary1}-\eqref{boundary2}. 
In Appendix~\ref{sec:app general} we will give a brief outline to show how to prove the well-posedness of the Cauchy problem for \eqref{isentropic-barotropic1}-\eqref{isentropic-barotropic2}-\eqref{boundary1}-\eqref{boundary2}. 
We choose to work on the hard phase model \eqref{eq:Voriginalintro2} for the sake of simplicity, as it already captures the main challenges in the more general problem. 

The work in this paper and in \cite{Wu97, Wu99} suggest that the general approach here should work for a variety of free boundary problems. Consider for instance the  Newtonian \emph{compressible} fluid, $\tilV$ would again satisfy a wave equation $\Box_{\tilde{h}}\tilV=0$ in the interior, where $\tilde{h}$ is a conformal metric of the \emph{acoustical metric} (cf. \cite{Ch-M,Luk-Speck})
\begin{align*}
	h:=\eta^{2}(dt)^{2}+\sum_{i=1}^{3}(dx^{i}-\tilV^{i}dt)^{2}.
\end{align*}
Here $\eta$ is the sound speed. While the wave equation in this case is quasilinear, we again see a formal similarity with \eqref{eq:V1}. 

\subsection{Earlier works}
With the exception of \cite{Rendall92} which shows existence of a class of solutions to certain relativistic gaseous models based on earlier work \cite{Makino3}, other advances for well-posedness of relativistic free boundary problems are quite recent. A priori estimates were obtained in \cite{HSS1,JLM1} for some gaseous models. In \cite{Trakhinin} an existence result  was obtained  for a gaseous model using Nash-Moser iteration, and in \cite{Oliynyk2} the existence of solutions was proved for a liquid model in two spacetime dimensions. In \cite{Oliynyk1,Oliynyk3,Oliynyk4},  using different methods, Oliynyk derived a priori estimates and an existence result for a similar liquid model. The same barotropic fluid free boundary problem as in this article was considered by Ginsberg in \cite{Ginsberg1}, who proved an a priori estimate under additional smallness assumptions on the initial data.

\subsection{Main ideas for a priori estimates and local well-posedness}\label{subsec:intro-apriori}

\subsubsection{A priori estimates}\label{subsubsec:apriori}
The a priori estimates for the proof of local existence in Theorem~\ref{thm:main} and the uniform in $c$ time of existence in Theorem~\ref{thm:limit} can be combined by working with the renormalized variables $\barV$ and $\barsigma^2$ defined in \eqref{eq:barVdefintro}. Introducing $\bara:=c^{-2}a$ it can be seen that $\barV$ and $\barsigma^2$ satisfy the same system as \eqref{eq:V1}--\eqref{eq:sigma1a-intro} with $(V,\sigma^2,a)$ replaced by $(\barV,\barsigma^2,\bara)$; see \eqref{eq:mainbar}--\eqref{eq:sigma1abar}. We begin by discussing the energy identity in Lemma~\ref{lem:energy}. For general quantities $\Theta$ satisfying equation \eqref{eq:theta},  with $a$ and $D_V$ replaced by $\bara$ and $D_\barV$, we have, by  Lemma~\ref{lem:energy},
\begin{align}\label{eq:energyid-intro}
	\begin{split}
	&\int_{\Omega_T}\left(c^{-1}D_\barV\Theta \nabla_{0}\Theta+\frac{\barV^0}{2c}\nabla^\mu\Theta\nabla_\mu\Theta\right)\ud x+\int_{\partial\Omega_T}\frac{\barV^0}{c\bara}\left(D_\barV\Theta\right)^2\ud S\\
	=&\int_{\Omega_0}\left(c^{-1}D_\barV\Theta \nabla_{0}\Theta+\frac{\barV^0}{2c}\nabla^\mu\Theta\nabla_\mu\Theta\right)\ud x+\int_{\partial\Omega_0}\frac{\barV^0}{c\bara}\left(D_\barV\Theta\right)^2\ud S\\
	&-c^{-1}\int_0^T\int_{\Omega_t}g D_\barV\Theta \ud x \ud t+2c^{-1}\int_{0}^T\int_{\partial\Omega_t}\frac{1}{a}f  D_\barV\Theta \ud S\ud t+c^{-1}\int_0^T\int_{\partial\Omega_t}\frac{\sdiv \barV}{\bara}\left(D_{\barV}\Theta\right)^{2}\ud S\ud t\\
	&+c^{-1}\int_{0}^T\int_{\Omega_t}(\nabla^\mu \barV^\nu)\nabla_\mu\Theta \nabla_\nu\Theta \ud x \ud t-c^{-1}\int_0^T\int_{\partial\Omega_t}\frac{1}{\bara^2}(D_\barV\bara)(D_\barV\Theta)^2 \ud S\ud t,
	\end{split}
	\end{align}
here $\sdiv \barV$ denotes the (spacetime) divergence of $\barV$ as a vectorfield on $\partial\Omega$ and should be thought of as a lower order term. Since $V$ is future-directed timelike,\footnote{Assume that the solution exists. By the assumption \eqref{eq:data} on initial data and continuity, $V$ will remain future-directed and timelike for a short period of time $[0, T]$ and $\barV^0\simeq c$.} the first term on the left satisfies

\begin{align*}
\begin{split}
(c^{-1}D_\barV\Theta)(\nabla_0\Theta)+\frac{\barV^0}{2c}(\nabla_\nu \Theta)(\nabla^\nu \Theta)\gtrsim |\nabla_{t,x}\Theta|^2.
\end{split}
\end{align*}
Therefore  
as long as $a$ is positive (see \eqref{rel-taylor}), the left hand side of the energy identity \eqref{eq:energyid-intro} controls 
\begin{align*}
\begin{split}
\int_{\partial\Omega_T}|D_\barV\Theta|^2\ud S+\int_{\Omega_T}|\nabla_{t,x}\Theta|^2\ud x.
\end{split}
\end{align*}
The factor $c^{-1}$ in front of the space-time integrals, which appears naturally in this renormalized formulation, is what allows us to prove the uniform in $c$ time of existence. As mentioned earlier, we will construct our energy by working on $D_\barV^k \barV$. Applying \eqref{eq:energyid-intro} to  $\Theta=D_\barV^k \barV^\nu$ and summing over $\nu$  
we get control for $\int_{\partial\Omega_T}|D_\barV^{k+1}\barV|^2\ud S+\int_{\Omega_T}|\nabla_{t,x}D_\barV^k\barV|^2\ud x$ by the right hand side of \eqref{eq:energyid-intro} with $\Theta=D_\barV^k \barV^\nu$. 
This motivates the definition of our $k$-th order energy: 
\begin{align*}
\begin{split}
\calE_k(T):=\int_{\partial\Omega_T}|D_\barV^{k+1}\barV|^2\ud S+\int_{\Omega_T}|\nabla_{t,x}D_\barV^k\barV|^2\ud x.
\end{split}
\end{align*}
The equations satisfied by $D_\barV^k\barV$ are derived in Lemmas~\ref{lem:Vho} and~\ref{lem:boxDVkV}, based on the commutator identities \eqref{eq:comtemp1}--\eqref{eq:comtemp4}. 
Observe that $D_{\barV}$ is defined globally both in the interior of the fluid domain and on the free boundary, being tangential there.  As demonstrated in \eqref{eq:comtemp1}--\eqref{eq:comtemp4},  commuting $D_\barV$ derivatives preserves all important structures of our equations. 

Next we discuss how to estimate the right-hand sides of \eqref{eq:energyid-intro} for $\Theta=\barV$. 
The main contribution is from the inhomogeneous term in the boundary equation in \eqref{eq:V1}, for which we need to control
\begin{align*}
\begin{split}
c^{-1}\int_0^T\int_{\partial\Omega_t}|\nabla D_\barV\barsigma^2|\ud S \ud t.
\end{split}
\end{align*}
For this we use the fact that $D_\barV\barsigma^2$ satisfies the analogue wave equation \eqref{eq:sigma1a-intro} with zero boundary data and with $V$ and $\sigma^2$ replaced by $\barV$ and $\barsigma^2$; see \eqref{eq:sigma1abar}. In Lemma~\ref{lem:benergy} we show, by an appropriate choice of  multiplier field $Q$ for the wave equation, that 
\begin{align}\label{eq:Venergy-intro-temp3}
\begin{split}
&\int_{\Omega_T}\left|\nabla_{t,x}D_\barV\barsigma^2\right|^2\ud x+c^{-1}\int_0^T\int_{\partial\Omega_t}\left|\nabla_{t,x} D_\barV\barsigma^2\right|^2\ud S \ud t\\
\lesssim &\int_{\Omega_0}|\nabla_{t,x}D_\barV\barsigma^2|^2\ud x +\left|c^{-1}\int_0^T\int_{\Omega_t}\left(\Box D_\barV\barsigma^2\right)\left(QD_\barV\barsigma^2\right)\ud x \ud t\right|\\
&+c^{-1}\left|\int_{0}^{T}\int_{\Omega_{t}}\left(\frac12(\nabla_{\mu}Q^{\mu})(\nabla_{\nu}D_\barV\barsigma^2)(\nabla^{\nu}D_\barV\barsigma^2)-(\nabla^{\mu}Q^{\nu})(\nabla_{\mu}D_\barV\barsigma^2)(\nabla_{\nu}D_\barV\barsigma^2)\right)\,dx\,dt\right|.
\end{split}
\end{align}
We will also need the analogue of this estimate with $D_\barV\barsigma^2$ replaced by $D_\barV^k\barsigma^2$ (also contained in Lemma~\ref{lem:benergy}), and the wave equation satisfied by $D_\barV^k\barsigma^2$ is derived in Lemma~\ref{lem:boxDVk1sigma}. 
For $\Box D_\barV\barsigma^2$ on the right-hand side of \eqref{eq:Venergy-intro-temp3}, we note that the term with two derivatives of $\barsigma^2$ on the right-hand side of \eqref{eq:sigma1a-intro} can be seen to be lower order by converting the wave equation for $\barsigma^2$ in \eqref{eq:sigma1a-intro} into an elliptic equation with $D_\barV^2\barsigma^2$ as the source term, as done in Lemma~\ref{lem:A1}, and using elliptic regularity. 

For the right-hand side of \eqref{eq:energyid-intro} with $\Theta=D_\barV^k\barV$, and the higher order analogue of  
\eqref{eq:Venergy-intro-temp3}, we need to estimate spacetime integrals involving the right-hand sides of the equations satisfied by $D_\barV^k\barV$ and $D_\barV^{k+1}\barsigma^2$ as derived in Lemmas~\ref{lem:Vho},~\ref{lem:boxDVkV}, and~\ref{lem:boxDVk1sigma}. The idea for treating the main source terms is similar to what was outlined above, and the treatment of the commutator errors is carried out in Subsection~\ref{subsec:aprioriproof}. Here we only mention that the most delicate commutator error is $c^{-1}\int_0^T\int_{\partial\Omega_t}|\nabla D_\barV^{k-1}\barV|^2\ud S\ud t$, and estimating this term using the energy $\calE_k$ defined above involves one more multiplier identity for wave equations on bounded domains, which is derived in Lemma~\ref{lem:oblique}.

 Finally, we explain how our energies give control of Sobolev norms. This will be needed, for instance, to bound lower order terms in $L^\infty$ in our estimates. The details of deriving Sobolev estimates from our energies are contained in Subsection~\ref{sec:Sobolev}, so here we mention the main idea which is quite simple: Boundedness of $\calE_{k+2}$ gives us control of $\|D_\barV^{k+2}\barV\|_{H^{\frac{1}{2}}(\partial\Omega_t)}$ on the boundary (by the trace theorem), and $\|D_\barV^2 D_\barV^k\barV\|_{L^2(\Omega_t)}$ and $\|\nabla D_\barV^k\barV\|_{L^2(\Omega_t)}$ in the interior (in fact we get more in the interior). Then using the higher order versions of \eqref{eq:V1}, as derived in Lemmas~\ref{lem:Vho} and~\ref{lem:boxDVkV}, this gives us control of $\nabla_n D_\barV^k\barV$ in $H^{\frac{1}{2}}(\partial\Omega_t)$ and an elliptic operator applied to $D_\barV^k\barV$ in $L^2(\Omega_t)$ (see Lemma~\ref{lem:A1}). Elliptic regularity with Neumann boundary conditions (see Lemma~\ref{lem: elliptic estimate}) then allow us to deduce an $H^2(\Omega_t)$ bound for $D_\barV^k\barV$. Similar ideas allow us to get control of $\barV$ in $H^{a}(\Omega_t)$ in terms of $\calE_k$ as long as $2a\leq k$, and similarly for $D_\barV\barsigma^2$. See Proposition~\ref{prop:L2Sobolev}.

\subsubsection{The iteration and Newtonian limit}\label{subsubsec:iteration}
The a priori estimates outlined in the previous subsection can be carried out completely in Eulerian coordinates (that is, over the fluid domain $\Omega$), and contain the main ideas for proving well-posedness. In practice, however, it is more convenient to set up the iteration for the proof of well-posedness in Lagrangian coordinates. The main reason is that in this way the domain becomes fixed, and all norms and function spaces are defined with respect to this fixed domain.
To recast the equations in Lagrangian coordinates we define the (renormalized) Lagrangian map $X:[0,T]\times B\to\bbR^{1+3}$ by the requirements that $X(t,\cdot)$ map $B$ to $\Omega_t\subseteq \bbR^{1+3}$ and
\begin{align*}
\begin{split}
\frac{\ud X}{\ud t}(t,y)=\frac{V}{V^0}(t,X(t,y)).
\end{split}
\end{align*}
In other words, $t\mapsto (t,X(t,\cdot))$ is the flow of the vectorfield $\frac{V}{V^0}$. Note that $\partial_t$ in these coordinates is just the renormalized material derivatives $\frac{1}{V^0}D_V$. The pullback Minkowski metric on $I\times B$ is
\begin{align*}
\begin{split}
g=&-\left(1-\sum_{i=1}^3\frac{(V^{i})^{2}}{(V^{0})^{2}}\circ X\right)\ud t^{2}+2\sum_{i,\ell=1}^3\frac{V^{i}}{V^{0}}\circ X\frac{\partial X^{i}}{\partial y^{\ell}}\ud t\ud y^{\ell}+\sum_{i,k,\ell=1}^3\frac{\partial X^{i}}{\partial y^{k}}\frac{\partial X^{i}}{\partial y^{\ell}}\ud y^{k}\ud y^{\ell}.
\end{split}
\end{align*}
We denote the Lagrangian velocity $V$ by $\Theta$, the enthalpy $\sigma^2$  by $\Sigma$, and its material derivative $D_V\sigma^2$ by $\Lambda$:
\begin{align*}
\begin{split}
\Theta\:= V\circ X \qquad \mathrm{and}\qquad \Sigma:=\sigma^2\circ X =\sqrt{ -g_{\alpha\beta}V^\alpha V^\beta},\qquad \Lambda:=(D_V\sigma^2)\circ X.
\end{split}
\end{align*}
The wave operator then becomes ($g^{\alpha\beta}$ denote the components of the inverse metric $g^{-1}$ and $|g|:=-\det g$)
\begin{align*}
\begin{split}
\Box_g f:= \frac{1}{\sqrt{|g|}}\partial_\alpha (\sqrt{|g|}g^{\alpha\beta}\partial_\beta f).
\end{split}
\end{align*}
Let (by a slight abuse of notation we continue to denote the normal in Lagrangian coordinate by $n$)
\begin{align*}
\begin{split}
\gamma:=\frac{\sqrt{\nabla_\mu\sigma^2 \nabla^\mu\sigma^2}}{2V^0}\circ X,\qquad \mathrm{and}\qquad n^\mu\equiv\frac{\nabla^\mu \sigma^2}{\sqrt{\nabla_\mu\sigma^2 \nabla^\mu\sigma^2}}\circ X= \frac{g^{\mu\nu}\partial_\nu\sigma^2}{\sqrt{g^{\alpha\beta}\partial_\alpha\sigma^2\partial_\beta \sigma^2}}.
\end{split}
\end{align*}
Equation \eqref{eq:V1} then becomes
\begin{align}
&\begin{cases}\label{eq:V2}
(\partial_t^2+\gamma \nabla_{n})\Theta^\nu=-\frac{1}{2(\Theta^0)^2}g^{\alpha\beta}(\partial_\beta X^\nu)\partial_\alpha \Lambda+\frac{1}{\Theta^0}\partial_t \Theta^0\partial_t\Theta^\nu,\qquad&\mathrm{on~}[0,T]\times \partial B\\
\Box_g\Theta^\nu=0,\qquad&\mathrm{in~}[0,T]\times B
\end{cases},\\
&\begin{cases}\label{eq:LambdaLag}
\Box_{g}\Lambda=S(\Theta,\Sigma),\qquad&\mathrm{in~}I\times B\\
\Lambda\equiv 0,\qquad&\mathrm{on~}I\times \partial B
\end{cases},\\
&\partial_t\Sigma=\frac{1}{\Theta_0}\Lambda, \label{eq:SigmaLag}
\end{align}
where
\begin{align}\label{eq:Sdefintro}
\begin{split}
S(\Theta,\Sigma)&:=4g^{\alpha\beta}(\partial_\beta\Theta^\nu)\partial_\alpha\big(m_{\mu\nu}g^{\gamma\delta}(\partial_\delta X^\mu)(\partial_\gamma\Sigma)\big)
+4m_{\rho\mu}m_{\nu\kappa}g^{\alpha\beta}g^{\gamma\delta}(\partial_\delta X^\kappa)(\partial_\alpha\Theta^\nu)(\partial_\beta\Theta^\mu)(\partial_\gamma\Theta^\rho),
\end{split}
\end{align}
The idea for the iteration is to iteratively define $\Thetam$, $\Lambdam$, and $\Sigmam$ as solutions of
\begin{align}
&\begin{cases}\label{eq:Thetam-intro}
(\partial_t^2+\gamma^{(m)} \nabla_{\bfn^{(m)}})(\Theta^{(m+1)})^\nu=-\frac{g_{(m)}^{\alpha\beta}(\partial_{\beta}(X^{(m)})^{\nu})\partial_{\alpha}\Lambdam}{2((\Theta^{(m)})^0)^2}+\frac{\partial_t (\Theta^{(m)})^0\partial_t(\Theta^{(m)})^\nu}{(\Theta^{(m)})^0},\qquad&\mathrm{on~}I\times \partial B\\
\Box_{g^{(m)}}(\Theta^{(m+1)})^\nu=0,\qquad&\mathrm{in~}I\times B
\end{cases},\\
&\begin{cases}\label{eq:Lambdam-intro}
\Box_{g^{(m)}}\Lambda^{(m+1)}=S(\Thetam,\Sigmam),\qquad&\mathrm{in~}I\times B\\
\Lambda^{(m+1)}\equiv 0,\qquad&\mathrm{on~}I\times \partial B
\end{cases},\\
&\partial_t\Sigmamp=\frac{1}{(\Thetamp)^0}\Lambdamp,\label{eq:Sigmam-intro}
\end{align}
where $S(\Sigma^{(k)},\Theta^{(k)})$ is defined from \eqref{eq:Sdefintro} by replacing $\Theta$, $g$, $\Lambda$, and $\Sigma$ by  their iterates  $\Theta^{(k)}$, $g^{(k)}$, and $\Sigma^{(k)}$, respectively. 
See Section~\ref{sec:iteration} for the precise definition of $g^{(k)}$. Once we can show the existence of solutions to each of these linear systems satisfying appropriate energy identities, the ideas from the previous subsection nicely carry over to prove the convergence of $\Theta^{(m)}$ to a solution of \eqref{eq:V2}.  For \eqref{eq:Lambdam-intro} the existence theory is standard, as this is a wave equation with variable coefficients and constant Dirichlet conditions. The only non-standard part is proving existence and energy estimates for the linear system \eqref{eq:Thetam-intro}. This can be achieved by formulating an appropriate weak version of the equation (the main challenge is treating the hyperbolic Dirichlet-Neumann map) and using Galerkin approximations. The weak equations are derived in Subsection~\ref{subsec:weakder}. 
The advantage of this weak formulation for proving existence is that it does not involve the hyperbolic Dirichlet-Neumann map, while it still allows us to derive the main energy identity as for \eqref{eq:energyid1} in Lemma~\ref{lem:energy} 
(see Proposition~\ref{prop:weak1}). Existence, higher regularity, and energy estimates for the weak solution of the linearized problem are proved using Galerkin approximations in Subsection~\ref{sec:weakex} (see \cite{Evans-book,Wloka-book,Lions-Magenes-book1} for some explanations of the Galerkin method.). Once energy estimates, which are modeled on our a priori estimates, are proved at the linear level, a standard iteration scheme produces our desired solution. This is carried out in Section~\ref{sec:iteration}.

Finally, for the proof of Theorem~\ref{thm:limit}, the uniform in $c$ time of existence follows from the local existence result in Theorem~\ref{thm:main} and the uniform in $c$ a priori estimates in Proposition~\ref{prop:apriori}. The existence of a limit also follows from standard compactness arguments and boundedness of higher order energies. The most delicate remaining point is that even though $\barV^0$ and $\partial_0$ are treated at the same order as $\barV^i$ and $\partial_i$ in  the energy estimates in Proposition~\ref{prop:apriori}, after the solution is obtained, one can use the equations to show that these quantities have further decay in $c$. This allows us to conclude that the limiting equations coincide with the classical equations in the Newtonian setting. The details of the argument are presented in Section~\ref{sec:Newtonian}.

\subsection{Notation and conventions}
On $\bbR^{1+3}$ containing the fluid we use $x=(x^0,\dots,x^3)$ as coordinates. On the Lagrangian side we use $y=(y^0,\dots,y^3)$. We also use $t$ for $y^0$ and $\bary$ for $(y^1,\dots,y^3)$. An arbitrary spatial derivative of order $a$ is denoted by $\partial_\bary^a$. Indices are raised and lowered with respect to the Minkowski metric $m$ and Greek indices run over $\{0,\dots,3\}$ while Roman indices run over $\{1,2,3\}$. On the Lagrangian side we use $0$ and $t$ interchangeably for the index of the time coordiante, and sometimes use $r$ for the radial coordinate $r^2=\sum_{i=1}^3(y^i)^2$. For a tensor $T^{\mu\nu}$ in rectangular coordinates we use the notation $T^{\mu r}:=\sum_{b=1}^3\frac{y^b}{r}T^{\mu b}$.

For the derivative of a function $\Theta$ along a vectorifeld $X$, such as the normal $n$, we use the notations $\nabla_n\Theta$ and $n\Theta$ interchangeably. The derivative along the fluid flow line will be denoted by $D_V:=V^\mu \partial_\mu$.

The dependency of constants on other parameters or unknowns is denoted by subscripts, so for instant $C_\delta$ denotes a constant depending on $C_\delta$. The exact value of constants may differ from inequality to inequality as should be clear from the context. 

On the Lagrangian side $B$ denotes the unit ball of radius one, which we use to parameterize the constant $x^0$ slices of the fluid, and $\partial B$ denotes the boundary of $B$. The $L^2$ pairing on $B$ is denoted by $\angles{\cdot}{\cdot}$ and the $L^2$ pairing on $\partial B$ by $\bangles{\cdot}{\cdot}$. The duality pairing between $H^1(B)$ and $(H^1(B))^\ast$ is denoted by $(\cdot,\cdot)$.

\section{A Priori Estimates}\label{sec:apriori}
In this section we assume that $V$, $\sigma^2$, and $D_V\sigma^2$ already exist and satisfy
\begin{align}\label{eq:main}
\begin{cases}
(D_V^2+\frac{1}{2}a\nabla_n)V=-\frac{1}{2}\nabla D_V\sigma^2,\\
\Box V=0.
\end{cases}
\end{align}
and
\begin{align}\label{eq:sigma1a}
\begin{cases}
\Box\sigma^2=(-2\nabla^\mu V^\nu) (\nabla_\mu V_\nu).\\
\Box D_V\sigma^2=4(\nabla^\mu V^\nu)\nabla_\mu \nabla_\nu \sigma^2+4(\nabla^\lambda V^\nu)(\nabla_\lambda V^\mu)(\nabla_\nu V_\mu).
\end{cases}
\end{align}
Here $n$ denotes the exterior unit normal to the timelike boundary of $\Omega$, which we denote by $\partial\Omega$, and $a$ is as in \eqref{eq:aintro}. With $\barsigma^2:=c^{-2}\sigma^2-c^2$ (see \eqref{eq:Newhdef}), $\bara=c^{-2}a$, and $\barV=c^{-1}V$ these equations can be written as
\begin{align}\label{eq:mainbar}
\begin{cases}
(D_\barV^2+\frac{1}{2}\bara\nabla_n)\barV=-\frac{1}{2}\nabla D_\barV\barsigma^2,\\
\Box \barV=0.
\end{cases}
\end{align}
and
\begin{align}\label{eq:sigma1abar}
\begin{cases}
\Box\barsigma^2=(-2\nabla^\mu \barV^\nu) (\nabla_\mu \barV_\nu).\\
\Box D_\barV\barsigma^2=4(\nabla^\mu \barV^\nu)\nabla_\mu \nabla_\nu \barsigma^2+4(\nabla^\lambda \barV^\nu)(\nabla_\lambda \barV^\mu)(\nabla_\nu \barV_\mu).
\end{cases}
\end{align}
Motivated by the discussion in the introduction, for any function $\Theta$ we define the energies
\begin{align*}
\begin{split}
&E[\Theta,t]:=\int_{\Omega_t}|\partial_{t,x}\Theta|^2\ud x+\int_{\partial\Omega_t}|D_\barV\Theta|^2\ud S,\\
&\Ebar[\Theta,T]:=\sup_{0\leq t\leq T}\int_{\Omega_t}|\partial_{t,x}\Theta|^2\ud x+c^{-1}\int_0^T\int_{\partial\Omega_t}|\partial_{t,x}\Theta|^2\ud S\ud t,
\end{split}
\end{align*}
where for $T=0$
\begin{align*}
\begin{split}
\Ebar[\Theta,0]:=\int_{\Omega_0}|\partial_{t,x}\Theta|^2\ud x.
\end{split}
\end{align*}
Higher order energies are defined as
\begin{align*}
\begin{split}
E_j[\Theta,t]=E[D_\barV^j\Theta,t],\quad E_{\leq k}[\Theta,t]= \sum_{j=0}^kE_j[\Theta,t],\quad\Ebar_j[\Theta,T]=\Ebar[D_\barV^j\Theta,T],\quad \Ebar_{\leq k}[\Theta,T]= \sum_{j=0}^k\Ebar_j[\Theta,T].
\end{split}
\end{align*}
To simplify notation we introduce the unified energy
\begin{align}\label{energy}
\begin{split}
\calE_{k}(T):=\Ebar_{\leq k+1}[\barsigma^2,T]+ \sup_{0\leq t\leq T}E_{\leq k}[\barV,t].
\end{split}
\end{align}
Our goal in this section is to prove the following a priori estimate.
\begin{proposition}\label{prop:apriori}
	Suppose $\barV$ is a solution to \eqref{eq:mainbar} with 
	\begin{align}\label{eq:aprioribootstrap}
	\begin{split}
	\calE_\ell(T)\leq C_1,
	\end{split}
	\end{align}
	for some constant $C_1>0$ and $\ell$ sufficiently large and let
\begin{align*}
\begin{split}
\scE_\ell(T):=\sup_{t\in[0,T]}\sum_{k\leq \ell+2-2j}(\|D_\barV^k\barV\|_{H^j(\Omega_t)}^2+\|D_\barV^{k+1}\barsigma^2\|_{H^j(\Omega_t)}^2)+\calE_\ell(T).
\end{split}
\end{align*}
	If $T=c\barT$ with $\barT>0$  sufficiently small, depending on $C_1$, $\scE_\ell(0)$, $c_0$, and $\ell$, then
	\begin{align}\label{eq:apriori}
	\begin{split}
	\scE_\ell(T)\leq P_\ell(\scE_\ell(0))
	\end{split}
	\end{align}  
	for some polynomial function $P_\ell$ (independent of $C_1$).
\end{proposition}
\begin{remark}\label{rmk: a priori}
Proposition~\ref{prop:apriori} is the backbone of the proof of both Theorems~\ref{thm:main} and~\ref{thm:limit}. Indeed, for local existence as in Theorem~\ref{thm:main}, we set up an iteration based on these a priori estimates for a fixed value of $c$ (which we can taken to be $c=1$, as we suggest the reader do on first reading). For the Newtonian limit as in Theorem~\ref{thm:limit}, we use Proposition~\ref{prop:apriori} to extend the solution to time $x^0=c\barT$ and use boundedness of higher energies to extract a limit.
\end{remark}
\subsection{General identities and estimates}\label{general estimates}
In this section we record a number of general identities and estimates which will be used in the proof of Proposition~\ref{prop:apriori}. We start by recording a general multiplier identity for the wave equation. Let $Q=Q^\mu\nabla_\mu$ be an arbitrary first order multiplier. Then a direct calculation shows that
\begin{equation}\label{eq:multid}
(\Box \Theta )(Q\Theta)= \nabla_\mu((Q\Theta)( \nabla^\mu \Theta)-\frac{1}{2}Q^\mu (\nabla_\nu\Theta)( \nabla^\nu\Theta))+\frac{1}{2}(\nabla_\mu Q^\mu)(\nabla_\nu\Theta)(\nabla^\nu\Theta)-(\nabla^\mu Q^\nu)( \nabla_\mu\Theta)(\nabla_\nu\Theta).
\end{equation}

Our first lemma  
is the main energy identity for \eqref{eq:mainbar}.
\begin{lemma}\label{lem:energy}
	Suppose $\Theta$ satisfies
	\begin{align}\label{eq:Theta}
	\begin{cases}
	(D_\barV^2+\frac{1}{2}\bara n)\Theta=f\\
	\Box \Theta=g
	\end{cases}.
	\end{align}
	Then
	\begin{align}\label{eq:energyid1}
	\begin{split}
	&\int_{\Omega_T}(c^{-1}D_\barV\Theta \nabla_{0}\Theta+\frac{\barV^0}{2c}\nabla^\mu\Theta\nabla_\mu\Theta)\ud x+\int_{\partial\Omega_T}\frac{\barV^0}{c\bara}(D_\barV\Theta)^2\ud S\\
	=&\int_{\Omega_0}(c^{-1}D_\barV\Theta \nabla_{0}\Theta+\frac{\barV^0}{2c}\nabla^\mu\Theta\nabla_\mu\Theta)\ud x+\int_{\partial\Omega_0}\frac{\barV^0}{c\bara}(D_\barV\Theta)^2\ud S\\
	&-c^{-1}\int_0^T\int_{\Omega_t}g D_\barV\Theta \ud x \ud t+2c^{-1}\int_{0}^T\int_{\partial\Omega_t}\frac{1}{\bara}f D_V\Theta \ud S\ud t+c^{-1}\int_0^T\int_{\partial\Omega_t}\frac{\sdiv \barV}{\bara}\left(D_{\barV}\Theta\right)^{2}\ud S\ud t\\
	&+c^{-1}\int_{0}^T\int_{\Omega_t}(\nabla^\mu \barV^\nu)\nabla_\mu\Theta \nabla_\nu\Theta \ud x \ud t-c^{-1}\int_0^T\int_{\partial\Omega_t}\frac{1}{\bara^2}(D_\barV\bara)(D_\barV\Theta)^2 \ud S\ud t,
	\end{split}
	\end{align}
where $\sdiv$ denotes the divergence operator on $\partial\Omega$.\footnote{For a simpler model, suppose $u$ satisfies 
\begin{align*}
\begin{cases}
\Box u =0,\qquad &\mathrm{in~}[0,T]\times B\\
(\partial_t^2 +\partial_r) u = f\qquad &\mathrm{on~}[0,T]\times \partial B
\end{cases},
\end{align*}
where $B$ is the unit ball in $\bbR^3$ with normal $\partial_r$. Then a similar argument using the multiplier $\partial_tu$ gives
\begin{align*}
\begin{split}
\frac{1}{2}\int_{B}|\partial_{t,x}u(T)|^2\ud x+\frac{1}{2}\int_{\partial B}(\partial_tu(T))^2\ud S = \frac{1}{2}\int_{B}|\partial_{t,x}u(0)|^2\ud x+\frac{1}{2}\int_{\partial B}(\partial_tu(0))^2\ud S +\int_0^T\int_{\partial B}(\partial_tu) f \ud S \ud t.
\end{split}
\end{align*}}
\end{lemma}
\begin{proof} 
	Multiplying the first equation in \eqref{eq:Theta} by $\frac{1}{\bara} D_\barV\Theta$ we get 
	\begin{align*}
	\begin{split}
	\frac{1}{2}D_{\barV}\left(\frac{1}{\bara}(D_\barV\Theta)^2\right)+\frac{1}{2}(n\Theta) (D_\barV\Theta)=\frac{1}{\bara}f D_\barV\Theta-\frac{1}{2\bara^2}(D_\barV \bara)(D_\barV\Theta)^2,
	\end{split}
	\end{align*}
	which upon integration over $\partial\Omega=\cup_{t\in[0,T]}\partial\Omega_t$ gives
	\begin{align}\label{eq:boundary}
	\begin{split}
	&\int_{\partial\Omega_T}\frac{\barV^0}{\bara}(D_\barV\Theta)^2\ud S+\int_0^T\int_{\partial\Omega_t}(n\Theta)(D_\barV\Theta)\ud S\ud t\\
	&=\int_{\partial\Omega_0}\frac{\barV^0}{\bara}(D_\barV\Theta)^2\ud S+\int_0^T\int_{\partial\Omega_t}\frac{2}{\bara}f D_\barV\Theta \ud S\ud t\\
	&-\int_0^T\int_{\partial\Omega_t}\frac{1}{\bara^2}(D_\barV\bara)(D_\barV\Theta)^2\ud S\ud t+\int_0^T\int_{\partial\Omega_t}\frac{\sdiv \barV}{\bara}\left(D_{\barV}\Theta\right)^{2}\ud S\ud t.
	\end{split}
	\end{align}
To treat the second term on the left, we integrate \eqref{eq:multid} with $Q=\barV$ over $\Omega\cap\{0\le t\le T\}$.  Using the fact that $\barV$ is tangent to $\partial\Omega$, we get 
	\begin{align}\label{eq:energytemp1}
	\begin{split}
	&\int_{\Omega_T}(D_\barV\Theta \nabla_{0}\Theta+\frac{\barV^0}{2}\nabla^\nu\Theta\nabla_\nu\Theta)\ud x-\int_0^T\int_{\partial\Omega_t}(D_\barV\Theta)( n\Theta )\ud S\ud t\\
	&=\int_{\Omega_0}(D_\barV\Theta \nabla_{0}\Theta+\frac{\barV^0}{2}\nabla^\nu\Theta\nabla_\nu\Theta)\ud x-\iint_{\Omega}g D_\barV\Theta \ud x \ud t+\iint_{\Omega}(\nabla^\mu \barV^\nu)(\nabla_\mu\Theta)(\nabla_\nu\Theta)\ud x \ud t.
	\end{split}
	\end{align}
The lemma follows by adding \eqref{eq:energytemp1} to \eqref{eq:boundary} and multiplying by $c^{-1}$.
\end{proof}
We will apply Lemma~\ref{lem:energy} to $\Theta=D_\barV^k \barV^\nu$,  for $0\le k\le \ell$.
The next energy estimate is used for the second equation in \eqref{eq:sigma1a} (see \eqref{eq:Venergy-intro-temp3}), as well as \eqref{eq:boxDVk1sigma}.
\begin{lemma}\label{lem:benergy}
	There is a (future-directed and timelike) vectorfield $Q$ such that for any $\Theta$ which is constant on $\partial\Omega$,
	\begin{align}\label{eq:benergy}
	\begin{split}
	&\int_{\Omega_T}|\partial_{t,x}\Theta|^2\ud x+c^{-1}\int_0^T\int_{\partial\Omega_t}|\partial_{t,x}\Theta|^2\ud S \ud t\\
	&\lesssim \int_{\Omega_0}|\partial_{t,x}\Theta|^2\ud x+\left|c^{-1}\int_0^T\int_{\Omega_t}(\Box \Theta)(Q\Theta)\ud x \ud t\right|\\
	&\quad+\left|c^{-1}\int_0^T\int_{\Omega_t}(\frac{1}{2}(\nabla_\mu Q^\mu)(\nabla_\nu\Theta)(\nabla^\nu\Theta)-(\nabla^\mu Q^\nu)(\nabla_\mu\Theta)(\nabla_\nu\Theta))\ud x \ud t\right|.
	\end{split}
	\end{align}
\end{lemma}
\begin{proof}
	Since $\Theta$ is constant on $\partial\Omega$,
	\begin{align*}
	\begin{split}
	\nabla_\nu \Theta \nabla^\nu \Theta = (n\Theta)^2
	\end{split}
	\end{align*}
	and
	\begin{align*}
	\begin{split}
	n_\mu((Q\Theta)( \nabla^\mu \Theta)-\frac{1}{2}Q^\mu (\nabla_\nu\Theta)( \nabla^\nu\Theta))=\frac{1}{2}Q^n (n\Theta)^2
	\end{split}
	\end{align*}
	on $\partial\Omega$, where $Q^n:=m(Q,n)$. Therefore letting $Q$ be a future-directed timelike vectorfield with $Q^n>0$,  in particular $Q=c^{-1}(\alpha \barV+n)$ for some large $\alpha$, we get the desired estimate upon integrating \eqref{eq:multid} over $\Omega$.
\end{proof}

Our last application of \eqref{eq:multid} will be to control arbitrary derivatives of an arbitrary function on the boundary in terms of the normal and $D_V$ derivatives.

\begin{lemma}\label{lem:oblique}
	There exists a (future-directed and timelike) vectorfield $Q$  such that or any function $\Theta$ 
	\begin{align}\label{eq:oblique}
	\begin{split}
	&\sup_{0\leq t\leq T}\int_{\Omega_t}|\partial_{t,x}\Theta|^2\ud x+c^{-1}\int_0^T\int_{\partial\Omega_t}|\partial_{t,x}\Theta|^2\ud S \ud t\\
	&\lesssim \int_{\Omega_0}|\partial_{t,x}\Theta|^2\ud x+c^{-1}\int_0^T\int_{\partial\Omega_t}((n\Theta)^2+(D_\barV\Theta)^2)\ud S\ud t\\
	&\quad+\left|c^{-1}\int_0^T\int_{\Omega_t}(\Box \Theta)(Q\Theta)\ud x \ud t\right|+\left|c^{-1}\int_0^T\int_{\Omega_t}(\frac{1}{2}(\nabla_\mu Q^\mu)(\nabla_\nu\Theta)(\nabla^\nu\Theta)-(\nabla^\mu Q^\nu)(\nabla_\mu\Theta)(\nabla_\nu\Theta))\ud x \ud t\right|.
	\end{split}
	\end{align}
\end{lemma}
\begin{proof}
	The proof is similar to that of Lemma~\ref{lem:benergy}, but this time we choose $Q^n=m(Q,n)<0$. For instance, let $Q=c^{-1}(\alpha \barV-n)$ with $\alpha>0$ chosen so that $Q$ is future-directed and timelike. Then on $\partial\Omega$,
	\begin{align*}
	\begin{split}
	&n_\mu((Q\Theta)( \nabla^\mu \Theta)-\frac{1}{2}Q^\mu (\nabla_\nu\Theta)( \nabla^\nu\Theta))\\
	=&c^{-1}\left(\left(\alpha D_{\barV}\Theta\right)\left( n\Theta\right)-(n\Theta)^2+\frac{1}{2}\nabla_\nu\Theta \nabla^\nu\Theta\right)\geq c^{-1}\left(c_1 |\partial_{t,x}\Theta|^2- c_2 ((n\Theta)^2+(D_\barV\Theta)^2)\right)
	\end{split}
	\end{align*}
	for some  constants $c_1, c_2>0$ depending only on $V$. The lemma now follows by integrating \eqref{eq:multid} on $\Omega$.
\end{proof}

The next lemma will be used to control $\|\nabla^2\Theta\|_{L^2(\Omega_t)}$ in terms of $\|\Box \Theta\|_{L^2(\Omega_t)}$ and $\|\nabla D_\barV\Theta\|_{L^2(\Omega_t)}$.

\begin{lemma}\label{lem:A1}
	Let $\barA$ be defined as $\barA:=\bara^{ij}\partial^2_{ij}$, where $\bara^{ij}=(m^{-1})^{ij}-\frac{\barV^i}{\barV^0}\frac{\barV^j}{\barV^0}$. Then for any $\Theta$,
	\begin{align*}
	\begin{split}
	\barA\Theta&=\Box\Theta+\frac{1}{\barV^0}\partial_0 D_\barV\Theta-\frac{\barV^i}{(\barV^0)^2}\partial_i D_\barV\Theta+\left(\frac{\barV^i}{\left(\barV^0\right)^2}\partial_i \barV^0\right)\partial_0\Theta+\left(\frac{\barV^i}{\left(\barV^0\right)^2}\partial_i \barV^j\right)\partial_j\Theta\\
	&\quad-\left(\frac{1}{\barV^0}\partial_0 \barV^0\right)\partial_0\Theta-\left(\frac{1}{\barV^0}\partial_0\barV^j\right)\partial_j\Theta.
	\end{split}
	\end{align*}
	Moreover,
	\begin{align*}
	\begin{split}
	&\partial_0^2\Theta = \Delta \Theta-\Box \Theta,\\
	&\partial^2_{0i}\Theta=\frac{1}{\barV^0}\left(\partial_i D_\barV\Theta-\barV^j\partial^2_{ij}\Theta-\left(\partial_i\barV^0\right)\partial_0\Theta-\left(\partial_i\barV^j\right)\partial_j\Theta\right).
	\end{split}
	\end{align*}
\end{lemma}
\begin{proof}
	The proof is a direct calculation using the identities
	\begin{align*}
	&\partial_0 D_\barV\Theta= \barV^0\partial_0^2\Theta+\barV^j\partial^2_{0j}\Theta+(\partial_0\barV^0)\partial_0\Theta+(\partial_0\barV^i)\partial_i\Theta,\\
	&\partial_i D_\barV\Theta= \barV^0\partial^2_{i0}\Theta+\barV^j\partial^2_{ij}\Theta+(\partial_i\barV^0)\partial_0\Theta+(\partial_i\barV^j)\partial_j\Theta.
	\end{align*}
\end{proof}
To use Lemma \ref{lem:A1}, we will apply the following standard elliptic estimates (cf. \cite{Taylor-book1}).

\begin{lemma}\label{lem: elliptic estimate}
	For any $t>0$, we have
	\begin{align*}
	\|\nabla_x^{(2)}\Theta\|_{L^2(\Omega_{t})}\lesssim \|\barA\Theta\|_{L^{2}(\Omega_{t})}+\|\Theta\|_{H^{\frac{3}{2}}(\partial\Omega_{t})},
	\end{align*}
	and 
	\begin{align*}
	\|\nabla_x^{(2)}\Theta\|_{L^{2}(\Omega_{t})}\lesssim \|\barA\Theta\|_{L^{2}(\Omega_{t})}+\|N\Theta\|_{H^{\frac{1}{2}}(\partial\Omega_{t})},
	\end{align*}
	where $N$ is a transversal vectorfield to $\partial\Omega_{t}\subseteq\Omega_t$,	and where the implicit constants depend on $\Omega_{t}$.
\end{lemma}
\subsection{Higher order equations}
Here we derive the higher order versions of \eqref{eq:mainbar} and \eqref{eq:sigma1abar}. The main commutator identities, valid for any $\Theta$, are:
\begin{align}
&[D_\barV,\nabla_\mu]\Theta= -(\nabla_\mu \barV^\nu)\nabla_\nu\Theta.\label{eq:comtemp1}\\
&[D_\barV,\nabla_\nu\nabla_\lambda]\Theta = -(\nabla_\lambda \barV^\kappa)\nabla_\kappa \nabla_\nu\Theta - (\nabla_\nu \barV^\kappa)\nabla_\kappa\nabla_\lambda\Theta-(\nabla_\nu\nabla_\lambda \barV^\kappa)\nabla_\kappa\Theta.\label{eq:comtemp2}\\
&[D_\barV,\Box]\Theta=-2(\nabla^\mu \barV^\nu)\nabla_\mu\nabla_\nu\Theta.\label{eq:comtemp3}\\
&[D_\barV,D_\barV^2-\frac{1}{2}\nabla^{\mu}\barsigma^{2}\nabla_{\mu}]\Theta=\frac{1}{2}(\nabla_\lambda\barsigma^2)(\nabla^\lambda \barV^\nu)(\nabla_\nu\Theta)+\frac{1}{2}(\nabla_{\lambda}\barsigma^{2})(\nabla^{\nu}\barV^{\lambda})(\nabla_{\nu}\Theta)\nonumber\\
&\phantom{[D_V,D_V^2-\frac{1}{2\nabla^{\mu}\barsigma^{2}}\nabla_{\mu}]\Theta=}-\frac{1}{2}(\nabla^\nu(D_\barV\barsigma^2))\nabla_\nu\Theta.\label{eq:comtemp4}
\end{align}
Applying these identities we can calculate the higher order versions of  \eqref{eq:mainbar} and \eqref{eq:sigma1abar}, which we record in the following lemmas.
\begin{lemma}\label{lem:Vho}
	For any $k\geq0$
	\begin{align}\label{eq:Vho}
	\begin{split}
	(D_\barV^2+\frac{1}{2}\bara n)D_\barV^k\barV= -\frac{1}{2}\nabla D_\barV^{k+1}\barsigma^2+ F_k
	\end{split}
	\end{align}
	where $F_k$ is a linear combination of terms of the forms
	\begin{enumerate}
		\item \label{eq:Fk1} $(\nabla D^{k_1}_\barV \barV)\dots(\nabla D_\barV^{k_m} \barV)(\nabla D^{k_{m+1}}_\barV \barsigma^2)$, where $k_1+\dots+k_{m+1}\leq k-1$.
		\item \label{eq:Fk2} $(\nabla D^{k_1}_\barV \barV)\dots(\nabla D_\barV^{k_m} \barV)(\nabla D^{k_{m+1}}_\barV D_\barV \barsigma^2)$, where $k_1+\dots+k_{m+1}\leq k-1$.
	\end{enumerate}
\end{lemma}
\begin{proof}
	We proceed inductively. For $k=0$ the desired identity holds with $F_k=0$. Assume it holds for $k=j$ and let us prove it for $k=j+1$. First,
	\begin{align*}
	\begin{split}
	-D_\barV (\bara n D_\barV^j \barV) = D_\barV (\nabla^\mu\barsigma^2 \nabla_\mu D_\barV^j \barV)&= \bara n D_\barV^{j+1}\barV+(\nabla^\mu D_\barV\barsigma^2)\nabla_\mu D_\barV^j \barV\\
	&\quad-(\nabla^\mu \barV^\nu)((\nabla_\nu\barsigma^2)\nabla_\mu D_\barV^j \barV+(\nabla_\mu\barsigma^2)\nabla_\nu D_\barV^j\barV),
	\end{split}
	\end{align*}
	so $[D_\barV+\frac{1}{2}\bara n,D_\barV]D_\barV^j\barV$ has the right form. Next, in view of \eqref{eq:comtemp1}, $D_\barV$ applied to the terms in \eqref{eq:Fk1} and \eqref{eq:Fk2} with $k$ replaced by $j$, as well as $\nabla D_\barV^{j+1}\barsigma^2$,  also has the desired form.
\end{proof}
The wave equation for $D_\barV^k\barV$ in $\Omega$ is derived in the next lemma.

\begin{lemma}\label{lem:boxDVkV}
	For any $k\geq0$
	\begin{align}\label{eq:boxDVkV}
	\begin{split}
	\Box D_\barV^k \barV= G_k
	\end{split}
	\end{align}
	where $G_k$ is a linear combination of terms of the form
	\begin{equation}\label{eq:Gk}
	(\nabla D_\barV^{k_1}\barV)\dots(\nabla D_\barV^{k_m}\barV)(\nabla^{(2)} D_\barV^{k_{m+1}}\barV),\qquad k_1+\dots+k_{m+1}\leq k-1.
	\end{equation}
\end{lemma}
\begin{proof}
	Again we proceed inductively. For $k=0$, $G_k=0$ so suppose the lemma holds with $k=j$ and let us prove it for $k=j+1$. By \eqref{eq:comtemp3}, $[D_\barV,\Box] D_\barV^j\barV$ has the right form. Similarly, $D_\barV$ applied to \eqref{eq:Gk} has the desired form by \eqref{eq:comtemp1} and\eqref{eq:comtemp2}.
\end{proof}
Next we derive the wave equation satisfied by $D_\barV^{k+1}\barsigma^2$.
\begin{lemma}\label{lem:boxDVk1sigma}
	For any $k\geq0$
	\begin{align}\label{eq:boxDVk1sigma}
	\begin{split}
	\Box D_\barV^{k+1}\barsigma^{2}=H_k
	\end{split}
	\end{align}
	where $H_k$ is a linear combination of terms of the forms
	\begin{enumerate}
		\item \label{eq:Hk1}  $(\nabla D_\barV^{k_1} \barV)\dots(\nabla D_\barV^{k_m}\barV)\nabla^{(2)}D_\barV^{k_{m+1}}\barsigma^2$, where $k_1+\dots+k_{m+1}\leq k$.
		\item \label{eq:Hk2} $(\nabla D_\barV^{k_1}\barV)\dots (\nabla D_\barV^{k_m}\barV)(\nabla D_\barV^{k_{m+1}}\barsigma^2) \nabla^{(2)}D_\barV^{k_{m+2}}\barV$, where $k_1+\dots+k_{m+2}\leq k$ and $k_{m+2}\leq k-1$.
		\item \label{eq:Hk3} $(\nabla D_\barV^{k_1} \barV) \dots(\nabla D_\barV^{k_m}\barV)$ with $k_1+\dots+k_m=k$.
	\end{enumerate}
\end{lemma}
\begin{proof}
	For $k=0$ the statement already contained in the second equation in \eqref{eq:sigma1a}, so let us assume it holds for $k=j$ and prove it for $k=j+1$. By \eqref{eq:comtemp3} the commutator $[D_\barV,\Box]D_\barV^{j+1}\barsigma^2$ has the right form. By \eqref{eq:comtemp1}, $D_\barV$ applied on terms of the form \eqref{eq:Hk3} also has the right form. Finally $D_\barV$ applied to terms in \eqref{eq:Hk1} and \eqref{eq:Hk2} has the desired form in view of \eqref{eq:comtemp1} and \eqref{eq:comtemp2}.
\end{proof}
\subsection{Sobolev estimates}\label{sec:Sobolev}
To prove Proposition~\ref{prop:apriori} we need to show that higher order energies give pointwise control of lower order derivatives of $V$ and $L^2$ control of lower order Sobolev norms of $V$. The main result of this subsection is the following proposition.
\begin{proposition}\label{prop:L2Sobolev}
	Suppose
	\begin{align}\label{a priori assumption H2 V}
	\sum_{k+2p\leq M+2}\|\partial_{t,x}^p D_\barV^{k}\barV\|_{L^2(\Omega_t)}^2+\sum_{k+2p\leq M+2}\|\partial_{t,x}^{p}D_{\barV}^{k+1}\barsigma^{2}\|_{L^{2}(\Omega_{t})}^{2}\leq C_{M}.
	\end{align}
	If $M>0$ is sufficiently large and $T=c\barT$ with $\barT>0$ sufficiently small, then under the assumptions of Proposition~\ref{prop:apriori}, for any $t\in[0,T]$
	\begin{align}\label{eq:sobolevgoal}
	\begin{split}
	&\sum_{k+2p\leq M+2}\|\partial_{t,x}^p D_\barV^{k}\barV\|_{L^2(\Omega_t)}^2+\sum_{k+2p\leq M+2}\|\partial_{t,x}^{p}D_{\barV}^{k+1}\barsigma^{2}\|_{L^{2}(\Omega_{t})}^{2}\\
	&\lesssim \sup_{0\leq \tau\leq t}\Ebar_{\leq M+1}[\barsigma^2,\tau]+ \sup_{0\leq \tau\leq t}E_{\leq M}[\barV,\tau]+\sum_{k+2p\leq M+2}\|\partial_{t,x}^p D_\barV^{k}\barV\|_{L^2(\Omega_0)}^2+\sum_{k+2p\leq M+2}\|\partial_{t,x}^{p}D_{\barV}^{k+1}\barsigma^{2}\|_{L^{2}(\Omega_{0})}^{2}.
	\end{split}
	\end{align}
	The implicit constant in this estimate is independent of $C_M$ and $c$.
\end{proposition}
Before discussing the proof of Proposition~\ref{prop:L2Sobolev} we state a few immediate corollaries.
\begin{corollary}\label{cor:Sobolev}
	Assuming the bootstrap assumption 
	\begin{align*}
	\begin{split}
	&\sup_{t\in[0,T]}\left[\sum_{1\leq k\leq M-2}\|D_\barV^k\barV\|_{H^2(\Omega_t)}+\sum_{j\leq 2}\sum_{1\leq k\leq M-2(j+1)}\|\nabla^{(j)}D_\barV^k\barV\|_{L^\infty(\Omega_t)}\right]\\
	&\lesssim \calE_M^{\frac{1}{2}}(T)+\sum_{k+2p\leq M+2}\|\nabla^p D_\barV^{k}\barV\|_{L^2(\Omega_0)}^2+\sum_{k+2p\leq M+2}\|\nabla^{p}D_{\barV}^{k+1}\barsigma^{2}\|_{L^{2}(\Omega_{0})}^{2},
	\end{split}
	\end{align*}
	where the implicit constant is independent of $C_M$ in \eqref{a priori assumption H2 V}. 
\end{corollary}
\begin{proof}
This is a direct consequence of Proposition~\ref{prop:L2Sobolev} and the Sobolev embedding $H^{2}(\Omega_t)\hookrightarrow L^\infty(\Omega_t)$.
\end{proof}
In order to use the energy estimates from the previous section we also need to show that $V$ remains timelike  and $a$ stays bounded away from zero. These statements are summarized in the following corollary.
\begin{corollary}\label{cor:aV}
	Suppose the hypothesis of Proposition~\ref{prop:apriori} hold. Then there are constant $a_0,v_0>0$ and $\gamma>1$ such that
	\begin{align*}
	\begin{split}
	\inf_{0\leq t\leq T}\bara>a_0,\qquad\inf_{0\leq t\leq T}\barV^0>v_0c,\qquad\inf_{0\leq t\leq T} \frac{(\barV^0)^2}{\sum_{j=1}^3(\barV^j)^2}>\gamma c^2.
	\end{split}
	\end{align*}
\end{corollary}
\begin{proof}
	The proof is by integrating in time combined with the $L^{\infty}$-bounds in Corollary~\ref{cor:Sobolev}.
\end{proof}
Before we give the proof of the proposition, we need some preparation. First, we introduce some notations:
\begin{align}\label{fixing notations}
\begin{split}
\nabla_{i}:=\partial_{i},\quad \nbar:=\frac{(\partial_{1}\barsigma^{2},\partial_{2}\barsigma^{2},\partial_{3}\barsigma^{2})}{\sqrt{\sum_{i=1}^{3}(\partial_{i}\barsigma^{2})^{2}}},\quad \nbar_i:=\delta_{ij}\nbar^j,\quad \slashed{\nabla}_{i}:=\partial_{i}-\nbar_{i}\nbar^{j}\partial_{j}.
\end{split}
\end{align}
Note that $\slashed{\nabla}_{i}, i=1,2,3$ are defined globally, are tangential to $\partial\Omega_{t}$, and span $T\partial\Omega_t$.

The following lemma is used to estimate $\|D_{\barV}^{k+1}\barsigma^{2}\|_{H^{j}(\Omega_{t})}$, and plays a crucial role in estimating $\|D_{\barV}^{k}\barV\|_{H^{j}(\Omega_{t})}$.

\begin{lemma}\label{lem: reduce to tangential}
	For any smooth function $\Theta$, the following estimate holds:
	\begin{align*}
	\|\Theta\|_{H^{j}(\Omega_{t})}&\lesssim \|\Theta\|_{H^{j-1}(\Omega_t)}+\|\nabla^{(j-2)}\barA\Theta\|_{L^2(\Omega_t)}+\|[\barA,\nabla^{(j-2)}]\Theta\|_{L^2(\Omega_t)}\\
	&\quad+\|[\snabla,\nabla^{(j-2)}]\Theta\|_{H^1(\Omega_t)}+\|\snabla\Theta\|_{H^{j-1}(\Omega_t)}.
	\end{align*}
\end{lemma}
\begin{proof}
	Using the first estimate in Lemma \ref{lem: elliptic estimate} and the trace theorem
	\begin{align*}
	\begin{split}
	\|\Theta\|_{H^j(\Omega_t)}&\lesssim \|\Theta\|_{H^{j-1}(\Omega_t)}+\|\barA\nabla^{(j-2)}\Theta\|_{L^2(\Omega_t)}+\|\snabla\nabla^{(j-2)}\Theta\|_{H^{\frac{1}{2}}(\partial\Omega_t)}+\|\nabla^{(j-2)}\Theta\|_{H^{\frac{1}{2}}(\partial\Omega_t)}\\
	&\lesssim \|\Theta\|_{H^{j-1}(\Omega_t)}+\|\barA\nabla^{(j-2)}\Theta\|_{L^2(\Omega_t)}+\|\snabla\nabla^{(j-2)}\Theta\|_{H^1(\Omega_t)}.
	\end{split}
	\end{align*}
	The desired estimate follows after commuting the operators.
\end{proof}
The next lemma allows us to bound lower order terms in $L^\infty$.
\begin{lemma}\label{lem:Linfty1}
	Under the bootstrap assumption \eqref{a priori assumption H2 V}, if $\barT>0$ is sufficiently small, then
	\begin{align}\label{eq:Linfty1}
	\begin{split}
	\left\|\nabla^{a} D_\barV^k \left(\barV^{0}-c,\barV^{i}\right)\right\|_{L^\infty(\Omega_t)}+\|\nabla^{a} D_\barV^{k+1} \barsigma^{2}\|_{L^\infty(\Omega_t)}\lesssim 1,\quad \forall 0\leq a\leq p-2,\quad k\leq M-2p-3,\quad t\in[0,T],
	\end{split}
	\end{align}
	where the implicit constant is independent of $C_{M}$. 
\end{lemma}
\begin{proof}
	Let $\barxi$ be the Lagrangian parameterization, that is,
	\begin{align*}
	\begin{split}
	\partial_\tau \barxi(\tau,y)= \left(\frac{\barV}{\barV^0}\right)\left(\xi(\tau,y)\right),\qquad \barxi(0,y)=y.
	\end{split}
	\end{align*}
	If $p_t$ is a point on $\Omega_t$, we let $p_0$ be the point on $\Omega_0$ such that $\barxi(t,p_0)=p_t$. 
	For any function $\Theta$
	\begin{align}\label{V inte formula}
	\begin{split}
	\Theta(p_t)-\Theta(p_0)=\int_0^t\left(\frac{D_\barV\Theta}{\barV^0}\right)(p_\tau)\ud \tau.
	\end{split}
	\end{align}
	It follows that, using the standard Sobolev estimate,
	\begin{align*}
	\begin{split}
	\|\Theta\|_{L^\infty(\Omega_t)}\leq \|\Theta\|_{L^\infty(\Omega_0)}+Cc^{-1}\,t\sup_{0\leq s\leq t}\|D_\barV\Theta\|_{L^\infty(\Omega_s)}\lesssim  \|\Theta\|_{H^2(\Omega_0)}+c^{-1}t\sup_{0\leq s\leq t}\|D_\barV\Theta\|_{H^2(\Omega_s)}.
	\end{split}
	\end{align*}
	We apply this estimate to $\Theta= \nabla^{a} D_V^k V$. Then, by \eqref{a priori assumption H2 V}, as long as $a+2\leq p$ and $k+1\leq M+2-2(a+2)$, 
	\begin{align*}
	\begin{split}
	\sup_{0\leq s\leq t}\|D_V\Theta\|_{H^2(\Omega_s)}\lesssim_{C_{\ell}} 1.
	\end{split}
	\end{align*}
	Therefore, estimate \eqref{eq:Linfty1} follows by taking $\barT$ small. The argument for $\nabla^{a}D_{V}^{k+1}\sigma^{2}$ is similar.
\end{proof}
We have a similar estimate for the $L^{2}$ norms:
\begin{lemma}\label{lem:L 2 1}
	Under the bootstrap assumption \eqref{a priori assumption H2 V}, if $\barT>0$ is sufficiently small, then
	\begin{align}\label{eq:L 2 1}
	\begin{split}
	\left\|\nabla^{a}D_\barV^k \left(\barV^{0}-c,\barV^{i}\right)\right\|_{L^2(\Omega_t)}+\|\nabla^{a} D_\barV^{k+1} \barsigma^{2}\|_{L^2(\Omega_t)}\lesssim 1,\qquad \forall 2a+k\leq M+1,\quad t\in[0,T],
	\end{split}
	\end{align}
	where the implicit constant is independent of $C_{M}$.
\end{lemma}
\begin{proof}
	The proof is again an application of the fundamental theorem of calculus, this time applied to $\|\Theta\|_{L^2(\Omega_t)}$, where we also bound the Jacobian of the Lagrangian coordinate transformation from $\Omega_0$ to $\Omega_t$ using the fundamental theorem of calculus. We omit the details.
\end{proof}
\begin{proof}[Proof of Proposition \ref{prop:L2Sobolev}]
Note that we only need to consider $\partial_{x}^{p}D_{\barV}^{k}\barV$. Indeed, using induction on the order of $\partial_{t}$, for $\partial_{t}\partial_{x}^{p-1}D_{\barV}^{k}\barV$, we have
	\begin{align*}
	\partial_{t}\partial_{x}^{p-1}D_{\barV}^{k}\barV=&\frac{1}{\barV^{0}}\cdot \barV^{0}\partial_{t}\partial_{x}^{p-1}D_{\barV}^{k}\barV=\frac{1}{\barV^{0}}\left(D_{\barV}\partial_{x}^{p-1}D_{\barV}^{k}\barV-\barV^{j}\partial_{j}\partial^{p-1}_{x}D_{\barV}^{k}\barV\right)\\
	=&\frac{1}{\barV^{0}}\left(\partial_{x}^{p-1}D_{\barV}^{k+1}\barV-\barV^{j}\partial_{j}\partial_{x}^{p-1}D_{\barV}^{k}\barV+[D_{\barV},\partial_{x}^{p-1}]D_{\barV}^{k}\barV\right).
	\end{align*}
	If we can estimate $\partial_{t}^{p'}\partial_{x}^{p-p'}D_{\barV}^{k}\barV$, for $\partial_{t}^{p'+1}\partial_{x}^{p-p'-1}D_{\barV}^{k}\barV$, we have
	\begin{align*}
	\partial_{t}^{p'+1}\partial_{x}^{p-p'-1}D_{\barV}^{k}\barV=&\partial_{t}\partial_{t}^{p'}\partial_{x}^{p-p'-1}D_{\barV}^{k}\barV.
	\end{align*}
	The induction argument follows exactly the same way as we treat the case when $p'=0$. The argument for $\barsigma^{2}$ is the same.
Turning to  $\partial_{x}^{p}D_{\barV}^{k}\barV$,	we will use an induction argument on $p$. When $p=1$, the result follows directly by definition. Now we assume that the estimate holds for index less or equal to $1\leq p\leq\frac{M+2}{2}-1$, that is,
	\begin{align}\label{eq:JIH}
	\begin{split}
	&\sum_{q\leq p}\sum_{k+2q\leq M+2}\|\partial_{t,x}^q D_\barV^{k}\barV\|_{L^2(\Omega_t)}^2+\sum_{q\leq p}\sum_{k+2q\leq M+2}\|\partial_{t,x}^{q}D_{\barV}^{k+1}\barsigma^{2}\|_{L^{2}(\Omega_{t})}^{2}\\
	&\lesssim \sup_{0\leq \tau\leq t}\Ebar_{\leq M+1}[\barsigma^2,\tau]+ \sup_{0\leq \tau\leq t}E_{\leq M}[\barV,\tau]\\
	&\quad\sum_{q\leq p}\sum_{k+2q\leq M+2}\|\partial_{t,x}^q D_\barV^{k}\barV\|_{L^2(\Omega_0)}^2+\sum_{q\leq p}\sum_{k+2q\leq M+2}\|\partial_{t,x}^{q}D_{\barV}^{k+1}\barsigma^{2}\|_{L^{2}(\Omega_{0})}^{2},
	\end{split}
	\end{align}
	and prove the estimates for $p+1$, that is,
	\begin{align}\label{eq:JIG}
	\begin{split}
	&\sum_{k\leq M-2p}\|\partial_{t,x}^{p+1} D_\barV^{k}\barV\|_{L^2(\Omega_t)}^2+\sum_{k\leq M-2p}\|\partial_{t,x}^{p+1}D_{\barV}^{k+1}\barsigma^{2}\|_{L^{2}(\Omega_{t})}^{2}\\
	&\lesssim \sup_{0\leq \tau\leq t}\Ebar_{\leq M+1}[\barsigma^2,\tau]+ \sup_{0\leq \tau\leq t}E_{\leq M}[\barV,\tau]\\
	&\quad\sum_{q\leq p+1}\sum_{k+2q\leq M+2}\|\partial_{t,x}^q D_\barV^{k}\barV\|_{L^2(\Omega_0)}^2+\sum_{q\leq p+1}\sum_{k+2q\leq M+2}\|\partial_{t,x}^{q}D_{\barV}^{k+1}\barsigma^{2}\|_{L^{2}(\Omega_{0})}^{2}.
	\end{split}
	\end{align}
	We start with the estimate for $\|\nabla_{x}^{p+1}D_{\barV}^{k+1}\barsigma^{2}\|_{L^{2}(\Omega_{t})}$ and in fact first estimate $\|\nabla_x^{(2)}\slashed{\nabla}^{(p-1)}D_{\barV}^{k+1}\barsigma^{2}\|_{L^{2}(\Omega_{t})}$. To apply Lemma \ref{lem:A1} to $\Theta:=\slashed{\nabla}^{p-1}D_{\barV}^{k+1}\barsigma^{2}$ we need to estimate $\|\barA\slashed{\nabla}^{p-1}D_{\barV}^{k+1}\barsigma^{2}\|_{L^2(\Omega_t)}$. Using the notation of Lemma \ref{lem:boxDVk1sigma}, we have\footnote{We use the schematic notation $A\sim A_1+\dots+A_m$ to mean $A$ is a linear combination of terms of the forms $A_1,\dots,A_m$.}
	\begin{align}\label{higher order A sigma tangential}
	\begin{split}
	\barA\slashed{\nabla}^{p-1}D_{\barV}^{k+1}\barsigma^{2}&\sim \slashed{\nabla}^{p-1}H_{k}+[\slashed{\nabla}^{p-1},\Box]D_{\barV}^{k+1}\barsigma^{2}+\frac{\barV\cdot \nabla \barV}{(\barV^{0})^{2}}\nabla\slashed{\nabla}^{p-1}D_{\barV}^{k+1}\barsigma^{2}\\
	&\quad+\frac{\barV}{(\barV^{0})^{2}}\nabla \slashed{\nabla}^{p-1}D_{\barV}^{k+2}\barsigma^{2}+\frac{\barV}{(\barV^{0})^{2}}\nabla[D_{\barV},\slashed{\nabla}^{p-1}]D_{\barV}^{k+1}\barsigma^{2}.
	\end{split}
	\end{align}
	Except for  $\slashed{\nabla}^{p-1}H_{k}$ the $L^2(\Omega_t)$ norms of all the terms on the right-hand side of \eqref{higher order A sigma tangential} are bounded by the right-hand side of \eqref{eq:JIG} using the induction hypothesis \eqref{eq:JIH}. Here for the terms where derivatives hit the coefficients of $\snabla$ it suffices to observe that these coefficients are functions of $\nabla\barsigma^2=-2D_\barV\barV$.
	Next we investigate the structure of $\slashed{\nabla}^{p-1}H_{k}$. In view of Lemma \ref{lem:boxDVk1sigma}, the top order terms in $\slashed{\nabla}^{p-1}H_{k}$ are 
	\begin{align*}
	\nabla^{p+1}D_{\barV}^{k}\barsigma^{2}\mand \nabla^{p+1}D_{\barV}^{k-1}\barV.
	\end{align*}
	The $L^{2}(\Omega_{t})$ norm of all other term appearing in $\snabla^{(p-1)}H_k$ can be bounded by the right-hand side of \eqref{eq:JIG} using  the using induction hypothesis \eqref{eq:JIH}. For the two top order terms above, since $k\leq M-2p$ we can use Lemma \ref{lem:L 2 1} to bound the $L^{2}(\Omega_t)$ norms of these terms by the right-hand side of \eqref{eq:JIG} as well. Based on this discussion, Using \eqref{higher order A sigma tangential} and Lemma~\ref{lem:A1}, for any $k\leq M-2p$ we obtain
	\begin{align}\label{Hs esti tangential sigma 1st iterate}
	\begin{split}
	\|\nabla^{2}\slashed{\nabla}^{p-1}D_{\barV}^{k+1}\barsigma^{2}\|_{L^{2}(\Omega_{t})}&\lesssim \sup_{0\leq \tau\leq t}\Ebar_{\leq M+1}[\barsigma^2,\tau]+ \sup_{0\leq \tau\leq t}E_{\leq M}[\barV,\tau]\\
	&\quad\sum_{q\leq p+1}\sum_{k+2q\leq M+2}\|\partial_{t,x}^q D_\barV^{k}\barV\|_{L^2(\Omega_0)}^2+\sum_{q\leq p+1}\sum_{k+2q\leq M+2}\|\partial_{t,x}^{q}D_{\barV}^{k+1}\barsigma^{2}\|_{L^{2}(\Omega_{0})}^{2}.
	\end{split}
	\end{align}
	Next we apply Lemma \ref{lem: reduce to tangential} to $\Theta:=\nabla\slashed{\nabla}^{p-2}D_{\barV}^{k+1}\barsigma^{2}$ to get
	\begin{align}\label{Hs sigma 2nd iterate pre}
	\begin{split}
	\|\nabla \snabla^{p-2}D_\barV^{k+1}\barsigma^2\|_{H^2(\Omega_t)}&\lesssim \|\nabla\snabla^{p-2}\barA D_\barV^{k+1}\barsigma^2\|_{L^2(\Omega_t)}+ \|[\nabla\snabla^{p-2},\barA]D_\barV^{k+1}\barsigma^2\|_{L^2(\Omega_t)}+\|\snabla^{p-2}D_\barV^{k+1}\barsigma^2\|_{H^2(\Omega_t)}\\
	&\quad +\|[\nabla,\snabla]\snabla^{p-2}D_\barV^{k+1}\barsigma^2\|_{H^1(\Omega_t)}+\|\snabla^{p-1}D_\barV^{k+1}\barsigma^2\|_{H^2(\Omega_t)}.
	\end{split}
	\end{align}
	By \eqref{Hs esti tangential sigma 1st iterate} and the arguments leading to it, all the terms on the right-hand side of \eqref{Hs sigma 2nd iterate pre} except  $$\|\nabla\snabla^{p-2}\barA D_\barV^{k+1}\barsigma^2\|_{L^2(\Omega_t)}$$ are bounded by the right-hand side of \eqref{eq:JIG}.
	The term $\|\nabla\slashed{\nabla}^{p-2}\barA D_{\barV}^{k+1}\barsigma^{2}\|_{L^{2}(\Omega_{t})}$ is bounded in the same way as in the treatment of $\snabla^{p-1}H_k$ above, using Lemmas~\ref{lem:boxDVk1sigma} and~\ref{lem:L 2 1}. Summarizing we have obtained
	\begin{align}\label{Hs sigma 2nd iterate}
	\begin{split}
	\|\nabla^{3}\slashed{\nabla}^{p-2}D_{\barV}^{k+1}\barsigma^{2}\|_{L^{2}(\Omega_{t})}&\lesssim \sup_{0\leq \tau\leq t}\Ebar_{\leq M+1}[\barsigma^2,\tau]+ \sup_{0\leq \tau\leq t}E_{\leq M}[\barV,\tau]\\
	&\quad\sum_{q\leq p+1}\sum_{k+2q\leq M+2}\|\partial_{t,x}^q D_\barV^{k}\barV\|_{L^2(\Omega_0)}^2+\sum_{q\leq p+1}\sum_{k+2q\leq M+2}\|\partial_{t,x}^{q}D_{\barV}^{k+1}\barsigma^{2}\|_{L^{2}(\Omega_{0})}^{2}.
	\end{split}
	\end{align}
	Repeating the  argument inductively for $\Theta:=\nabla^{2}\slashed{\nabla}^{p-3}D_{\barV}^{k+1}\barsigma^{2}, \nabla^{3}\slashed{\nabla}^{p-4}D_{\barV}^{k+1}\barsigma^{2}, ...$, we finally obtain
	\begin{align}\label{Hs sigma final}
	\begin{split}
	\|\nabla^{p+1}D_{\barV}^{k+1}\barsigma^{2}\|_{L^{2}(\Omega_{t})}&\lesssim \sup_{0\leq \tau\leq t}\Ebar_{\leq M+1}[\barsigma^2,\tau]+ \sup_{0\leq \tau\leq t}E_{\leq M}[\barV,\tau]\\
	&\quad\sum_{q\leq p+1}\sum_{k+2q\leq M+2}\|\partial_{t,x}^q D_\barV^{k}\barV\|_{L^2(\Omega_0)}^2+\sum_{q\leq p+1}\sum_{k+2q\leq M+2}\|\partial_{t,x}^{q}D_{\barV}^{k+1}\barsigma^{2}\|_{L^{2}(\Omega_{0})}^{2}.
	\end{split}
	\end{align}
	Next we use the second estimate in Lemma \ref{lem: elliptic estimate} to estimate $\|\nabla^{p+1}D_{\barV}^{k}\barV\|_{L^{2}(\Omega_{t})}, k+2p+2\leq M+2$, under the induction hypothesis \eqref{eq:JIH}. The second estimate in Lemma \ref{lem: elliptic estimate} gives
	\begin{align}\label{Hs V pre}
	\begin{split}
	\|\nabla^{p+1}D_{\barV}^{k}\barV\|_{L^{2}(\Omega_{t})}\lesssim \|\barA\nabla^{p-1}D_{\barV}^{k}\barV\|_{L^{2}(\Omega_{t})}+\|(\nabla^\mu\barsigma^{2})\partial_\mu\left(\nabla^{p-1}D_{\barV}^{k}\barV\right)\|_{H^{\frac{1}{2}}(\partial\Omega_{t})}.
	\end{split}
	\end{align}
	The term $\barA\nabla^{p-1}D_{V}^{k}V$ has the similar structure to the corresponding term in \eqref{higher order A sigma tangential} and can be handled using similar considerations,
	so we concentrate on the boundary contribution $\|(\nabla^\mu\barsigma^{2})\partial_\mu\left(\nabla^{p-1}D_{\barV}^{k}\barV\right)\|_{H^{\frac{1}{2}}(\partial\Omega_{t})}$. Using the trace theorem and Lemma~\ref{lem:Vho}
	\begin{align*}
	\begin{split}
	\|(\nabla^\mu\barsigma^{2})\partial_\mu\left(\nabla^{p-1}D_{\barV}^{k}\barV\right)\|_{H^{\frac{1}{2}}(\partial\Omega_{t})}&\lesssim \|(\nabla^\mu\barsigma^{2})\partial_\mu\left(\nabla^{p-1}D_{\barV}^{k}\barV\right)\|_{H^{1}(\Omega_{t})}\\
	&\lesssim \|[(\nabla^\mu\barsigma^{2})\partial_\mu,\nabla^{p-1}]D_{\barV}^{k}\barV\|_{H^{1}(\Omega_{t})}+\|\nabla^{p-1}D_{\barV}^{k+2}\barV\|_{H^{1}(\Omega_{t})}\\
	&\quad+\|\nabla^{p}D_{\barV}^{k+1}\barsigma^2\|_{H^{1}(\Omega_{t})}+\|\nabla^{p-1}F_k\|_{H^{1}(\Omega_{t})}.
	\end{split}
	\end{align*}
	Except for the last term $\|\nabla^{p-1}F_k\|_{H^{1}(\Omega_{t})}$, all other terms on the right above are bounded by the right-hand side of \eqref{eq:JIG} using the induction hypothesis\footnote{Here note that $k\leq M-2p$ is equivalent to $k+2\leq M+2-2p$, so $\|\nabla^{p-1}D_{\barV}^{k+2}\barV\|_{H^{1}(\Omega_{t})}$ can indeed be bounded by the induction hypothesis \eqref{eq:JIH}.} \eqref{eq:JIH} and \eqref{Hs sigma final}. For $\|\nabla^{p-1}F_k\|_{H^{1}(\Omega_{t})}$, in view of Lemma~\ref{lem:Vho} the highest order terms in $\nabla^pF_k$ are
	\begin{align*}
	\nabla^{p+1}D_{\barV}^{k-1}\barV,\quad \textrm{and}\quad \nabla^{p+1}D_{\barV}^{k}\barsigma^{2}.
	\end{align*}
	The term $\|\nabla^{p+1}D_{\barV}^{k}\barsigma^{2}\|_{L^2(\Omega_t)}$ was already bounded in \eqref{Hs sigma final}, and $\|\nabla^{p+1}D_{\barV}^{k-1}\barV\|_{L^2(\Omega_t)}$ can be handled using Lemma \ref{lem:L 2 1}. Putting everything together we have proved that
	\begin{align}\label{Hs V 2}
	\begin{split}
	\|(\nabla^\mu\barsigma^{2})\partial_\mu\left(\nabla^{p-1}D_{\barV}^{k}\barV\right)\|_{H^{\frac{1}{2}}(\partial\Omega_{t})}&\lesssim \sup_{0\leq \tau\leq t}\Ebar_{\leq M+1}[\barsigma^2,\tau]+ \sup_{0\leq \tau\leq t}E_{\leq M}[\barV,\tau]\\
	&\quad\sum_{q\leq p+1}\sum_{k+2q\leq M+2}\|\partial_{t,x}^q D_\barV^{k}\barV\|_{L^2(\Omega_0)}^2+\sum_{q\leq p+1}\sum_{k+2q\leq M+2}\|\partial_{t,x}^{q}D_{\barV}^{k+1}\barsigma^{2}\|_{L^{2}(\Omega_{0})}^{2}.
	\end{split}
	\end{align}
	Combining \eqref{Hs V pre} and \eqref{Hs V 2}, we finally obtain
	\begin{align}\label{Hs V final}
	\begin{split}
	\|\nabla^{p+1}D_{\barV}^{k}\barV\|_{L^{2}(\Omega_{t})}&\lesssim \sup_{0\leq \tau\leq t}\Ebar_{\leq M+1}[\barsigma^2,\tau]+ \sup_{0\leq \tau\leq t}E_{\leq M}[\barV,\tau]\\
	&\quad\sum_{q\leq p+1}\sum_{k+2q\leq M+2}\|\partial_{t,x}^q D_\barV^{k}\barV\|_{L^2(\Omega_0)}^2+\sum_{q\leq p+1}\sum_{k+2q\leq M+2}\|\partial_{t,x}^{q}D_{\barV}^{k+1}\barsigma^{2}\|_{L^{2}(\Omega_{0})}^{2},
	\end{split}
	\end{align}
	which completes the proof of \eqref{eq:JIG}.
\end{proof}
\subsection{Proof of Proposition~\ref{prop:apriori}}\label{subsec:aprioriproof}
We are now ready to prove Proposition~\ref{prop:apriori}. Throughout the proof we use the fact that in view of Corollary~\ref{cor:aV}
\begin{align*}
\begin{split}
E[\Theta,t]\simeq \int_{\Omega_t}\left(c^{-1}D_\barV\Theta \partial_t\Theta+\frac{\barV^0}{2c}\nabla^\mu\Theta\nabla_\mu\Theta\right)\ud x+\int_{\partial\Omega_t}\frac{\barV^0}{c\bara}(D_\barV\Theta)^2\ud S.
\end{split}
\end{align*}
Recall that our goal is to prove estimate \eqref{eq:apriori}.
The following auxiliary lemma which relies on elliptic estimates for $A$ is an important ingredient of the proof.
\begin{lemma}\label{lem:aux1}
	Suppose the hypotheses of Proposition~\ref{prop:apriori} hold. Then for any $k\leq \ell$ and $t\in[0,T]$
	\begin{align}\label{eq:ellipticaux1}
	\begin{split}
	\int_{\Omega_t}|\nabla^{(2)}D_\barV^{k}\barsigma^2|^2\ud x\lesssim \calE_\ell(T)+\sum_{m+2\leq \ell}\|\partial_{t,x}^{2}D_{\barV}^{m}\barV\|^{2}_{L^{2}(\Omega_{0})}+\sum_{m+2\leq\ell}\|\partial^{2}_{t,x}D_{\barV}^{m+1}\barsigma^{2}\|^{2}_{L^{2}(\Omega_{0})}.
	\end{split}
	\end{align}
	Here the implicit constant depends polynomially on $\calE_{k-1}(T)$ if $k$ is sufficiently large, and on $C_1$ for small $k$.
\end{lemma}
\begin{proof}
	For $k=0$ this follows for instance by writing
	\begin{align*}
	\begin{split}
	\nabla^\mu \barsigma^2 = -2 D_\barV \barV^\mu.
	\end{split}
	\end{align*}
	Inductively, suppose the statement of the lemma holds for $k\leq j-1\leq \ell-1$ and let us prove it for $k=j$. In view of Corollary~\ref{cor:aV} the operator $\barA$ is elliptic, so applying Lemma~\ref{lem:A1} with $\Theta=D_\barV^j\barsigma^2$ we get
	\begin{align*}
	\begin{split}
	\int_{\Omega_t}|\nabla^{(2)}\barD_V^{j}\barsigma^2|^2\ud x\lesssim \int_{\Omega_t}|\partial_{t,x}D_\barV^{j+1}\barsigma^2|^2\ud x +\int_{\Omega_t}|\partial_{t,x}D_\barV^{j}\barsigma^2|^2\ud x +\int_{\Omega_t}|H_j|^2\ud x,
	\end{split}
	\end{align*}
	where $H_{j}$ is given in Lemma \ref{lem:boxDVk1sigma}. The first two terms already have the right form so we concentrate on the last term 
	for which we use Lemmas~\ref{lem:boxDVk1sigma} and~\ref{lem:Linfty1}. The contribution of line \eqref{eq:Hk1} in Lemma \ref{lem:boxDVk1sigma} is bounded by
\begin{align}\label{esti:ellipticaux1}
	\sum_{k\leq j-1}\int_{\Omega_t}|\nabla^{(2)}D_{V}^{k}\barsigma^{2}|^{2}+\sum_{k\leq j-1}\int_{\Omega_{t}}|\nabla D_{\barV}^{k}\barV|^{2}dx.
\end{align}	
The first term in \eqref{esti:ellipticaux1} can be bounded using induction hypothesis, while the second term is directly bounded by $\calE_{\ell-1}(T)\leq \calE_{\ell}(T)$. The contribution from line \eqref{eq:Hk2} in Lemma \ref{lem:boxDVk1sigma} is bounded by 
\begin{align}\label{esti:ellipticaux2}
	\begin{split}
	\sum_{k\leq j-2}\int_{\Omega_{t}}|\nabla^{(2)}D_{\barV}^{k}\barV|^{2}dx+\sum_{k\leq j-1}\int_{\Omega_{t}}\left(|\nabla D_{\barV}^{k}\barsigma^{2}|^{2}+|\nabla D_{\barV}^{k}\barV|^{2}\right)dx.
	\end{split}
\end{align}
The second term in \eqref{esti:ellipticaux2} is bounded by $\calE_{\ell-1}(T)$. The first term can be bounded using  Proposition~\ref{prop:L2Sobolev}, because $2\cdot 2+\ell-2=\ell+2$. The contribution of line \eqref{eq:Hk3} in Lemma~\ref{lem:boxDVk1sigma} is directly bounded by $\calE_{\ell-1}(T)$.
	\end{proof}
The next lemma allows us to estimate $\nabla D_\barV^{k-1}\barV$ on the boundary assuming boundedness of the $k$th energy.
\begin{lemma}\label{lem:nablaDV1}
	Suppose the hypotheses of Proposition~\ref{prop:apriori} hold. Given $\eta>0$ (small), if $\barT>0$ is sufficiently small then for any $k\leq \ell$ (recall that $T=c\barT$)
	\begin{align}\label{eq:nablaDV1}
	c^{-1}\int_{0}^{T}\int_{\partial\Omega_t}|\nabla D_\barV^j\barV|^2\ud S\ud t \lesssim &\scE_\ell(0)+R_{j,\eta}(\calE_{k-1}(T))+ \eta \calE_{k}(T),
	\end{align}
	where the implicit constant is independent of $C_1$, and $R_{j,\eta}$ is some polynomial function for each $j\leq k-1$.
\end{lemma}
\begin{proof}
	For $j=0$ this follows for instance from Corollary~\ref{cor:Sobolev} and the trace theorem. Proceeding inductively, we assume \eqref{eq:nablaDV1} holds for $j\leq i-1\leq k-2$ and prove it for $j=i$. We apply Lemma~\ref{lem:oblique} to $\Theta=D_\barV^i\barV$. The last integral on the right in \eqref{eq:oblique} can be absorbed in the left if $\barT$ is sufficiently small. The term
	\begin{align*}
	\begin{split}
	c^{-1}\int_0^{c\barT}\int_{\partial\Omega_t}\left(D_\barV D_\barV^i \barV\right)^2 \ud S \ud t
	\end{split}
	\end{align*}
	is bounded by the right-hand side of \eqref{eq:nablaDV1} if $\barT$ is sufficiently small. For the contribution $nD_\barV^i\barV$ on $\partial\Omega$ in the right-hand side of \eqref{eq:oblique} we use \eqref{eq:Vho} to write
	\begin{align}\label{eq:nVeq}
	\begin{split}
	-nD_\barV^i\barV =\frac{2}{\bara}D_\barV D_\barV^{i+1}\barV+\frac{1}{\bara}\nabla D_\barV^{i+1}\barsigma^2-\frac{2}{\bara}F_i,
	\end{split}
	\end{align}
	where $F_i$ is as in Lemma~\ref{lem:Vho}.
	Using Lemma \ref{lem:Linfty1}, the contribution of $F_i$ is bounded by
	\begin{align}\label{esti:auxbdy1}
	\begin{split}
		c^{-1}\sum_{m\leq i}\int_{0}^{c\barT}\int_{\partial\Omega_{t}}|\nabla D_{\barV}^{m}\barsigma^{2}|^{2}dx+c^{-1}\sum_{m\leq i-1}\int_{0}^{c\barT}\int_{\partial\Omega_{t}}|\nabla D_{\barV}^{m}\barV|^{2}dx.
		\end{split}
	\end{align}
	The first term on the right-hand side of \eqref{esti:auxbdy1} is bounded by $\calE_{k-2}(T)$, and hence by the right-hand side of  \eqref{eq:nablaDV1}. Using the trace theorem, Proposition~\ref{prop:L2Sobolev}, and choosing $\barT>0$ sufficiently small, the second term in \eqref{esti:auxbdy1} is bounded by the right-hand side of \eqref{eq:nablaDV1}.

	 Next, the integral
	\begin{align*}
	\begin{split}
	c^{-1}\int_0^T\int_{\partial\Omega_t}\left|\frac{1}{\bara}D_\barV D_\barV^{i+1}\barV\right|^2 \ud S\ud t \end{split}
	\end{align*}
	from the right-hand side of \eqref{eq:nVeq} is bounded by $\eta\calE_k(T)$ if $\barT$ is small, because by assumption $i\leq k-1$. By the same restriction on $i$ the integral
	\begin{align*}
	\begin{split}
	c^{-1}\int_0^T\int_{\partial\Omega_t}\left|\frac{1}{\bara}\nabla D_\barV^{i+1}\barsigma^2\right|^2\ud S\ud t
	\end{split}
	\end{align*} 
	is bounded by $\calE_{k-1}(T)$. To complete the proof of the lemma we still need to consider the term
	\begin{align*}
	\begin{split}
	c^{-1}\int_0^T\int_{\Omega_t}(\Box D_\barV^i\barV)(QD_\barV^i\barV)\ud x \ud t
	\end{split}
	\end{align*}
	on the right-hand side of \eqref{eq:oblique}. Here $Q$ is the multiplier in Lemma~\ref{lem:oblique}. By Lemma~\ref{lem:boxDVkV}, $\Box D_\barV^i\barV$ contains terms which involve $\nabla^{(2)}D_\barV^{i-1}\barV$. Since $i\leq k-1$, this corresponds to (at worst) $\nabla^{(2)} D_\barV^{k-2}\barV$. But in view of Corollary~\ref{cor:Sobolev}, $c^{-1}\int_{0}^{T}\int_{\Omega_t}|\nabla^{2}D^{i-1}_{\barV}\barV|^{2}\ud x\ud t$ is bounded by the right-hand side of \eqref{eq:nablaDV1}, if $\barT>0$ is sufficiently small. The term $c^{-1}\int_{0}^{T}\int_{\Omega_{t}}|QD_{\barV}^{i}\barV|^{2}\ud x \ud t$ is simply bounded by $\calE_{k-1}(T)$.
\end{proof}
	We turn to the proof of Proposition~\ref{prop:apriori}.
	\begin{proof}[Proof of  Proposition~\ref{prop:apriori}]
		We prove estimate \eqref{eq:apriori} inductively. First by Lemma~\ref{lem:energy} applied to \eqref{eq:mainbar}, and Corollary~\ref{cor:Sobolev}
		\begin{align*}
		\begin{split}
		\sup_{0\leq t\leq T}E[\barV,t]\lesssim E[\barV,0]+c^{-1}T (1+C_1)^m,
		\end{split}
		\end{align*}
		for some $m>0$. Therefore, if $\barT=c^{-1}T$ is sufficiently small
		\begin{align*}
		\begin{split}
		\sup_{0\leq t\leq T}E[\barV,t]\lesssim E[\barV,0].
		\end{split}
		\end{align*}
		Similarly, by Lemma~\ref{lem:benergy} applied to each of the equations in \eqref{eq:sigma1a} (note that $\barsigma^2$ and $D_\barV\barsigma^2$ are constant on $\partial\Omega$), and Corollary~\ref{cor:Sobolev},
		\begin{align*}
		\begin{split}
		\Ebar_{\leq 1}[\barsigma^2,T]\lesssim E_{\leq \ell}[\barV,0]+\Ebar_{\leq\ell+1}[\barsigma^2,0]. 
		\end{split}
		\end{align*}
		It follows that
		\begin{align*}
		\begin{split}
		\calE_{0}(T)\lesssim E_{\leq \ell}[\barV,0]+\Ebar_{\leq\ell+1}[\barsigma^2,0].
		\end{split}
		\end{align*}
		Now we assume that smallness of $\barT$ implies
		\begin{align}\label{eq:indhyp1}
		\begin{split}
		\scE_{ k-1}(T)\leq P_{k-1}(\scE_{\ell}(0))
		\end{split}
		\end{align}
		for some $1\leq k\leq \ell$ and some polynomial $P_{k-1}$, and use this to prove
		\begin{align}\label{eq:indcon1}
		\begin{split}
		\scE_{ k}(T)\leq P_k(\scE_{\ell}(0)),
		\end{split}
		\end{align}
		for some polynomial $P_k$, possibly by taking $\barT$ even smaller. In \eqref{eq:indhyp1} and \eqref{eq:indcon1} the polynomials $P_{k-1}$ and $P_k$ are taken to be independent of $C_1$. In fact, below we assume that $k$ is sufficiently large,
		because otherwise the desired bounds follow Corollary~\ref{cor:Sobolev} by taking $\barT$ small, in the same manner as above.
		
		{\bf{Step 1:}} First we show that
		\begin{align}\label{eq:step2goal}
		\begin{split}
		\Ebar_{k+1}[\barsigma^2,T]\leq \tilP_k(\calE_\ell(0))+\kappa \calE_k(T),
		\end{split}
		\end{align}
		for some polynomial $\tilP_k$, and where $\kappa$ is a given small absolute constant to be chosen later. The idea is to apply Lemma~\ref{lem:benergy} with $\Theta=D_\barV^{k+1}\barsigma^2$ to equation \eqref{eq:boxDVk1sigma}, which can be done because $D_\barV^{k+1}\barsigma^2\equiv0$ on $\partial\Omega$. Note that using Corollary~\ref{cor:Sobolev} to estimate $\nabla Q$ in $L^\infty$ (here $Q$ is as in the statement of Lemma~\ref{lem:benergy}) the last term on the right-hand side of \eqref{eq:benergy} can be absorbed on the left, provided $\barT$ is small, to give
		\begin{align}\label{eq:DVsigmatemp1}
		\begin{split}
		&\int_{\Omega_{\tau}}|\partial_{t,x}D_\barV^{k+1}\barsigma^2|^2\ud x+c^{-1}\int_0^{\tau}\int_{\partial\Omega_t}|\partial_{t,x}D_\barV^{k+1}\barsigma^2|^2\ud S \ud t\\
		&\lesssim \int_{\Omega_0}|\partial_{t,x}D_\barV^{k+1}\barsigma^2|^2\ud x+\left|c^{-1}\int_0^{\tau}\int_{\Omega_t}H_k\,QD_\barV^{k+1}\barsigma^2\ud x \ud t\right|,
		\end{split}
		\end{align}
		for any $\tau\leq T$, where $H_k$ is as in Lemma~\ref{lem:boxDVk1sigma}. If $H_k$ is of the form \eqref{eq:Hk3} in Lemma~\ref{lem:boxDVk1sigma}, then we can use Corollary~\ref{cor:Sobolev} and Cauchy-Schwarz on the last term on the right in \eqref{eq:DVsigmatemp1}, to bound this contribution by
		\begin{align*}
		\begin{split}
		\sum_{j\leq k}c^{-1}\int_0^\tau\int_{\Omega_t} |\nabla D_\barV^j\barV|^{2}\ud x\ud t+c^{-1}\int_0^\tau\int_{\Omega_t}|\nabla D_\barV^{k+1}\barsigma^2|^2\ud x \ud t.
		\end{split}
		\end{align*}
		If $\barT$ is small, the first term can be bounded by $\kappa \calE_k(T)$ and the second term can be absorbed on the left-hand side of \eqref{eq:DVsigmatemp1}. If $H_k$ is of the form \eqref{eq:Hk1} in Lemma~\ref{lem:boxDVk1sigma} we use elliptic estimates. The term that needs special attention is when $k_{m+1}=\max\{k_1,\dots,k_{m+1}\}$, in which case, after using Cauchy-Schwarz and Corollary~\ref{cor:Sobolev} as above, we need to estimate
		\begin{align*}
		\begin{split}
		c^{-1}\int_0^\tau\int_{\Omega_t}|\nabla^{(2)}D_\barV^{k}\barsigma^2|^2\ud x\ud t.
		\end{split}
		\end{align*}
		But, by Lemma~\ref{lem:aux1} this term is bounded by the right-hand side of \eqref{eq:step2goal} provided $\barT$ is sufficiently small.
		
		It remains to treat the contribution of $H_k$ replaced by \eqref{eq:Hk2} in Lemma~\ref{lem:boxDVk1sigma} to \eqref{eq:DVsigmatemp1}. Here, we only treat the most difficult case when $k_{m+2}=k-1$, and for brevity write the resulting expression in \eqref{eq:Hk2} in Lemma~\ref{lem:boxDVk1sigma} as
		\begin{align*}
		\begin{split}
		F^{\mu\nu}\nabla_\mu\nabla_\nu D_\barV^{k-1}\barV,
		\end{split}
		\end{align*}
		where in view of Corollary~\ref{cor:Sobolev}, $F$ satisfies
		\begin{align*}
		\begin{split}
		\|F\|_{L^\infty(\overline{\Omega})}+\|\nabla F\|_{L^\infty(\overline{\Omega})}\leq (\calE_{k-1}(T))^n
		\end{split}
		\end{align*}
		for some positive integer $n$. Replacing $H_k$ on the right-hand side of \eqref{eq:DVsigmatemp1} by this expression, we write
		\begin{align}\label{eq:DVsigmatemp2}
		\begin{split}
		(F^{\mu\nu}\nabla_\mu\nabla_\nu D_\barV^{k-1}\barV)(QD_\barV^{k+1}\barsigma^2)&=\nabla_\mu[(F^{\mu\nu}\nabla_{\nu}D_\barV^{k-1}\barV)(QD_\barV^{k+1}\barsigma^2)]-(\nabla_\mu F^{\mu\nu})(\nabla_\nu D_\barV^{k-1}\barV)(QD_\barV^{k+1}\barsigma^2)\\
		&\quad-(F^{\mu\nu}\nabla_\nu D_\barV^{k-1}\barV)(\nabla_\mu Q^\lambda)(\nabla_\lambda D_\barV^{k+1}\barsigma^2)\\
		&\quad-(F^{\mu\nu}\nabla_\nu D_\barV^{k-1}\barV)(Q\nabla_\mu D_\barV^{k+1}\barsigma^2).
		\end{split}
		\end{align}
		The last term can again be massaged as
		\begin{align*}
		\begin{split}
		(F^{\mu\nu}\nabla_\nu D_\barV^{k-1}\barV)(Q\nabla_\mu D_\barV^{k+1}\barsigma^2)
		&=\nabla_\lambda[(F^{\mu\nu}\nabla_\nu D_\barV^{k-1}\barV)(\barV^\lambda Q\nabla_\mu D_\barV^{k}\barsigma^2)]-(F^{\mu\nu}\nabla_\nu D_\barV^{k}\barV)(Q\nabla_\mu D_\barV^{k}\barsigma^2)\\
		&\quad-((D_\barV F^{\mu\nu})\nabla_\nu D_\barV^{k-1}\barV)(Q\nabla_\mu D_\barV^{k}\barsigma^2)+(F^{\mu\nu}(\nabla_\nu V^\lambda) \nabla_\lambda D_\barV^{k-1}\barV)(Q\nabla_\mu D_\barV^{k}\barsigma^2)\\
		&\quad -(F^{\mu\nu}\nabla_\nu D_\barV^{k-1}\barV)((D_\barV Q^\lambda)\nabla_\lambda\nabla_\mu D_\barV^{k}\barsigma^2)\\
		&\quad-(F^{\mu\nu}\nabla_\nu D_\barV^{k-1}\barV)(Q^\lambda[D_\barV,\nabla_\lambda\nabla_\mu] D_\barV^{k}\barsigma^2).
		\end{split}
		\end{align*}
		Plugging back into \eqref{eq:DVsigmatemp2} we get
		\begin{align}\label{eq:DVsigmatemp3}
		\begin{split}
		(F^{\mu\nu}\nabla_\mu\nabla_\nu D_\barV^{k-1}\barV)(QD_\barV^{k+1}\barsigma^2)&=\nabla_\mu[(F^{\mu\nu}\nabla_{\nu}D_\barV^{k-1}\barV)(QD_\barV^{k+1}\barsigma^2)-(F^{\lambda\nu}\nabla_\nu D_\barV^{k-1}\barV)(\barV^\mu Q\nabla_\lambda D_\barV^{k}\barsigma^2)]\\
		&\quad -(\nabla_\mu F^{\mu\nu})(\nabla_\nu D_\barV^{k-1}\barV)(QD_\barV^{k+1}\barsigma^2)-(F^{\mu\nu}\nabla_\nu D_\barV^{k-1}\barV)(\nabla_\mu Q^\lambda)(\nabla_\lambda D_\barV^{k+1}\barsigma^2)\\
		&\quad +(F^{\mu\nu}\nabla_\nu D_\barV^{k}\barV)(Q\nabla_\mu D_\barV^{k}\barsigma^2)+((D_\barV F^{\mu\nu})\nabla_\nu D_\barV^{k-1}\barV)(Q\nabla_\mu D_\barV^{k}\barsigma^2)\\
		&\quad-(F^{\mu\nu}(\nabla_\nu \barV^\lambda) \nabla_\lambda D_\barV^{k-1}\barV)(Q\nabla_\mu D_\barV^{k}\barsigma^2)+(F^{\mu\nu}\nabla_\nu D_\barV^{k-1}\barV)((D_\barV Q^\lambda)\nabla_\lambda\nabla_\mu D_\barV^{k}\barsigma^2)\\
		&\quad+(F^{\mu\nu}\nabla_\nu D_\barV^{k-1}\barV)(Q^\lambda[D_\barV,\nabla_\lambda\nabla_\mu] D_\barV^{k}\barsigma^2).
		\end{split}
		\end{align}
		We need to consider the integration of the terms on the right in \eqref{eq:DVsigmatemp3} over the space-time region $\cup_{0\leq t\leq \tau}\Omega_t$. The terms 
		\begin{align*}
		\begin{split}
		(\nabla_\mu F^{\mu\nu})(\nabla_\nu D_\barV^{k-1}\barV)(QD_\barV^{k+1}\barsigma^2)\mand (F^{\mu\nu}\nabla_\nu D_\barV^{k-1}\barV)(\nabla_\mu Q^\lambda)(\nabla_\lambda D_\barV^{k+1}\barsigma^2)
		\end{split}
		\end{align*}
		can simply be bounded by Cauchy-Schwarz, assuming $\barT$ is sufficiently small. The last five terms,
		\begin{align*}
		\begin{split}
		&(F^{\mu\nu}\nabla_\nu D_\barV^{k}\barV)(Q\nabla_\mu D_\barV^{k}\barsigma^2)\quad((D_\barV F^{\mu\nu})\nabla_\nu D_\barV^{k-1}\barV)(Q\nabla_\mu D_\barV^{k}\barsigma^2), \quad (F^{\mu\nu}(\nabla_\nu \barV^\lambda) \nabla_\lambda D_\barV^{k-1}\barV)(Q\nabla_\mu D_\barV^{k}\sigma^2),\\
		&(F^{\mu\nu}\nabla_\nu D_\barV^{k-1}\barV)((D_\barV Q^\lambda)\nabla_\lambda\nabla_\mu D_\barV^{k}\barsigma^2),\quad (F^{\mu\nu}\nabla_\nu D_\barV^{k-1}\barV)(Q^\lambda[D_\barV,\nabla_\lambda\nabla_\mu] D_\barV^{k}\barsigma^2),
		\end{split}
		\end{align*}
		can also be treated by Cauchy-Schwarz, this time combined with elliptic estimates as above, using Lemma~\ref{lem:aux1} and \eqref{eq:comtemp2}. For the first term, $\nabla_\mu I^\mu$, with
		\begin{align*}
		\begin{split}
		I^\mu:=(F^{\mu\nu}\nabla_{\nu}D_\barV^{k-1}\barV)(QD_\barV^{k+1}\barsigma^2)-(F^{\lambda\nu}\nabla_\nu D_\barV^{k-1}\barV)(\barV^\mu Q\nabla_\lambda D_\barV^{k}\barsigma^2),
		\end{split}
		\end{align*}
		on the right-hand side of \eqref{eq:DVsigmatemp3}, by the divergence theorem
		\begin{align*}
		\begin{split}
		c^{-1}\int_0^\tau\int_{\Omega_t} \nabla_\mu I^\mu \ud x \ud t&=c^{-1}\int_{\Omega_0}I^0 \ud x-c^{-1}\int_{\Omega_\tau}I^0\ud x+c^{-1}\int_{0}^\tau\int_{\partial\Omega_t} n_\mu I^\mu \ud S \ud t.
		\end{split}
		\end{align*}
		The first term on the right is bounded by the initial data. The second term on the right is bounded by
		\begin{align}\label{eq:divItemp1}
		\begin{split}
		C_\delta\int_{\Omega_\tau}|\partial_{t,x}D_\barV^{k-1}\barV|^2 \ud x+\delta \int_{\Omega_\tau}|\partial_{t,x} D_\barV^{k+1}\barsigma^2|^2\ud x,
		\end{split}
		\end{align}
		where $C_\delta$ depends polynomially on $\calE_{k-1}(T)$. The second term on the right in \eqref{eq:divItemp1} can be absorbed on the left-hand side of \eqref{eq:DVsigmatemp1} if $\delta$ is chosen sufficiently small (an absolute constant). The first term in \eqref{eq:divItemp1} is bounded by the right-hand side of \eqref{eq:step2goal} by the induction hypothesis. Finally, using the fact that $\barV^\mu n_\mu=0$ on $\partial\Omega$,
		\begin{align}\label{eq:divItemp2}
		\begin{split}
		c^{-1}\int_{0}^\tau\int_{\partial\Omega_t} n_\mu I^\mu \ud S \ud t\leq C_\delta c^{-1}\int_{0}^\tau\int_{\partial\Omega_t}|\nabla D_\barV^{k-1}\barV|^2 \ud S\ud t+\delta c^{-1} \int_0^\tau\int_{\partial\Omega_t}|\nabla D_\barV^{k+1}\barsigma^2|^2\ud S \ud t,
		\end{split}
		\end{align}
		where again $C_\delta$ can depend polynomially on $\calE_{k-1}(T)$. The last term on the right can be absorbed on the left in \eqref{eq:DVsigmatemp1} if $\delta$ is chosen sufficiently small. For the first term on the right in \eqref{eq:divItemp2} we use Lemma~\ref{lem:nablaDV1} with $\eta$ small (depending on $C_\delta$ in \eqref{eq:divItemp2} and $\kappa$ in \eqref{eq:step2goal}), where the second term on the right in \eqref{eq:nablaDV1} is bounded by the right-hand side of \eqref{eq:step2goal} using the induction hypothesis. This finishes the proof of \eqref{eq:step2goal}. 
{\bf{Step 2:}} Here we prove that given $\delta>0$, if $\barT$ is sufficiently small then for any $j\leq k$
\begin{align}\label{eq:step1goal}
\begin{split}
\left|\int_{0}^T\int_{\Omega_t}(D_\barV^{j+1}\barV)(\Box D_\barV^j\barV)\ud x\ud t\right|\lesssim \calE_\ell(0)+\tilde{\tilde{P}}_j(\calE_{k-1}(T))+\delta\calE_k(T) 
\end{split}
\end{align}
where the polynomial $\tilde{\tilde{P}}_j$ and the implicit constant are independent of $C_1$. In view of Lemma~\ref{lem:energy} this estimate is needed in estimating $$\sup_{0\leq t\leq T}E_{\leq k}[\barV,t].$$ 
Recall that $\Box D_\barV^j\barV= G_j$ where $G_j$ is as in Lemma~\ref{lem:boxDVkV}. We treat the hardest case when $k_{m+1}=\max\{k_1,\dots,k_{m+1}\}$ (the other cases can be handled using Cauchy-Schwarz and Corollary~\ref{cor:Sobolev}). In fact we concentrate on the most difficult case $k_{m+1}=j-1$. In this case we write $G_j$ as
\begin{align*}
\begin{split}
G^{\mu\nu}\nabla_\mu\nabla_\nu D_\barV^{j-1}\barV,
\end{split}
\end{align*}
where $G$ satisfies
\begin{align*}
\begin{split}
\|G\|_{L^\infty(\overline{\Omega})}+\|\nabla G\|_{L^\infty(\overline{\Omega})}\leq (\calE_{k-1}(T))^n
\end{split}
\end{align*}
for some integer $n\geq0$. We now proceed as in the derivation of \eqref{eq:DVsigmatemp3}. First
\begin{align}\label{eq:Gtemp1}
\begin{split}
(G^{\mu\nu}\nabla_\mu\nabla_\nu D_\barV^{j-1}\barV)(D_\barV^{j+1}\barV)&=\nabla_\mu[(G^{\mu\nu}\nabla_{\nu}D_\barV^{j-1}\barV)(D_\barV^{j+1}\barV)]-(\nabla_\mu G^{\mu\nu})(\nabla_\nu D_\barV^{j-1}\barV)(D_\barV^{j+1}\barV)\\
&\quad-(G^{\mu\nu}\nabla_\nu D_\barV^{j-1}\barV)(\nabla_\mu D_\barV^{j+1}\barV).
\end{split}
\end{align}
The last term can again be massaged as
\begin{align*}
\begin{split}
(G^{\mu\nu}\nabla_\nu D_\barV^{j-1}\barV)(\nabla_\mu D_\barV^{j+1}\barV)
&=\nabla_\lambda[(G^{\mu\nu}\nabla_\nu D_\barV^{j-1}\barV)(\barV^\lambda \nabla_\mu D_\barV^{j}\barV)]-(G^{\mu\nu}\nabla_\nu D_\barV^{j}\barV)(\nabla_\mu D_\barV^{j}\barV)\\
&\quad-((D_\barV G^{\mu\nu})\nabla_\nu D_\barV^{j-1}\barV)(\nabla_\mu D_\barV^{j}\barV)+(G^{\mu\nu}(\nabla_\nu \barV^\lambda) \nabla_\lambda D_\barV^{j-1}\barV)(\nabla_\mu D_\barV^{j}\barV)\\
&\quad +(G^{\mu\nu}\nabla_\nu D_\barV^{j-1}\barV)((\nabla_\mu \barV^\lambda)\nabla_\lambda D_\barV^{j}\barV).
\end{split}
\end{align*}
Plugging back into \eqref{eq:Gtemp1} we get
\begin{align}\label{eq:Gtemp2}
\begin{split}
(G^{\mu\nu}\nabla_\nu D_\barV^{j-1}\barV)(\nabla_\mu D_\barV^{j+1}\barV)&=\nabla_\mu[(G^{\mu\nu}\nabla_{\nu}D_\barV^{j-1}\barV)(D_\barV^{j+1}\barV)-(G^{\lambda\nu}\nabla_\nu D_\barV^{j-1}\barV)(\barV^\mu \nabla_\lambda D_\barV^{j}\barV)]\\
&\quad -(\nabla_\mu G^{\mu\nu})(\nabla_\nu D_\barV^{j-1}\barV)(D_\barV^{j+1}\barV)+(G^{\mu\nu}\nabla_\nu D_\barV^{j}\barV)(\nabla_\mu D_\barV^{j}\barV)\\
&\quad +((D_\barV G^{\mu\nu})\nabla_\nu D_\barV^{j-1}\barV)(\nabla_\mu D_\barV^{j}\barV)-(G^{\mu\nu}(\nabla_\nu \barV^\lambda) \nabla_\lambda D_\barV^{j-1}\barV)(\nabla_\mu D_\barV^{j}\barV)\\
&\quad +(G^{\mu\nu}\nabla_\nu D_\barV^{j-1}\barV)((\nabla_\mu \barV^\lambda)\nabla_\lambda D_\barV^{j}\barV).
\end{split}
\end{align}
We want to integrate \eqref{eq:Gtemp2} over $\cup_{t\in[0,T]}\Omega_t$. The contribution of the last five terms can be bounded by $\calE_\ell(0)$, as required in \eqref{eq:step1goal}, provided $\barT$ is small. It remains to consider the divergence terms
\begin{align*}
\begin{split}
\nabla_\mu I_1^\mu:= \nabla_\mu[(G^{\mu\nu}\nabla_{\nu}D_\barV^{j-1}\barV)(D_\barV^{j+1}\barV)],\qquad \nabla_\mu I_2^\mu:= \nabla_\mu[(G^{\lambda\nu}\nabla_\nu D_\barV^{j-1}\barV)(\barV^\mu \nabla_\lambda D_\barV^{j}\barV)].
\end{split}
\end{align*}
Note that $n_\mu I_2^\mu=0$ on $\partial\Omega$ because $n_\mu \barV^\mu=0$ there. Therefore the contribution of $\nabla_\mu I_2^\mu$ is bounded by the initial data plus
\begin{align*}
\begin{split}
\int_{\Omega_T}|\nabla D_\barV^{j-1}\barV||\nabla D_\barV^j \barV|\ud x\leq C_\delta \int_{\Omega_T}|\nabla D_\barV^{j-1}\barV|^2\ud x+ \delta\int_{\Omega_T}|\nabla D_\barV^j\barV|^2 \ud x,
\end{split}
\end{align*}
with $C_\delta$ depending polynomially on $\calE_{k-1}(T)$. The second term on the right is in the form required by \eqref{eq:step1goal}, and the first term on the right is bounded by the induction hypothesis. The contribution of $\nabla_\mu I_1^\mu$ on $\Omega_T$ is bounded in an identical fashion. On $\partial\Omega$ the contribution of $\nabla_\mu I_1^\mu$ is bounded by
\begin{align*}
\begin{split}
c^{-1}\int_0^T\int_{\partial\Omega_t}|\nabla D_\barV^{j-1}\barV|^2\ud S\ud t+c^{-1}\int_0^T\int_{\partial\Omega_t}|D_\barV^{j+1}\barV|^2\ud S\ud t
\end{split}
\end{align*}
which, using Lemma~\ref{lem:nablaDV1}, can be bounded by the right-hand side of \eqref{eq:step1goal} by choosing $\barT$ and $\eta$ in Lemma~\ref{lem:nablaDV1} small. This completes the proof of \eqref{eq:step1goal}.

{\bf{Step 3:}} Finally we show that
\begin{align}\label{eq:step3goal}
\begin{split}
\sup_{0\leq t \leq T}E_{k}[\barV,t]\leq S_k(\calE_\ell(0)),
\end{split}
\end{align}
for some polynomial $S_k$. Note that \eqref{eq:step3goal} and \eqref{eq:step2goal} complete the proof of the proposition upon taking $\kappa$ in \eqref{eq:step2goal} small. We apply the energy identity \eqref{eq:energyid1} to $\Theta= D_\barV^k\barV$. The first two terms on the right in \eqref{eq:energyid1} are bounded by the initial data. The last three terms there can be absorbed on the left by taking $\barT$ sufficiently small. The term $c^{-1}\int_0^T\int_{\Omega_t}g D_\barV\Theta \ud x \ud t$
with $g=\Box D_\barV^k\barV$ was treated in Step 2 above. Indeed, by \eqref{eq:step1goal} and the induction hypothesis, this term can be bounded by a polynomial of $\calE_\ell(0)$ plus a term which can be absorbed in the left in \eqref{eq:energyid1} and \eqref{eq:step2goal}. Finally we consider the term
\begin{align*}
\begin{split}
c^{-1}\int_{0}^T\int_{\partial\Omega_t}\frac{1}{\bara}f D_\barV\Theta \ud S\ud t\
\end{split}
\end{align*}
on the right-hand side of \eqref{eq:energyid1}, where $f= (D_\barV^2-\frac{1}{2}\bara n)D_\barV^k\barV$ as in Lemma~\ref{lem:Vho}. The main term $\nabla D_\barV^{k+1}\barsigma^2$ was already treated in \eqref{eq:step2goal} in Step 1 above. The contribution of the terms of the forms (1) and (2) in Lemma~\ref{lem:Vho} are also bounded by a polynomial of $\calE_\ell(0)$ or absorbed in the left in \eqref{eq:energyid1} and \eqref{eq:step2goal}, in view of Corollary~\ref{cor:Sobolev} and Lemma~\ref{lem:nablaDV1}, and by the induction hypothesis.
\end{proof}

\section{The Linear Theory}\label{sec:linear}
In this section we discuss the linearized equations for $V$ and $D_V\sigma^2$ and prove existence and energy estimates for them. Since this concerns the local existence result of Theorem~\ref{thm:main} which is for fixed $c$, to simplify notation we simply let $c=1$ in this section and in Section~\ref{sec:iteration}. Since the parameter $c$ is a constant it is clear that the proof is identical for any other choice of $c$. The assumptions on the coefficients and the source terms are of course such that they can be recovered in the iteration for the quasilinear system. The general scheme is the one outlined in Subsection~\ref{subsubsec:apriori} above, and involves proving Sobolev estimates in terms of the energies. However, as explained in Subsection~\ref{subsubsec:iteration} this scheme will be carried out on the Lagrangian side. The actual iteration for the nonlinear problem will be the subject of the next section.

\subsection{The weak formulation of the equations}\label{subsec:weakder}
\subsubsection{The equation for $V$} Recall from \eqref{eq:V1} that the boundary equation for $V$ is
	\begin{align}\label{boundary eq}
	\left(D_V^2+\frac{1}{2}a\nabla_n\right) V^\nu =-\frac{1}{2}\nabla^\nu D_V\sigma^2.
	\end{align}
When considering the linearized equation for $V$ we assume that in \eqref{boundary eq} the $V$ appearing in $D_{V}$,$a$ and $n$ are given, and that the right-hand side is replaced by a fixed function. To be more precise, we consider the following linear system for the unknown $\Theta$:
	\begin{align}\label{linear system Euler}
	\begin{split}
	&\Box\Theta=0,\quad \textrm{in}\quad \Omega,\\
	&D_{V}^{2}\Theta+\frac{1}{2}a\nabla_{n}\Theta=\tilf,\quad \textrm{on}\quad \partial\Omega.
	\end{split}
	\end{align}
At the linear level the Eulerian coordinates are given by 
	\begin{align}\label{Lagrangian def}
	\frac{\ud x^{0}(t,y)}{\ud t}=1,\qquad  \frac{\ud x^{i}(t,y)}{\ud t}=\left(\frac{V^{i}}{V^{0}}\right)(t,y), \quad i=1,2,3.
	\end{align}
On the Lagrangian domain the parameterization domain is $[0,T]\times B$ with timelike boundary $[0,T]\times \partial B$. The linearized Minkowski metric becomes
	\begin{align}\label{metric Lagrangian}
	\begin{split}
	g=&-\left(1-\sum_{i}\frac{(V^{i})^{2}}{(V^{0})^{2}}\right)\ud t^{2}+2\sum_{i,a}\frac{V^{i}}{V^{0}}\frac{\partial x^{i}}{\partial y^{a}}\ud t\ud y^{a}+\sum_{i,a,b}\frac{\partial x^{i}}{\partial y^{a}}\frac{\partial x^{i}}{\partial y^{b}}\ud y^{a}\ud y^{b}.
	\end{split}
	\end{align}
By redefining the source function $f$, we write the linear system \eqref{linear system Euler} on the Lagrangian side as
	\begin{align}\label{linear system Lagrangian}
	\begin{split}
	&\Box_{g}\Theta=0\quad \textrm{in}\quad B\times[0,T],\\
	&\partial_{t}^{2}\Theta+\gamma \nabla_{n}\Theta=f\quad\textrm{on}\quad \partial B\times[0,T],
	\end{split}
	\end{align}
where $\gamma:=\frac{a}{2(V^{0})^{2}}$ as in Subsection~\ref{subsubsec:iteration}. To derive the weak formulation as in Subsection~\ref{subsubsec:iteration}, we assume all functions are smooth and multiply the first equation in \eqref{linear system Lagrangian} by a test function $\fy$ and integrate by parts to get 
	\begin{align*}
	\begin{split}
	0&=\int_0^s\int_B\varphi\Box_g\Theta \,\ud \bary\ud t= \int_0^s\int_B \frac{1}{\sqrt{|g|}}\partial_\alpha(\varphi \sqrt{|g|}g^{\alpha\beta}\partial_\beta\Theta)\,\ud \bary\ud t-\int_0^s\int_B g^{\alpha\beta}\partial_\alpha\Theta \partial_\beta \fy \,\ud \bary\ud t\\
	&=-\int_{B_s}\partial_t\Theta \varphi \,\ud\bary+\int_{B_s}g^{ta}\partial_a\Theta \fy \,\ud\bary-\int_{B_0}g^{\alpha0}\partial_\alpha\Theta \varphi \,\ud\bary+\frac{1}{2}\int_0^s\int_B\varphi g^{\alpha\beta}\partial_\beta\Theta \partial_\alpha\log|g|\,\ud\bary\ud t\\
	&\quad+\int_0^s\int_{\partial B} \frac{1}{\gamma}(f-\partial_t^2\Theta)\fy\,\ud S\ud t-\int_0^s\int_B g^{\alpha\beta}\partial_\alpha\Theta \partial_\beta \fy \,\ud \bary\ud t.
	\end{split}
	\end{align*}
Since this identity holds for all $s$, we can differentiate in $s$ to get, after some manipulation,
\begin{align*}
\begin{split}
\int_{\partial B}\frac{1}{\gamma}f\varphi \ud S&
=\int_B \partial_t^2\Theta\fy\,\ud\bary+\int_{\partial B}\frac{1}{\gamma}\partial_t^2\Theta \varphi \,\ud S-\int_{\partial B}g^{tr}\partial_t\Theta\varphi\,\ud S-\int_B\partial_a\Theta \fy \partial_tg^{ta}\,\ud\bary+\int_B \partial_t\Theta\varphi \partial_ag^{ta}\,\ud\bary\\
&\quad+\int_B g^{\alpha\beta}\partial_\alpha\Theta \partial_\beta\fy \,\ud \bary-\int_B\partial_t\fy g^{0\alpha}\partial_\alpha\Theta \,\ud\bary+\int_B\partial_t\Theta  g^{ta}\partial_a\fy\,\ud\bary-\frac{1}{2}\int_B g^{\alpha\beta}\partial_\beta\Theta \fy \partial_\alpha\log|g|\,\ud\bary.
\end{split}
\end{align*}
Simplifying the line before last, we arrive at
	\begin{align}\label{eq:wtemp}
	\begin{split}
	\int_{\partial B}\frac{1}{\gamma}f\varphi \ud S&=\int_B \partial_t^2\Theta\fy\,\ud\bary+\int_{\partial B}\frac{1}{\gamma}\partial_t^2\Theta \varphi \,\ud S\\
	&\quad+\int_Bg^{ab}\partial_a\Theta\partial_b\fy\,\ud\bary+2\int_Bg^{ta}\partial_t\Theta\partial_a\fy\,\ud\bary-\int_{\partial B}g^{tr}\partial_t\Theta\varphi\,\ud S\\
	&\quad-\frac{1}{2}\int_B g^{\alpha\beta}\partial_\beta\Theta \fy \partial_\alpha\log|g|\,\ud\bary-\int_B\partial_a\Theta \fy \partial_tg^{ta}\,\ud\bary+\int_B \partial_t\Theta\varphi \partial_ag^{ta}\,\ud\bary.
	\end{split}
	\end{align}
	Equation \eqref{eq:wtemp} is our guide for formulating a weak problem. Recall that  $\angles{\cdot}{\cdot}$ denotes the inner product in $L^2(B)$ with respect to $\ud \bary$, and $\bangles{\cdot}{\cdot}$ denotes the inner product in $L^2(\partial B)$ with respect to the induced Euclidean measure $\ud S$. The pairing between $(H^1(B))^\ast$ and $H^1(B)$ is denoted by $(\cdot,\cdot)$. We define the bounded linear map $\Phi:H^1(B)\to(H^1(B))^\ast$ by
	\begin{align*}
	\begin{split}
	(\Phi(u),v):=\angles{u}{v}+\bangles{\gamma^{-1}\tr\,u}{\tr\,v}.
	\end{split}
	\end{align*}
Note that if $u$ is a sufficiently regular function of time, and $v$ is independent of time, we have
	\begin{align*}
	\begin{split}
	\angles{u''}{v}+\bangles{\gamma^{-1}\tr\,u''}{\tr\,v}=(\Phi(u)'',v)-2\bangles{(\gamma^{-1})'\tr\,u'}{\tr\,v}-\bangles{(\gamma^{-1})''\tr\,u}{\tr\,v}.
	\end{split}
	\end{align*}
	Also note that $\Phi$ is an embedding, because if $\Phi(u)=\Phi(w)$, then for all $\varphi\in C^\infty_0(B)$
	\begin{align*}
	\begin{split}
	\angles{u-w}{\varphi}=0
	\end{split}
	\end{align*}
	implying $u=w$ almost everywhere in $B$. But then
	\begin{align*}
	\begin{split}
	\bangles{\tr\,u-\tr\, w}{\gamma^{-1}\tr\,v}=0,\quad \forall v\in H^1(B),
	\end{split}
	\end{align*}
	and since $\tr:H^1(B)\to L^2(\partial B)$ is onto, it follows $\tr\, u =\tr\,w$. We define the following bilinear forms:
	\begin{align}\label{eq:defbilin}
	\begin{split}
	&B:H^1(B)\times H^1(B)\to \bbR,\qquad C:L^2(B)\times H^1(B)\to\bbR,\qquad D,E:L^2(\partial B)\times H^1(B)\to \bbR,\\
	&B(u,v):=\angles{g^{ab}\partial_au}{\partial_bv} -\frac{1}{2}\angles{\partial_au}{ vg^{a\alpha}\partial_\alpha\log |g|}-\angles{\partial_au}{v\partial_tg^{ta}},\\
	&C(u,v):=2\angles{u}{g^{ta} \partial_av}-\frac{1}{2}\angles{u}{v g^{t\alpha}\partial_\alpha\log |g|}+\angles{u}{v\partial_a g^{ta}},\\
	&D(u,v):=-\bangles{u}{g^{tr}\tr\,v}-2\bangles{u}{(\gamma^{-1})'\tr\,v},\\
	&E(u,v):=-\bangles{u}{(\gamma^{-1})''\tr\,v}.
	\end{split}
	\end{align}
	To simplify notation, for any $\Theta:[0,T]\to H^1(B)$ satisfying  
	\begin{align}\label{eq:Thetaspaces1}
	\begin{split}
	&\Theta\in L^2([0,T];H^1(B)),\quad \Theta'\in L^2([0,T],L^2(B)),\quad (\tr\,\Theta)'\in L^2([0,T];L^2(\partial B)),\\
	&\Phi(\Theta), \Phi(\Theta)', \Phi(\Theta)''\in L^2([0,T];(H^1(B))^\ast),
	\end{split}
	\end{align}
	let
	\begin{align}\label{eq:defcalL}
	\begin{split}
	\calL(\Theta,v):=B(\Theta,v)+C(\Theta',v)+D((\tr\,\Theta)',v)+E(\tr\,\Theta,v).
	\end{split}
	\end{align}
	The weak equation then becomes
	\begin{align}\label{eq:weak1}
	\begin{split}
	(\Phi(\Theta)'',v)+\calL(\Theta,v)=\bangles{\gamma^{-1}f}{\tr\,v},\qquad \forall v\in H^1(B),
	\end{split}
	\end{align}
	for almost every $t\in [0,T].$ 	To complete our formulation of the weak problem we also need to discuss the initial data. As in the model problem, the initial data consists of 
	\begin{align*}
	\begin{split}
	\theta_0\in H^1(B),\quad \theta_1\in L^2(B), \quad \tiltheta_1\in L^2(\partial B),
	\end{split}
	\end{align*}
	and the initial requirement on $\Theta$ is that
	\begin{align}\label{eq:weakdata1}
	\begin{split}
	&\Theta(0)=\theta_0\quad\mathrm{in~} L^2(B),\\
	&(\Phi(\Theta)'(0),v)=\angles{\theta_1}{v}+\bangles{\tiltheta_1}{\tr\,v},\qquad \forall v\in H^1(B).
	\end{split}
	\end{align}
	Note that since we seek $\Theta\in L^2([0,T];H^1(B))$, $\Theta'\in L^2([0,T];L^2(B))$ we must in fact have $\Theta\in C([0,T];L^2(B))$ (after possible modification on a set of zero measure), so the initial value $\Theta(0)$ makes sense in $L^2(B)$. By a similar reasoning we can make sense of $\Phi(\Theta)'(0)$.\footnote{In fact, with a little more work one can argue that $\Theta'\in C([0,T];L^2(B))$ and $(\tr\,\Theta)'\in C([0,T];L^2(\partial B))$; see for instance \cite{Lions-Magenes-book1, Wloka-book}. This will not be needed to prove energy estimates.} 
\begin{remark}\label{rem:strongeqn}
Suppose $\Theta$ satisfies \eqref{eq:weak1} and that it is sufficiently regular (say $C^3$). Then taking $v=\fy\in C^\infty_0(B)$, we can integrate by parts back in the definition of $\calL$ to conclude from \eqref{eq:weak1} that
\begin{align*}
\begin{split}
\int_B (\Box_g\Theta) \fy\ud x=0 \qquad \forall \fy\in C^\infty_0(B).
\end{split}
\end{align*}
It follows that $\Box_g\Theta\equiv0$ in $B$. Then using this fact and the surjectivity of $\tr:H^1(B)\to L^2(\partial B)$, we can take $v=\fy\in C^\infty(\overline{B})$ arbitrarily, and integrate by parts again in the definition of $\calL$ in \eqref{eq:weak1} to conclude that
\begin{align*}
\begin{split}
\partial_{t}^{2}\Theta+\gamma \nabla_{n}\Theta=f.
\end{split}
\end{align*}
A similar conclusion holds for equation~\eqref{eq:weakLambda1} below if $\Lambda$ is sufficiently regular.
\end{remark}
	We also need to consider the equations obtained by commuting several $\partial_{t}$ derivatives. To be more systematic, let us denote by $F_{0}=0$ and $f_{0}=\gamma^{-1}f$ the right-hand sides of the interior and boundary equations respectively. Similarly, we use $F_{k}$ and $f_{k}$ to denote the right-hand sides of the interior and boundary equations obtained by commuting $\partial_{t}$ derivatives $k$ times, and let $\Theta_k$ be the $k$-times differentiated unknown. Now we compute the commutator in the weak form. Assuming for the moment that all functions are sufficiently regular,
	\begin{align*}
		\left((\Phi(\Theta)'',v)+\calL(\Theta,v)\right)^{\prime}=&(\Phi(\Theta')'',v)+\calL(\Theta',v)\\&+\angles{\partial_{t}g^{ab}\partial_{a}\Theta}{\partial_{b}v}-\frac{1}{2}\angles{\partial_a\Theta}{ v\partial_{t}(g^{a\alpha}\partial_\alpha\log |g|)}-\angles{\partial_a\Theta}{v\partial^{2}_{t}g^{ta}}\\
		&+2\angles{\Theta'}{\left(\partial_{t}g^{ta}\right)\partial_{a}v}-\frac12\angles{\Theta'}{v\partial_{t}(g^{t\alpha}\partial_{\alpha}\log|g|)}+\angles{\Theta'}{v\partial_{t}\partial_{a}g^{ta}}\\
		&-\bangles{\Theta'}{(\partial_{t}g^{tr})\tr\,v}+\bangles{\Theta''}{(\partial_{t}(\gamma^{-1}))\tr\,v}.
	\end{align*}
	If we define
	\begin{align}\label{def:commutator weak}
		\begin{split}
		\angles{\calC(\Theta)}{v}&:=-\frac{1}{2}\angles{\partial_a\Theta}{ v\partial_{t}(g^{a\alpha}\partial_\alpha\log |g|)}-\angles{\partial_a\Theta}{v\partial^{2}_{t}g^{ta}}\\&\quad-\frac12\angles{\Theta'}{v\partial_{t}(g^{t\alpha}\partial_{\alpha}\log|g|)}+\angles{\Theta'}{v\partial_{t}\partial_{a}g^{ta}},\\
		\angles{\calC^{a}(\Theta)}{\partial_{a}v}&:=\angles{\partial_{t}g^{ab}\partial_{b}\Theta}{\partial_{a}v}+2\angles{\Theta'}{\left(\partial_{t}g^{ta}\right)\partial_{a}v},\\
		\bangles{\calC_{\calB}(\Theta)}{\tr\,v}&:=-\bangles{(\tr\,\Theta)'}{(\partial_{t}g^{tr})\tr\,v},\\
		\bangles{\tilcalC_{\calB}(\Theta)}{\tr\,v}&:=-\bangles{(\tr\,\Theta)'}{(\gamma^{-1})'\tr\,v},\\
		\bangles{\tiltilcalC_\calB(\Theta)}{v}&:=\bangles{(\tr\,\Theta)'}{(\gamma^{-1})''\tr\,v},
		\end{split}
	\end{align}
	then we have the following formula for the weak form of the commutator:
	\begin{align}\label{commutator weak}
		\begin{split}
		&\frac{\ud}{\ud t}\left((\Phi(\Theta)'',v)+\calL(\Theta,v)\right)-(\Phi(\Theta')'',v)-\calL(\Theta',v)-\bangles{\tilcalC_{\calB}(\Theta')}{\tr\,v}\\
		&=\angles{\calC(\Theta)}{v}+\angles{\calC^{a}(\Theta)}{\partial_{a}v}+\bangles{\calC_{\calB}(\Theta)}{\tr\,v}.
		\end{split}
	\end{align}
Therefore, $\Theta_1:=\partial_t\Theta$ satisfies the weak equation (again assuming sufficient regularity)
\begin{align*}
\begin{split}
(\Phi(\Theta_1)'',v)+\calL(\Theta_1,v)+\bangles{\tilcalC_{\calB}(\Theta_1)}{v}=\angles{F_{1}}{v}+\angles{\calF^{a}_{1}}{\partial_{a}v}+\bangles{f_{1}}{\tr\,v},
\end{split}
\end{align*}
where in terms of $\Theta_0:=\Theta$
\begin{align*}
\begin{split}
F_{1}=\partial_{t}F_{0}-\calC(\Theta_{0}),\quad
		\calF^{a}_{1}=\partial_{t}\calF^{a}_{0}-\calC^{a}(\Theta_{0}),\quad
		f_{1}=\partial_{t}f_{0}-\calC_{\calB}(\Theta_{0}).
\end{split}
\end{align*}
For the next derivative, note that with $\Theta_2:=\partial_t^2\Theta$ and $\tiltilcalC_\calB$ as in \eqref{def:commutator weak}
\begin{align*}
\begin{split}
\frac{\ud}{\ud t}\bangles{\tilcalC_{\calB}(\Theta_1)}{v}=\bangles{\tilcalC_{\calB}(\Theta_2)}{v}-\bangles{\tiltilcalC_\calB(\Theta_1)}{v},
\end{split}
\end{align*}
so
\begin{align*}
\begin{split}
(\Phi(\Theta_2)'',v)+\calL(\Theta_2,v)+2\bangles{\tilcalC_{\calB}(\Theta_2)}{v}=\angles{F_{2}}{v}+\angles{\calF^{a}_{2}}{\partial_{a}v}+\bangles{f_{2}}{\tr\,v},
\end{split}
\end{align*}
where
\begin{align*}
\begin{split}
F_{2}=\partial_{t}F_{1}-\calC(\Theta_{1}),\quad
		\calF^{a}_{2}=\partial_{t}\calF^{a}_{1}-\calC^{a}(\Theta_{1}),\quad
		f_{2}=\partial_{t}f_{1}-\calC_{\calB}(\Theta_{1})-\tiltilcalC_{\calB}(\Theta_1).
\end{split}
\end{align*}
Continuing in this way we see that $\Theta_k:=\partial_t^k\Theta$ satisfies the weak equation 
	\begin{align}\label{eq:weak2}
	\begin{split}
		(\Phi(\Theta_{k})'',v)+\calL(\Theta_{k},v)+k\bangles{\tilcalC_\calB(\Theta_k)}{v}=\angles{F_{k}}{v}+\angles{\calF^{a}_{k}}{\partial_{a}v}+\bangles{f_{k}}{\tr\,v},\qquad \forall v\in H^1(B),
		\end{split}
	\end{align}
where the source terms are defined by the following recursive formulas:
	\begin{align}\label{inhomogeneous source recursive}
		\begin{split}
		&F_{0}=0,\quad \calF^{a}_{0}=0,\quad f_{0}=\gamma^{-1}f,\\
		&F_{k}=\partial_{t}F_{k-1}-\calC(\Theta_{k-1}),\quad
		\calF^{a}_{k}=\partial_{t}\calF^{a}_{k-1}-\calC^{a}(\Theta_{k-1}),\quad
		f_{k}=\partial_{t}f_{k-1}-\calC_{\calB}(\Theta_{k-1})-(k-1)\tiltilcalC_{\calB}(\Theta_{k-1}).
		\end{split}
	\end{align}
Using an induction argument, one can derive the following explicit formulas valid for $k\geq 1$:
	\begin{align}\label{higher order source}
	\begin{split}
		F_{k}=&\partial_{t}^{k}F_{0}-\sum_{\ell=0}^{k-1}\partial_{t}^{\ell}\calC(\Theta_{k-\ell-1})=-\sum_{\ell=0}^{k-1}\partial_{t}^{\ell}\calC(\Theta_{k-\ell-1}),\\
		\calF^{a}_{k}=&\partial_{t}^{k}\calF_{0}^{a}-\sum_{\ell=0}^{k-1}\partial_{t}^{\ell}\calC^{a}(\Theta_{k-\ell-1})=-\sum_{\ell=0}^{k-1}\partial_{t}^{\ell}\calC^{a}(\Theta_{k-\ell-1}),\\
		f_{k}=&\partial_{t}^{k}f_{0}-\sum_{\ell=0}^{k-1}\partial_{t}^{\ell}\calC_{\calB}(\Theta_{k-\ell-1})-\sum_{\ell=0}^{k-2}(k-\ell-1)\partial_t^\ell\tiltilcalC_\calB(\Theta_{k-\ell-1}).
		\end{split}
	\end{align}
To rigorously justify the derivations above we would of course need $\Theta$ to be sufficiently regular, which we cannot a priori assume. Instead, in the next subsection we will inductively \emph{define} $\Theta_k$ to be the solution of \eqref{eq:weak2} (with appropriate initial data) and show that $\Theta_k=\Theta_{k-1}'$, proving regularity of $\Theta$.

\subsubsection{The equation for $D_V\sigma^2$}
The derivation of the weak formulation for the linearized equation for $D_V\sigma^2$ is similar to that of $V$, but simpler because the boundary condition is the standard zero Dirichlet condition. Recall from \eqref{eq:sigma1a-intro} that this equation is
\begin{align*}
\begin{split}
\Box D_V\sigma^2=4(\nabla^\mu V^\nu)\nabla_\mu\nabla_\nu\sigma^2+4(\nabla^\lambda V^\nu)(\nabla_\lambda V^\mu)(\nabla_\nu V_\mu),\qquad D_V\sigma^2\equiv0\mathrm{~on~}\partial\Omega.
\end{split}
\end{align*}
We will use $\Lambda$ to denote the linearized unknown for the equation of $D_V\sigma^2$. Going through similar computations as for $\Theta$ above, we can derive the weak equation for $\Lambda$. Let us restrict the domains of the bilinear forms $B$ and $C$:
\begin{align*}
\begin{split}
B:H^1_0(B)\times H^1_0(B)\to \bbR\qquad \mathrm{and} \qquad C:L^2(B)\times H^1_0(B)\to \bbR,
\end{split}
\end{align*}
To write the weak equation for $\Lambda$ we use the standard embedding of $H^1_0(B)$ in $H^{-1}(B):=(H^1_0(B))^\ast$ given by
\begin{align*}
\begin{split}
(\iota(u),v):=\angles{u}{v},\qquad u\in H^1_0(B),\quad v\in H^{1}_0(B).
\end{split}
\end{align*}
By a slight abuse of notation we will simply write $u$ for $\iota(u)$ from now on. For any $\Lambda$ with
\begin{align}\label{eq:Lambdaspaces1}
\begin{split}
&\Lambda\in L^2([0,T];H^1_0(B)),\quad \Lambda'\in L^2([0,T],L^2(B)),\quad\Lambda''\in L^2([0,T];H^{-1}(B)),
\end{split}
\end{align}
and $v\in H^1_0(B)$ define
\begin{align*}
\begin{split}
\calL_\sigma(\Lambda,v):=B(\Lambda,v)+C(\Lambda',v).
\end{split}
\end{align*}
The linearized weak equation for $\Lambda$ then takes the form
\begin{align}\label{eq:weakLambda1}
\begin{split}
(\Lambda'',v)+\calL_\sigma(\Lambda,v)=\angles{F_\sigma}{v},\qquad \forall v\in H^1_0(B),
\end{split}
\end{align}
where $F_\sigma\in L^2([0,T];L^2(B))$ is a given function. Also given initial data $\lambda_0\in H^1_0(B)$ and $\lambda_1\in L^2(B)$ the initial conditions are
\begin{align}\label{eq:Lambdaid}
\begin{split}
\Lambda(0)=\lambda_0,\qquad \quad ((\Lambda)'(0),v)=\angles{\lambda_1}{v},\quad\forall v\in H^1_0(B).
\end{split}
\end{align}
The discussion of the higher order equations for $\Lambda^k:=\partial_t^k\Lambda$ is also similar to the case of $\Theta$ but simpler.  Let
\begin{align}\label{eq:Lambdalinsource}
\begin{split}
F_{\sigma,0}&:=F_0,\qquad \calF^{a}_{\sigma,0}:=0,\\
F_{\sigma,k}=&\partial_{t}^{k}F_{\sigma,0}-\sum_{\ell=0}^{k-1}\partial_{t}^{\ell}\calC(\Lambda_{k-\ell-1})=-\sum_{\ell=0}^{k-1}\partial_{t}^{\ell}\calC(\Lambda_{k-\ell-1}),\\
\calF^{a}_{\sigma,k}=&\partial_{t}^{k}\calF_{\sigma,0}^{a}-\sum_{\ell=0}^{k-1}\partial_{t}^{\ell}\calC^{a}(\Lambda_{k-\ell-1})=-\sum_{\ell=0}^{k-1}\partial_{t}^{\ell}\calC^{a}(\Lambda_{k-\ell-1}),\\
\end{split}
\end{align}
The higher order equations for $\Lambda_k$ are then
\begin{align}\label{eq:Lambdalinweak2}
\begin{split}
(\Lambda_k'',v)+\calL_\sigma(\Lambda_{k},v)=\angles{F_{\sigma,k}}{v}+\angles{\calF^{a}_{\sigma,k}}{\partial_{a}v},\qquad \forall v\in H^1_0(B).
\end{split}
\end{align}

\subsection{Existence and uniqueness of the weak solution}\label{sec:weakex}
This subsection contains the existence theory and energy estimates which are the main ingredient of the nonlinear iteration. Let $K$ be a fixed large integer representing the total number of derivatives we commute. Throughout this section we assume that $g$, $\gamma$, $f_0$, $F_0$, $\calF_0^a$, $F_{\sigma,0}$, $\calF_{\sigma,0}^a$ are given functions satisfying the following conditions:
\begin{align}\label{eq:g-assumption1}
\begin{split}
\partial_{\bary}^a\partial_t^k g\in L^\infty([0,T];L^2(B)),\qquad &k\leq K+1,~ \begin{cases}2a\leq K+1-k,\quad &k\geq1\\ 2a\leq K,\quad &k=0\end{cases},\\
\partial_t^k (\tr\,g)\in L^2([0,T];L^2(\partial B)),\qquad &k\leq K,\\
\partial_t^k \gamma^{-1},\partial_t^k\gamma\in L^2([0,T];L^2(\partial B)),\qquad &k\leq K,\\
\partial_t^k\gamma^{-1},\partial_t^k\gamma\in L^\infty([0,T];L^\infty(\partial B)),\qquad &k\leq K-5,\\
\partial_t^kf_{0}\in L^2([0,T];L^2(\partial B)),\qquad&k\leq K,\\
\partial_t^k F_{\sigma,0}\in L^{2}([0,T];L^2( B)),\qquad&k\leq K.\\
\end{split}
\end{align}
We will also use the notation introduced in \eqref{def:commutator weak}, \eqref{higher order source}, and\eqref{eq:Lambdalinsource}.

We start with the more difficult case of equation \eqref{eq:weak2} for $\Theta$. The treatment of \eqref{eq:weak2} is divided into two parts. First, in Proposition~\ref{prop:weak1} we prove existence and uniqueness for the linear systems under quite general assumptions, and prove higher regularity of the solution up to order $K-5$. Then in Proposition~\ref{prop:weak2} we prove higher regularity up to order $K$. The reason for this distinction is that for the first $K-5$ derivatives we can always bound the coefficients appearing in the commutator errors in $L^\infty$, whereas for the last five derivatives we sometimes need to bound these coefficients in $L^2$. This requires estimating the lower order derivatives of the solution in $L^\infty$, which in turn calls for Sobolev estimates in terms of the energies. The main step in going from Proposition~\ref{prop:weak1} to Proposition~\ref{prop:weak2} is proving these Sobolev estimates.

\begin{proposition}\label{prop:weak1}
Suppose \eqref{eq:g-assumption1} holds and that there exist
\begin{align*}
\begin{split}
\theta_k\in H^1(B),\quad \theta_{k+1}\in L^2(B), \quad \tiltheta_{k+1}\in L^2(\partial B), \quad k=0,...,K-6
\end{split}
\end{align*}
such that the following two conditions hold: 
\begin{itemize}
\item  
For $k=0,\dots,K-7$
\begin{align*}
\begin{split}
\angles{\theta_{k+2}}{v}+\bangles{\tiltheta_{k+2}}{\gamma^{-1}\tr\,v}+\calL(\theta_{k},v)+k\bangles{\tilcalC_\calB(\theta_k)}{v}
=\bangles{f_{k}(0)}{\tr\,v}+\angles{F_{k}(0)}{v}+\angles{\calF^{a}_{k}(0)}{\partial_av}.
\end{split}
\end{align*}
Here $f_k$, $F_k$, $\calF^{a}_{k}$ are given by the formulas in \eqref{higher order source}, where the initial values of $\Theta_k$ and $\partial_t\Theta_k$ are defined to be $\theta_k$ and $\theta_{k+1}$ respectively (see Remark~\ref{rem:iddef}). 
\item 
$\tiltheta_k=\tr\,\theta_k$ for $k=1,\dots, K-6$.
\end{itemize}
Then there exists a unique $\Theta_{k}$ satisfying \eqref{eq:Thetaspaces1} and \eqref{eq:weakdata1} (with $(\theta_0,\theta_1,\tiltheta_1)$ replaced by $(\theta_k,\theta_{k+1},\tiltheta_{k+1})$), such that for all $v\in H^1(B)$ equation \eqref{eq:weak2} holds for almost every $t\in[0,T]$. The solution satisfies
\begin{align}\label{eq:energy}
\begin{split}
&\sup_{t\in [0,T]}\big(\|\Theta'_{k}\|_{L^2(B)}+\|\Theta_{k}\|_{H^1(B)}+\|\tr\,\Theta'_{k}\|_{L^2(\partial B)}\big)\\
&\leq C_1e^{C_2T}\Big(\|\theta_k\|_{H^1(B)}+\|\theta_{k+1}\|_{L^2(B)}+\|\tiltheta_{k+1}\|_{L^2(\partial B)}+\|f_{k}\|_{L^2([0,T];L^2(\partial B))}\\
&\phantom{\leq C_1e^{C_2T}\Big(}+\|F_{k}\|_{L^2([0,T];L^2(B))}+\|\calF^{a}_{k}\|_{L^{\infty}([0,T];L^2(B))}\Big),
\end{split}
\end{align}
In these estimates $C_1$, $C_2$, and $C_3$ are constants depending only on $\|g\|_{L^\infty([0,T]\times B)}$, $\|\partial_{t,\bary} g\|_{L^\infty([0,T]\times B)}$, $\|g^{-1}\|_{L^\infty([0,T]\times B)}$, $\|\partial_{t,\bary} g^{-1}\|_{L^\infty([0,T]\times B)}$, $\|\gamma^{-1}\|_{L^\infty([0,T]\times \partial B)}$, and $\|\partial_t\gamma^{-1}\|_{L^\infty([0,T]\times \partial B)}$. Moreover, we have $\Theta'_{k-1}=\Theta_{k}$ for $k=1,...,k-5$.
\end{proposition}
Before presenting the proof let us make a few remarks about the assumptions of the proposition.
\begin{remark}
Note that in assuming existence of $\theta_2$, $\theta_3$, etc, we are not imposing additional data. Rather, we are requiring additional regularity on $\theta_0$, $\theta_1$,  and $\tiltheta_1$. Indeed, if $\theta_2$, $\theta_3$, etc, exist they are uniquely determined by $\theta_0$, $\theta_1$,  and $\tiltheta_0$, so there is no freedom in prescribing them. To see this note that for instance in the regularity condition for $\theta_2,\tiltheta_2$ we can first take $v\in H^1_0(B)$ to get uniqueness of $\theta_2$ and then use the surjectivity of $\tr:H^1(B)\to L^2(\partial B)$ to get uniqueness of $\tiltheta_2$. Similarly, $\theta_3$ and $\tiltheta_3$ are determined by $\theta_2$.

The compatibility conditions are there to guarantee that the initial and boundary conditions match on the initial boundary, as required for wave equations on bounded domains.
\end{remark}
\begin{remark}\label{rem:iddef}
We clarify the meaning of the initial value of $f_k$, $F_k$, and $\tilF_k^a$. By definition, these are simply given by replacing $\Theta_k$ and $\partial_t\Theta_k$ by their initial data $\theta_k$ and $\theta_{k+1}$. For instance, the regularity for $(\theta_2,\tiltheta_2)$ is (see \eqref{eq:defbilin}, \eqref{eq:defcalL}, and \eqref{def:commutator weak})
\begin{align*}
\begin{split}
&\angles{\theta_2}{v}+\bangles{\tiltheta_2}{\gamma^{-1}\tr\,v}+B(\theta_0,v)+C(\theta_1,v)+D(\tiltheta_1,v)+E(\tr\,\theta_0,v)-\bangles{\tiltheta_1}{(\partial_t\gamma^{-1})\tr\,v}\\
&=\bangles{f_{0}(0)}{\tr\,v}+\angles{F_{0}(0)}{v}+\angles{\calF^{a}_{0}(0)}{\partial_av}.
\end{split}
\end{align*}
For higher values of $k$ we also replace $\Theta_\ell$ and $\partial_t\Theta_\ell$ appearing in \eqref{higher order source} by $\theta_\ell$, $\theta_{\ell+1}$, and $\tiltheta_{\ell+1}$ in the same way as above.
\end{remark}
\begin{proof}[Proof of Proposition~\ref{prop:weak1}]
{\bf{Existence.}} We proceed inductively. For $k=0$, we have $F_{0}=0$ and $\calF^{a}_{0}=0$. Once $\Theta_{k-1}$ is constructed, then $F_{k}$ and $\calF^{a}_{k}$ is defined as in \eqref{higher order source}, and satisfy
\begin{align*}
	\|F_{k}\|_{L^{2}([0,T]\times B)}<\infty,\quad\textrm{and}\quad \|\calF^{a}_{k}\|_{L^{\infty}([0,T];L^{2}(B))}<\infty.
\end{align*}
Let $\{e_\ell\}_{\ell=1}^\infty$ be an orthogonal basis of $H^1(B)$ which at the same time is an orthonormal basis of $L^2(B)$ (see for instance \cite{Jost-PDEbook}). All inner products are with respect to the time-independent measure $\ud\bary$. The linear span of $\{e_1,\dots,e_m\}$ is denoted by $E_m$ and the $L^2(B)$ orthogonal projection onto $E_m$ by $P_m$. Note that $P_m$ is also the $H^1(B)$ orthogonal projection. Indeed, if $v=\sum_{\ell=1}^\infty v^\ell e_\ell$ is an arbitrary element of $H^1(B)$, then both the $L^2(B)$ and $H^1(B)$ projections of $v$ on $E_m$ are given by $\sum_{\ell=1}^mv^\ell e_\ell$.  We define the map $\Phi_m:H^1(B)\to (H^1(B))^\ast$ by
\begin{align*}
\begin{split}
(\Phi_m(u),v):=(\Phi(u),P_mv)=\angles{u}{P_mv}+\bangles{\tr \,u}{\gamma^{-1}\tr\,P_mv},
\end{split}
\end{align*}
and say that
\begin{align*}
\begin{split}
\Theta_{k,m}(t,x):=\sum_{\ell=1}^m \Theta_{k,m}^\ell(t)e_\ell(x),\qquad \Theta_{k,m}^\ell\in C^2([0,T]),~\ell=1,\dots,m,
\end{split}
\end{align*}
satisfies the $\mth$ approximate weak equation if for $\ell=1,\dots,m$,
\begin{align}\label{eq:gw1}
\begin{split}
&(\Phi_m(\Theta_{k,m})'',e_\ell)+\calL(\Theta_{k,m},e_\ell)+k\bangles{\tilcalC_\calB(\Theta_{k,m})}{v}=\bangles{f_{k}}{\tr\,e_\ell}+\angles{F_{k}}{e_\ell}+\angles{\calF^{a}_{k}}{\partial_ae_\ell},
\end{split}
\end{align}
and
\begin{align}\label{eq:gdata1}
\begin{split}
\Theta_{k,m}^\ell(0)=\angles{\theta_k}{e_\ell},\qquad &\ell=1,\dots,m,\\
\angles{\Theta_{k,m}'(0)}{e_\ell}+\bangles{\tr\, \Theta_{k,m}'(0)}{e_\ell}=\angles{\theta_{k+1}}{e_\ell}+\bangles{\tiltheta_{k+1}}{\tr\,e_\ell},\qquad &\ell=1,\dots,m.
\end{split}
\end{align}
Existence and uniqueness of $\Theta_{k,m}$ satisfying the $\mth$ approximate weak equation is a consequence of existence theory for ODEs. Indeed, the equation reduces to a system of linear second order ODEs for the unknowns $\Theta_m^1,\dots,\Theta_m^m$. The matrix coefficient of the second order derivative is
\begin{align*}
\begin{split}
I^m+G^m
\end{split}
\end{align*}
where $G^m$ is the positive semi-definite matrix\footnote{To see that $G^m$ is positive semi-definite, let $g^m(x)$ be the $m\times m$ matrix with entries $e_i(x)e_j(x)$. For each $x$, $g^m(x)$ has $m-1$ zero eigenvectors (take $m-1$ linearly independent vectors which are perpendicular to $\vece(x):=(e_1(x),\dots ,e_m(x))^\intercal$) and one positive eigenvector (namely $\vece(x)$) with eigenvalue $\|\vece(x)\|_{\bbR^m}^2$. In particular $g^m(x)$ is positive semi-definite for each $x$. Then note that for each constant vector $X_0\in\bbR^m$
\begin{align*}
\begin{split}
X_0^\intercal G^m X_0=\int_{\partial B} X_0^\intercal g^m(x) X_0 \gamma^{-1}\ud S(x)\geq0.
\end{split}
\end{align*}}
with entries $$G^m_{ij}=\bangles{\gamma^{-1}\tr\,e_i}{\tr\, e_j}.$$ It follows that $I^m+G^m$ is invertible and the ODE can be put in standard form. Similarly, for the initial data for the time derivatives $(\Theta_m^k)'(0)$ we use the invertibility of $I^m+\tilG^m$ and the second equation in \eqref{eq:gdata1}, where $\tilG^m$ is the positive semi-definite matrix with entries $\tilG^m_{ij}=\bangles{\tr\,e_i}{\tr\, e_j}$. The initial data for $\Theta_m^k(0)$ are given by the first equation in \eqref{eq:gdata1}.
	
Our next goal is to prove energy estimates for the $\mth$ approximate weak equation. For this we multiply \eqref{eq:gw1} by $(\Theta_{k,m}^{\ell})'$ and sum up in $\ell=1,\dots,m$ to get, with $|\nabla_gu|^2:=g^{ab}\partial_au\partial_bu$,
	\begin{align}\label{eq:ge1}
	\begin{split}
	&\int_{0}^{t}\left(\frac{1}{2}\frac{\ud}{\ud t}\big(\|\Theta_{k,m}'\|_{L^2(B)}^2+\||\nabla_g\Theta_{k,m}|\|_{L^2(B)}^2+\|\gamma^{-\frac{1}{2}}\tr\,\Theta_{k,m}'\|_{L^2(\partial B)}^2\big)\right)ds\\
	&=\int_{0}^{t}\left(\frac{1}{2}\angles{\Theta_{k,m}'}{\Theta_{k,m}' (g^{t\alpha}\partial_\alpha\log|g|)}+\frac{1}{2}\angles{\partial_a\Theta_{k,m}}{\Theta_{k,m}' (g^{a\alpha}\partial_\alpha\log|g|)}\right)ds\\
	&\quad+\int_{0}^{t}\left(\angles{\Theta_{k,m}}{\Theta_{k,m}' \partial_ag^{ta}}+\frac{1}{2}\angles{\partial_tg^{ab} \partial_a\Theta_{k,m}}{\partial_b\Theta_{k,m}}\right)ds\\
	&\quad+\int_{0}^{t}\left(\frac{2k+1}{2}\bangles{(\partial_t\gamma^{-1})\Theta_{k,m}'}{\Theta'_{k,m}}+\bangles{f^{h}}{\tr\,\Theta_{k,m}'}+\angles{F_{k}}{\Theta_{k,m}'}+\angles{\calF^{a}_{k}}{\partial_a\Theta_{k,m}'}\right)ds.
	\end{split}
	\end{align}
	To obtain an energy estimate, we need to bound the initial norms $\||\nabla_g\Theta_{k,m}(0)|\|_{L^2(B)}$, $\|\Theta_{k,m}'(0)\|_{L^2(B)}$, and $\|\tr\,\Theta_{k,m}'(0)\|_{L^2(\partial B)}$. The first term, $\||\nabla_g\Theta_{k,m}(0)|\|_{L^2(B)}$, can be bounded by the $H^1(B)$ norm of $\theta_0$. For the time derivatives we multiply the second equation in \eqref{eq:gdata1} by $(\Theta_{k,m}^\ell)'(0)$ and sum up in $\ell=1,\dots,m$ to get
	\begin{align*}
	\begin{split}
	\|\Theta_{k,m}'(0)\|_{L^2(B)}^2+\|\tr\,\Theta_{k,m}'(0)\|_{L^2(\partial B)}^2\leq \|\theta_{k+1}\|_{L^2(B)}\|\Theta_{k,m}'(0)\|_{L^2(B)}+\|\tiltheta_{k+1}\|_{L^2(\partial B)}\|\tr\,\Theta_{k,m}'(0)\|_{L^2(\partial B)},
	\end{split}
	\end{align*}
	from which it follows that
	\begin{align*}
	\begin{split}
	\|\Theta_{k,m}'(0)\|_{L^2(B)}^2+\|\tr\,\Theta_{k,m}'(0)\|_{L^2(\partial B)}^2\leq \|\theta_{k+1}\|_{L^2(B)}^2+\|\tiltheta_{k+1}\|_{L^2(\partial B)}^2.
	\end{split}
	\end{align*}
	To obtain the uniform bounds on $\Theta_{k,m}(t)$, we use \eqref{eq:ge1} to get
	\begin{align*}
	\begin{split}
	&\sup_{t\in [0,T]}\big(\|\Theta_{k,m}'\|_{L^2(B)}^2+\||\nabla_g\Theta_{k,m}|\|_{L^2(B)}^2+\|\tr\,\Theta_{k,m}'\|_{L^2(\partial B)}^2\big)\\
	&\leq \|\theta_k\|_{H^1(B)}^2+\|\theta_{k+1}\|_{L^2(B)}^2+\|\tiltheta_{k+1}\|_{L^2(\partial B)}^2\\
	&+\int_0^T(\|f_{k}\|_{L^2(\partial B)}\|\tr\,\Theta_{k,m}'\|_{L^2(\partial B)}\,\ud t+\frac{2k+1}{2}\int_0^T\|\partial_t\gamma^{-1}\|_{L^\infty(\partial B)}\|\Theta_{k,m}'\|_{L^2(\partial B)}^2\ud t\\
	&+\int_0^T\|F_{k}\|_{L^2(B)}\|\Theta_{k,m}'\|_{L^2(B)}\,\ud t+\left|\int_{0}^{T}\int_{B}\calF^{a}_{k}\partial_{a}\Theta'_{k,m}\,d\bary\,\ud t\right|\\
	&+\int_0^T(\|g\|_{L^\infty(B)}\|\partial_{t,\bary}\log |g|\|_{L^\infty(B)}+\|\partial_{t,\bary}g\|_{L^\infty(B)})(\|\Theta_{k,m}'\|_{L^2(B)}^2+\|\partial_\bary \Theta_{k,m}\|_{L^2(B)}^2)\,\ud t.
	\end{split}
	\end{align*}
	The term $\int_{0}^{T}\int_{B}\calF^{a}_{k}\partial_{a}\Theta'_{k,m}\,d\bary\,\ud t$ needs some special care:
	\begin{align*}
	\int_{0}^{T}\int_{B}\calF^{a}_{k}\partial_{a}\Theta'_{k,m}\,d\bary\,\ud t=&-\int_{0}^{T}\int_{B}\partial_{t}\calF^{a}_{k}\partial_{a}\Theta_{k,m}d\bary\,\ud t+\int_{B}\calF^{a}_{k}(T)\partial_{a}\Theta_{k,m}(T)d\bary-\int_{B}\calF^{a}_{k}(0)\partial_{a}\Theta_{k,m}(0)d\bary,
	\end{align*}
	which gives
	\begin{align*}
	\left|	\int_{0}^{T}\int_{B}\calF^{a}_{k}\partial_{a}\Theta'_{k,m}\,d\bary\,\ud t\right|\lesssim& \int_{0}^{T}\|\partial_{t}\calF^{a}_{k}(t,\cdot)\|^{2}_{L^{2}(B)}dt+\int_{0}^{T}\|\nabla_{g}\Theta_{k,m}(t,\cdot)\|^{2}_{L^{2}(B)}dt\\
	&+\delta\sup_{t\in[0,T]}\|\nabla_{g}\Theta_{k,m}(t,\cdot)\|^{2}_{L^{2}(B)}+C_{\delta}\sup_{t\in[0,T]}\|\calF^{a}_{k}(t,\cdot)\|_{L^{2}(B)}^{2},
	\end{align*}
	where $\delta>0$ is sufficiently small. Then it follows from Gronwall that
	\begin{align}\label{eq:gen1}
	\begin{split}
	&\sup_{t\in [0,T]}\big(\|\Theta_{k,m}'\|_{L^2(B)}+\|\Theta_{k,m}\|_{H^1(B)}+\|\tr\,\Theta_{k,m}'\|_{L^2(\partial B)}\big)\\
	&\leq C_1e^{C_2T}\Big(\|\theta_k\|_{H^1(B)}+\|\theta_{k+1}\|_{L^2(B)}+\|\tiltheta_{k+1}\|_{L^2(\partial B)}+\|f_{k}\|_{L^2([0,T];L^2(\partial B))}\\
	&\phantom{\leq C_1e^{C_2T}\Big(}+\|F_{k}\|_{L^2([0,T];L^2(B))}+\|\calF^{a}_{k}\|_{L^{\infty}([0,T];L^2(B))}\Big).
	\end{split}
	\end{align}
	Here the constants $C_1$ and $C_2$ depend only on
	\begin{align*}
	\begin{split}
	&\|g\|_{L^\infty([0,T]\times B)}, \|\partial_{t,\bary} g\|_{L^\infty([0,T]\times B)},~ \|g^{-1}\|_{L^\infty([0,T]\times B)}, \|\partial_{t,\bary} g^{-1}\|_{L^\infty([0,T]\times B)},\\
	&\|\gamma^{-1}\|_{L^\infty([0,T]\times \partial B)}, ~\|\partial_t\gamma^{-1}\|_{L^\infty([0,T]\times \partial B)}.
	\end{split}
	\end{align*}
	Estimate \eqref{eq:gen1} is the main energy estimate on $\Theta_{k,m}$ which allows us to pass to a limit. To get a bound on $\Phi_m(\Theta_{k,m})''$ note that by \eqref{eq:gw1} and the definition of $\Phi_{k,m}$, for any $v\in H^1(B)$
	\begin{align*}
	\begin{split}
	(\Phi_m(\Theta_{k,m})'',v)&=(\Phi_m(\Theta_{k,m})'',P_mv)\\
	&=\bangles{f_{k}}{\tr\,P_mv}+\angles{F_{k}}{P_m v}-\angles{\calF^{a}_{k}}{\partial_a P_mv}-\calL(\Theta_{k,m},P_m v).
	\end{split}
	\end{align*}
	Appealing to the bounds \eqref{eq:gen1} we get\footnote{Since $\tr \,\Theta_{k,m}$ and $\tr \,\Theta_{k,m}'$ are well defined, uniform $(H^1(B))^\ast$ bounds on $\Phi_m(\Theta_{k,m})$ and $\Phi_m(\Theta_{k,m})'$ also follow from the bounds on $\Theta_{k,m}$, $\Theta_{k,m}'$, and $\tr\,\Theta_{k,m}'$.}
	\begin{align*}
	\begin{split}
	\|\Phi_m(\Theta_{k,m})''\|_{L^2([0,T];(H^1(B))^\ast)}\leq C,
	\end{split}
	\end{align*}
	where $C$ depends on the upper bounds on the right-hand side of \eqref{eq:gen1}. This means that $\Theta_{k,m}\in L^2([0,T];H^1(B))$, $\Theta_{k,m}'\in L^2([0,T];L^2(B))$, $(\tr \,\Theta_{k,m})'\in L^2([0,T];L^2(\partial B))$, and $\Phi(\Theta_{k,m})''\in L^2([0,T];(H^1(B))^\ast)$, and they form bounded sequences, and hence have weak limits (along a subsequence) in the same spaces. Let $\Theta_k$ denote the limit of $\Theta_{k,m}$ and $U$ the limit of $\Phi_m(\Theta_{k,m})$. We claim that $U=\Phi(\Theta_{k})$ and that $\Theta_{k}$ satisfies \eqref{eq:weak2}. For the former claim suppose $v:=\sum_{\ell=1}^M v^\ell e_\ell$ is element in $H^1(B)$ belonging to the span of $\{e_1,\dots, e_M\}$ for some $M$. Then
	\begin{align*}
	\begin{split}
	(\Phi(\Theta_{k}),v)&=\angles{\Theta_{k}}{v}+\bangles{\tr\,\Theta_{k}}{\gamma^{-1}\tr\,v}=\lim_{m\to\infty}\angles{\Theta_{k,m}}{v}+\lim_{m\to\infty}\bangles{\tr\,\Theta_{k,m}}{\gamma^{-1}\tr\,v}\\
	&=\lim_{m\to\infty}\angles{\Theta_{k,m}}{P_m \,v}+\lim_{m\to\infty}\bangles{\tr\,\Theta_{k,m}}{\gamma^{-1}\tr\,P_m \,v}=\lim_{m\to\infty}(\Phi_m(\Theta_{k,m}),v)=:(U,v).
	\end{split}
	\end{align*}
	Since elements $v$ of this form are dense in $H^1(B)$, it follows that $U$ and $\Phi(\Theta_{k})$ agree as elements of $(H^1(B))^\ast$. Since $\Phi_m(\Theta_{k,m})''\rightharpoonup U''$ it follows that $\Phi_m(\Theta_{k,m})''\rightharpoonup\Phi(\Theta)''$. Indeed, for any smooth compactly supported (in $(0,T)$) $\varphi:(0,T)\to H^1(B)$
	\begin{align*}
	\begin{split}
	\int_0^T (\Phi(\Theta_{k}), \varphi'' )\ud t=  \lim_{m\to\infty} \int_0^T (\Phi_m(\Theta_{k,m}),\varphi'')\ud t=\lim_{m\to\infty}\int_0^T(\Phi_m(\Theta_{k,m})'',\varphi)\ud t .
	\end{split}
	\end{align*}
Using a similar argument as above, starting with elements of the form $v:\sum_{\ell=1}^M v^\ell e_\ell$, we can pass to the limit in \eqref{eq:gw1} and conclude that $\Theta_{k}$ is a solution of \eqref{eq:weak2}, and the initial data are attained. Finally, the energy estimates \eqref{eq:energy} follows from passing to the limit in \eqref{eq:gen1}.
	
		{\bf{Uniqueness.}} To prove uniqueness (to ease notation we consider only the case $k=0$, but the proof easily extends to other values of $k$) suppose $\Theta$ is a solution with initial data $\theta_0$, $\theta_1$, $\tiltheta_1$, and sources $f$, $F$, and $\calF$ all equal to zero, and for some fixed $s\in (0,T)$ let
	\begin{align*}
	\begin{split}
	\zeta(t):=\begin{cases} -\int_t^s\Theta(\tau)\ud\tau,\quad &t\leq s\\ 0,\quad&t\geq s\end{cases}.
	\end{split}
	\end{align*}
	Then $\zeta(t)\in H^1(B)$ for each $t$ so integrating \eqref{eq:weak2} with $v=\zeta(t)$ gives
	\begin{align*}
	\begin{split}
	\int_0^s (\Phi(\Theta)'',\zeta)\ud t +\int_0^s \bigg(B(\Theta,\zeta)+C(\Theta',\zeta)+D((\tr\,\Theta)',\zeta)+E(\tr\,\Theta,\zeta)\bigg)\ud t =0.
	\end{split}
	\end{align*}
	Note that $\zeta'(t)=\Theta(t)$ for $t\leq s$, so in particular $\zeta'(0)=0$. Also by definition $\zeta(t)=0$ for $t\geq s$. Since $\Theta$ has zero initial data by assumption, it follows that 
	\begin{align*}
	\begin{split}
	\int_0^s(\Phi(\Theta)'',\zeta)\ud t&=-\int_0^s(\Phi(\Theta)',\zeta')\ud t\\
	& =-\int_0^s\angles{\Theta'}{\Theta}\ud t-\int_0^s\bigg(\bangles{(\tr\,\Theta)'}{\gamma^{-1}\tr\,\Theta}+\bangles{\tr\,\Theta}{(\gamma^{-1})'\tr\,\Theta}\bigg)\ud t\\
	&=-\frac{1}{2}\|\Theta(s)\|_{L^2(B)}^2-\frac{1}{2}\|\gamma^{-\frac{1}{2}}\tr\,\Theta\|_{L^2(\partial B)}^2-\int_0^s\bangles{\tr\,\Theta}{(\gamma^{-1})'\tr\,\Theta}\ud t.
	\end{split}
	\end{align*}
	Next replacing $\Theta$ by $\zeta'$ and integrating by parts we can write
	\begin{align*}
	\begin{split}
	B(\Theta,\zeta) &= \frac{1}{2}\partial_{t}\angles{\partial_a\zeta}{g^{ab}\partial_b\zeta}-\frac{1}{2}\angles{\partial_a\zeta}{(\partial_tg^{ab})\partial_b\zeta}+\frac{1}{2}\angles{\zeta'}{ (\partial_a\zeta) g^{a\alpha}\partial_\alpha\log |g|}+\frac{1}{2}\angles{\zeta'}{ \zeta\partial_a( g^{a\alpha}\partial_\alpha\log |g|)}\\
	&\quad-\frac{1}{2}\bangles{\zeta'}{\zeta g^{r\alpha}\partial_\alpha\log |g|}-\bangles{\zeta'}{\zeta\partial_tg^{tr}}+\angles{\zeta'}{(\partial_a\zeta)\partial_tg^{ta}}+\angles{\zeta'}{\zeta\partial^2_{ta}g^{ta}}.
	\end{split}
	\end{align*}
	Similar calculations give
	\begin{align*}
	\begin{split}
	C(\Theta',\zeta)&=2\partial_t\angles{\zeta'}{g^{ta} \partial_a\zeta}+\angles{\zeta'}{(\partial_ag^{ta}) \zeta'}+\bangles{\zeta'}{g^{tr}\zeta'}-2\angles{\zeta'}{(\partial_tg^{ta}) \partial_a\zeta}-\frac{1}{2}\partial_t\angles{\zeta'}{\zeta g^{t\alpha}\partial_\alpha\log |g|}\\
	&\quad +\frac{1}{2}\angles{\zeta'}{\zeta' g^{t\alpha}\partial_\alpha\log |g|}+\frac{1}{2}\angles{\zeta'}{\zeta \partial_t(g^{t\alpha}\partial_\alpha\log |g|)}+\partial_t\angles{\zeta'}{\zeta\partial_a g^{ta}}-\angles{\zeta'}{\zeta'\partial_a g^{ta}}-\angles{\zeta'}{\zeta\partial_a \partial_tg^{ta}},
	\end{split}
	\end{align*}
	and
	\begin{align*}
	\begin{split}
	D((\tr\,\Theta)',\zeta)&=-\partial_t\bangles{\tr\,\zeta'}{g^{tr}\tr\,\zeta}+\bangles{\tr\,\zeta'}{g^{tr}\tr\,\zeta'}+\bangles{\tr\,\zeta'}{(\partial_tg^{tr})\tr\,\zeta}\\
	&\quad-2\partial_t\bangles{\tr\,\zeta'}{(\gamma^{-1})'\tr\,\zeta}+2\bangles{\tr\,\zeta'}{(\gamma^{-1})'\tr\zeta'}+2\bangles{\tr\,\zeta'}{(\gamma^{-1})''\tr\zeta},
	\end{split}
	\end{align*}
	as well as 
	\begin{align*}
	E(\tr\,\Theta,\zeta)=-\bangles{\tr\,\zeta'}{(\gamma^{-1})''\tr\,\zeta}
	\end{align*}
	Putting everything together (and keeping in mind that $\zeta'(0)=\zeta(s)=0$) we arrive at
	\begin{align*}
	\begin{split}
	&\|\nabla_g \zeta(0)\|_{L^2(B)}^2+\|\Theta(s)\|_{L^2(B)}^2+\|\tr\,\Theta(s)\|_{L^2(\partial B)}^2\\
	&\leq C \int_0^s(\|\zeta(t)\|_{H^1(B)}^2+\|\zeta'(t)\|_{L^2(B)}^2+\|\tr\,\zeta'(t)\|_{L^2(\partial B)}^2)\ud t
	\end{split}
	\end{align*}
	where the constant $C$ depends only on
	\begin{align*}
	\begin{split}
	\sup_{0\leq t\leq T}\sum_{k=0}^2\bigg(\|\partial^kg\|_{L^\infty(B)}+\|\partial^k_t \gamma\|_{L^\infty(\partial B)}+\|\partial^k \gamma^{-1}\|_{L^\infty(\partial B)}\bigg).
	\end{split}
	\end{align*}
	Since $\zeta(0)=-\int_0^s\zeta'(t)\ud t$ we can add $\|\zeta(0)\|_{L^2(B)}^2$ to the left-hand side above to get (with possible different $C$ but with the same dependence on the coefficients)
	\begin{align*}
	\begin{split}
	&\|\zeta(0)\|_{H^{1}(B)}^2+\|\Theta(s)\|_{L^2(B)}^2+\|\tr\,\Theta(s)\|_{L^2(\partial B)}^2\\
	&\leq C \int_0^s(\|\zeta(t)\|_{H^1(B)}^2+\|\Theta(t)\|_{L^2(B)}^2+\|\tr\,\Theta(t)\|_{L^2(\partial B)}^2)\ud t.
	\end{split}
	\end{align*}
	Now let
	\begin{align*}
	\begin{split}
	w(t):=\int_0^t \Theta(\tau)\ud\tau
	\end{split}
	\end{align*}
	so that for $t\leq s$
	\begin{align*}
	\begin{split}
	\zeta(t) = w(t)-w(s).
	\end{split}
	\end{align*}
	It follows that 
	\begin{align*}
	\begin{split}
	&\| w(s)\|_{H^{1}(B)}^2+\|\Theta(s)\|_{L^2(B)}^2+\|\tr\,\Theta(s)\|_{L^2(\partial B)}^2\\
	&\leq C \int_0^s(\|w(t)-w(s)\|_{H^1(B)}^2+\|\Theta(t)\|_{L^2(B)}^2+\|\tr\,\Theta(t)\|_{L^2(\partial B)}^2)\ud t,
	\end{split}
	\end{align*}
	which in turn gives
	\begin{align*}
	\begin{split}
	&(1-2Cs)\| w(s)\|_{H^1(B)}^2+\|\Theta(s)\|_{L^2(B)}^2+\|\tr\,\Theta(s)\|_{L^2(\partial B)}^2\\
	&\leq 2C \int_0^s(\|w(t)\|_{H^1(B)}^2+\|\Theta(t)\|_{L^2(B)}^2+\|\tr\,\Theta(t)\|_{L^2(\partial B)}^2)\ud t.
	\end{split}
	\end{align*}
	If $s\leq s_0$ with $s_0<\frac{1}{4C}$, it follows from Gronwall that $\Theta(s)\equiv0$ on $[0,s_0]$. But since the choice of $s_0$ is independent of the choice of the time origin, we can repeat the argument  to get that $\Theta(s)\equiv 0$ on $[s_0,2s_0]$ and  continue in this way to get that $\Theta(s)\equiv0$ on $[0,T]$.
	
		{\bf{Recursive relation.}} It remains to prove the recursive relation $\Theta'_{j-1}=\Theta_{j}$. Let us define $\Gamma(t):=\Theta_{j-1}(0)+\int_{0}^{t}\Theta_{j}(s)\ud s$. By the compatibility condition, we have
		\begin{align}\label{uniqueness data j-1}
		\Gamma(0)=\Theta_{j-1}(0),\quad (\Phi(\Gamma)'(0),v)=(\Phi(\Theta_{j-1})'(0),v),\quad \forall v\in H^{1}(B).
		\end{align}
		Let us consider
		\begin{align*}
		&Z(t):=	(\Phi(\Gamma)''(t),v)+\calL(\Gamma(t),v)+(j-1)\bangles{\tilde{\calC}_{\calB}(\Gamma(t))}{v}.
		\end{align*}
		We have
		\begin{align*}
		Z'(t)=&(\Phi(\Theta_{j})''(t),v)+\calL(\Theta_{j}(t),v)+j\,\bangles{\tilde{\calC}_{\calB}(\Theta_{j}(t))}{v}+\left((\Phi(\Gamma)''(t),v)+\calL(\Gamma(t),v)+(j-1)\bangles{\tilde{\calC}_{\calB}(\Gamma(t))}{v}\right)'\\
		&-(\Phi(\Theta_{j})''(t),v)-\calL(\Theta_{j}(t),v)-j\,\bangles{\tilde{\calC}_{\calB}(\Theta_{j}(t))}{v},
		\end{align*}
		and according to \eqref{inhomogeneous source recursive},
		\begin{align*}
		&\left((\Phi(\Gamma)'',v)+\calL(\Gamma,v)+(j-1)\bangles{\tilde{\calC}_{\calB}(\Gamma)}{v}\right)'-(\Phi(\Theta_{j})'',v)-\calL(\Theta_{j},v)-j\,\bangles{\tilde{\calC}_{\calB}(\Theta_{j})}{v}\\
		&=\angles{\calC(\Gamma)}{v}+\angles{\calC^{a}(\Gamma)}{\partial_{a}v}+\bangles{\calC_{\calB}(\Gamma)}{\tr\,v}+(j-1)\bangles{\tiltilcalC_{B}(\Gamma)}{\tr\,v}.
		\end{align*}
		By the equation satisfied by $\Theta_{j}$, we have
		\begin{align*}
		Z'=&\angles{F_{j}}{v}+\angles{\calF_j^a}{\partial_a v}+\bangles{f_{j}}{\tr\,v}+\angles{\calC(\Gamma)}{v}+\angles{\calC^{a}(\Gamma)}{\partial_{a}v}+\bangles{\calC_{\calB}(\Gamma)}{\tr\,v}+(j-1)\bangles{\tiltilcalC_{B}(\Gamma)}{\tr\,v}\\
		=&\angles{F'_{j-1}}{v}+\angles{{\calF_{j-1}^{a}}'}{\partial_a v}+\bangles{f'_{j-1}}{\tr\,v}+\angles{\calC(\Gamma)-\calC(\Theta_{j-1})}{v}+\angles{\calC^{a}(\Gamma)-\calC^{a}(\Theta_{j-1})}{\partial_{a}v}\\
		&+\bangles{\calC_{\calB}(\Gamma)-\calC_{\calB}(\Theta_{j-1})}{\tr\,v}+(j-1)\bangles{\tiltilcalC_{\calB}(\Gamma)-\tiltilcalC_{\calB}(\Theta_{j-1})}{\tr\,v},
		\end{align*}
		which gives
		\begin{align*}
		Z(t)-Z(0)&=\angles{F_{j-1}(t)-F_{j-1}(0)}{v}+\angles{\calF_{j-1}^{a}(t)-\calF_{j-1}^{a}(0)}{\partial_a v}+\bangles{f_{j-1}(t)-f_{j-1}(0)}{\tr\,v}\\
		&\quad+\int_{0}^{t}\left(\angles{\calC(\Gamma)-\calC(\Theta_{j-1})}{v}+\angles{\calC^{a}(\Gamma)-\calC^{a}(\Theta_{j-1})}{\partial_{a}v}+\bangles{\calC_{\calB}(\Gamma)-\calC_{\calB}(\Theta_{j-1})}{\tr\,v}\right)\ud s\\
		&\quad+(j-1)\int_{0}^{t}\bangles{\tiltilcalC_{\calB}(\Gamma)-\tiltilcalC_{\calB}(\Theta_{j-1})}{\tr\,v}\ud s.
		\end{align*}
		In view of \eqref{uniqueness data j-1}, this implies that
		\begin{align*}
		Z(t)&=\angles{F_{j-1}(t)}{v}+\angles{\calF_{j-1}^{a}(t)}{\partial_a v}+\bangles{f_{j-1}(t)}{\tr\,v}+(j-1)\int_{0}^{t}\bangles{\tiltilcalC_{\calB}(\Gamma)-\tiltilcalC_{\calB}(\Theta_{j-1})}{\tr\,v}\ud s\\
		&\quad+\int_{0}^{t}\left(\angles{\calC(\Gamma)-\calC(\Theta_{j-1})}{v}+\angles{\calC^{a}(\Gamma)-\calC^{a}(\Theta_{j-1})}{\partial_{a}v}+\bangles{\calC_{\calB}(\Gamma)-\calC_{\calB}(\Theta_{j-1})}{\tr\,v}\right)\ud s.
		\end{align*}
	Comparing with the equation satisfied by $\Theta_{j-1}$, it follows that $w(t):=\Gamma(t)-\Theta_{j-1}(t)$ has zero initial data and satisfies
		\begin{align*}
		&(\Phi(w)''(t),v)+\calL(w(t),v)+(j-1)\bangles{\tilde{\calC}_{\calB}(w)}{\tr\,v}\\
		=&\int_{0}^{t}\left(\angles{\calC(w(s))}{v}+\angles{\calC^{a}(w(s))}{\partial_{a}v}+\bangles{\calC_{\calB}(w(s))}{\tr\,v}+(j-1)\bangles{\tiltilcalC_{\calB}(w)}{\tr\,v}\right)ds.
		\end{align*}
		Since this is a homogeneous equation with zero initial data, we can argue as in the proof of uniqueness to conclude that $w\equiv0$, which implies $\Theta_{j-1}'=\Theta_j$ as desired.
	\end{proof}

As discussed earlier, to prove higher order regulairty we need to prove Sobolev estimates on $\Theta$, using the fact that $\Theta$ satisfies the conclusions of Proposition~\ref{prop:weak1}. The main ingredient for this is the following elliptic estimate for weak solutions of the Neumann problem.
\begin{lemma}\label{lem:Neumann}
	Suppose $u\in H^1(B)$ satisfies 
	\begin{align}\label{eq:weakN-gen}
	\begin{split}
	\angles{g^{ab}\partial_au}{\partial_bv}=\bangles{w}{\tr\,v}+\angles{W}{v},\qquad \forall v\in H^1(B),
	\end{split}
	\end{align}
	for some $w\in H^{\frac{1}{2}}(\partial B)$ and $W\in L^2(B)$. Then $u\in H^2(B)$ and for some constant depending only on $g$ 
	\begin{align}\label{eq:Neumann-bound1}
	\begin{split}
	\|u\|_{H^2(B)}\leq C (\|w\|_{H^{\frac{1}{2}}(\partial B)}+\|W\|_{L^2(B)}).
	\end{split}
	\end{align}
	More generally, for each $k$  if $g,W\in H^k(B)$ and $w\in H^{k+\frac{1}{2}}(\partial B)$, then $u\in H^{k+2}(B)$ and there exists a function $P_k$ depending polynomially on its arguments such that 
	\begin{align}\label{eq:Neumann-boundk}
	\begin{split}
	\|u\|_{H^{k+2}(B)}\leq P_k(\|g\|_{H^5(B)},\|g\|_{H^k(B)},\|w\|_{H^{k+\frac{1}{2}}(\partial B)}, \|W\|_{H^k(B)}).
	\end{split}
	\end{align}
	\end{lemma}
For future reference in the treatment of the equation for $D_V\sigma^2$ we also record the following elliptic estimates for the Dirichlet problem.
\begin{lemma}\label{lem:Dirichlet}
Suppose $u\in H^1_0(B)$ satisfies
\begin{align*}
\begin{split}
\angles{g^{ab}\partial_au}{\partial_bv}=\angles{W}{v},\qquad  \forall v\in H^1_0(B),
\end{split}
\end{align*}
for some $W\in L^2(B)$. Then $u\in H^2(B)$ and for some constant depending only on $g$
\begin{align}\label{eq:Dirichlet-bound1}
	\begin{split}
	\|u\|_{H^2(B)}\leq C \|W\|_{L^2(B)}.
	\end{split}
	\end{align}
	More generally, for each $k$  if $g,W\in H^k(B)$, then $u\in H^{k+2}(B)$ and there exists a function $P_k$ depending polynomially on its arguments such that 
	\begin{align}\label{eq:Dirichlet-boundk}
	\begin{split}
	\|u\|_{H^{k+2}(B)}\leq P_k(\|g\|_{H^5(B)},\|g\|_{H^k(B)}, \|W\|_{H^k(B)}).
	\end{split}
	\end{align}
\end{lemma}
Lemmas~\ref{lem:Neumann} and~\ref{lem:Dirichlet} are standard elliptic estimates with transversal and Dirichlet boundary conditions respectively, and their proofs, which we omit, can be found in many references. See for instance \cite{Taylor-book1}. We can now prove our first Sobolev estimate on the lower derivatives of $\Theta$.
\begin{proposition}\label{prop:Sobolev1}
	Suppose $\Theta$ is as in Proposition~\ref{prop:weak1}. Then for each $k\leq K-5$ and $2a\leq K-3-k$,
	\begin{align*}
	\begin{split}
	\partial^a_\bary\Theta_{k}\in L^\infty([0,T];L^2(B))
	\end{split}
	\end{align*}
	and for each $\tau\leq T$ 
	\begin{align}
	&\|\partial^a_\bary\Theta_{k}(\tau)\|_{L^\infty([0,\tau];L^2(B))}\leq P_k\bigg(\sup_{t\leq\tau}\sum_{\ell\leq 2a+k-2}(\|\nabla \Theta_{\ell}(t)\|_{L^{2}(B)}+\| \Theta_{\ell+1}(t)\|_{L^2(B)}+\| \Theta_{\ell+1}(t)\|_{L^2(\partial B)}),\nonumber\\
	&\phantom{\|\partial^a_\bary\Theta_{k}(\tau)\|_{L^\infty([0,\tau];L^2(B))}\leq P_k\bigg(}\|g\|_{L^\infty([0,\tau];H^{\max\{a-2,5\}}(B))},\sum_{\ell\leq k}\|\partial_t^{\ell}f\|_{L^{\infty}([0,T];H^{a-\frac32}(B))}\bigg),\label{eq:Sobolev-estimate1}
	\end{align}
	where $P_k$ is a function (not the same as in Lemma~\ref{lem:Neumann}) depending polynomially on its arguments.
\end{proposition}
\begin{proof}
	The argument is a simpler (at linear lever) version of the proof of Proposition~\ref{prop:L2Sobolev}, so we will be brief on details. By Proposition~\ref{prop:weak1}, we already know that
	\begin{align*}
	\begin{split}
	\Theta_{k}\in H^1(B),\qquad k\leq K-5.
	\end{split}
	\end{align*}
	We proceed inductively. First note, that in view of Proposition~\ref{prop:weak2} we can apply Lemma~\ref{lem:Neumann} to $\Theta_{k}$, $k\leq K-7$, to get
	\begin{align*}
	\begin{split}
	\Theta_{k}\in H^2(B),\quad k\leq K-7.
	\end{split}
	\end{align*}
	This estimate now allows us to improve the regularity of the lower derivatives. Indeed, using the higher regularity statement in Lemma~\ref{lem:Neumann} it follows that
	\begin{align*}
	\begin{split}
	\Theta_{k}\in H^3(B),\quad k\leq K-9,
	\end{split}
	\end{align*}
	and inductively, for $3\leq m \leq \frac{1}{2}(K-1)$,
	\begin{align*}
	\begin{split}
	\Theta_{k}\in H^{m-1}(B),\quad k\leq K-(2m+1).
	\end{split}
	\end{align*}
	The desired estimate \eqref{eq:Sobolev-estimate1} also follows from Lemma~\ref{lem:Neumann}. 
\end{proof}
\begin{remark}\label{rmk: Linfty Sobolev}
Using Proposition \ref{prop:Sobolev1}, we can control the $L^{\infty}([0,T];L^{\infty}(B))$ norm for the lower derivatives. More precisely, under the assumptions of Proposition \ref{prop:Sobolev1}, we have, for $a\geq 2$,
\begin{align}
	&\|\partial^{a-2}_{\bary}\Theta_k(\tau)\|_{L^\infty([0,\tau];L^{\infty}(B))}\leq P_k\bigg(\sup_{t\leq\tau}\sum_{\ell\leq 2a+k-4}(\|\nabla \Theta_{\ell}(t)\|_{L^{2}(B)}+\| \Theta_{\ell+1}(t)\|_{L^2(B)}+\| \Theta_{\ell+1}(t)\|_{L^2(\partial B)}),\nonumber\\
	&\phantom{\|\partial^{a-2}_{\bary}\Theta_k(\tau)\|_{L^\infty([0,\tau];L^{\infty}(B))}\leq P_k\bigg(}\|g\|_{L^\infty([0,\tau];H^{\max\{a-2,5\}}(B))},\sum_{\ell\leq k}\|\partial_t^\ell f\|_{L^{\infty}([0,T];H^{a-\frac32}(B))}\bigg).\label{eq:Sobolev-estimate2}
\end{align}
\end{remark}
Based on the Sobolev estimate \eqref{eq:Sobolev-estimate2}, under the assumptions \eqref{eq:g-assumption1}, we have the following improved version of Proposition \ref{prop:weak1}:
	\begin{proposition}\label{prop:weak2}
	Suppose \eqref{eq:g-assumption1} holds and that there exist
	\begin{align*}
	\begin{split}
	\theta_k\in H^1(B),\quad \theta_{k+1}\in L^2(B), \quad \tiltheta_{k+1}\in L^2(\partial B), \quad k=0,...,K
	\end{split}
	\end{align*}
	such that the following two conditions hold:
	\begin{itemize}
		\item {\bf{Regularity:}} For $k=0,\dots,K-1$
		\begin{align*}
		\begin{split}
		\angles{\theta_{k+2}}{v}+\bangles{\tiltheta_{k+2}}{\gamma^{-1}\tr\,v}+\calL(\theta_{k},v)+k\bangles{\tilcalC_\calB(\theta_k)}{v}
		=\bangles{f_{k}(0)}{\tr\,v}+\angles{F_{k}(0)}{v}+\angles{\calF^{a}_{k}(0)}{\partial_av}.
		\end{split}
		\end{align*}
		Here $f_k(0)$, $F_k(0)$, $\calF^{a}_{k}(0)$ are defined as in Proposition~\ref{prop:weak1}.
		\item {\bf{Compatibility:}} $\tiltheta_k=\tr\,\theta_k$ for $k=1,\dots, K$.
	\end{itemize}
	Then there exists a unique $\Theta_{k}$ satisfying \eqref{eq:Thetaspaces1} and \eqref{eq:weakdata1}, such that for all $v\in H^1(B)$ equation \eqref{eq:weak2} holds for almost every $t\in[0,T]$. The solution satisfies
	\begin{align}\label{eq:energy high}
	\begin{split}
	&\sup_{t\in [0,T]}\big(\|\Theta'_{k}\|_{L^2(B)}+\|\Theta_{k}\|_{H^1(B)}+\|\tr\,\Theta'_{k}\|_{L^2(\partial B)}\big)\\
	&\leq C_1e^{C_2T}\Big(\|\theta_k\|_{H^1(B)}+\|\theta_{k+1}\|_{L^2(B)}+\|\tiltheta_{k+1}\|_{L^2(\partial B)}+\|f_{k}\|_{L^2([0,T];L^2(\partial B))}\\
	&\phantom{\leq C_1e^{C_2T}\Big(}+\|F_{k}\|_{L^2([0,T];L^2(B))}+\|\calF^{a}_{k}\|_{L^{\infty}([0,T];L^2(B))}\Big).
	\end{split}
	\end{align}
	In these estimates $C_1$, $C_2$, and $C_3$ are constants depending on the various norms of $g,\tr\,g, \gamma, \gamma^{-1}$ appearing in \eqref{eq:g-assumption1}.  Moreover, we have $\Theta'_{k-1}=\Theta_{k}$ for $k=1,...,K$, and there exist functions $P_k$ depending polynomially on their arguments such that \eqref{eq:Sobolev-estimate1} holds for $k\leq K$ and $2a+k\leq K+2$.
\end{proposition}
\begin{proof}
	The proof is similar to that of Proposition \ref{prop:weak1}. The only difference is that when most derivatives fall on the coefficients $g,\gamma$ we bound these terms in $L^2$ and bound the lower order derivatives of $\Theta$ in $L^{\infty}(B)$, using the Sobolev estimate \eqref{eq:Sobolev-estimate2}. We omit the routine details.
\end{proof}

We turn to the equation for $D_V\sigma^2$. The overall proofs of existence, uniqueness, and higher regularity are similar to those in Propositions~\ref{prop:weak1} and~\ref{prop:weak2}. Therefore, we will omit most details and concentrate on deriving the appropriate energy estimate.

\begin{lemma}\label{lem:Lambdanormal}
Assume \eqref{eq:g-assumption1} hold. Given $H\in L^2([0,T];L^2(B))$ and $\calH^a\in L^\infty([0,T];L^2(B))$, suppose $\Lambda$ satisfying \eqref{eq:Lambdaspaces1} and \eqref{eq:Lambdaid} is a weak solution of 
\begin{align}\label{eq:Lambdaep}
\begin{split}
(\Lambda'',v)+\calL_\sigma(\Lambda,v) = \angles{H}{,v}+\angles{\calH^a}{\partial_av},\qquad \forall v\in H^1_0(B),
\end{split}
\end{align}
satisfying
\begin{equation}\label{eq:energystandardtemp1}
\sup_{t\in[0,T]}(\|\Lambda'\|_{L^2(B)}^2+\|\Lambda\|_{H^1(B)}^2)\leq c_0( \|\lambda_0\|_{H^1(B)}^2+\|\lambda_1\|_{L^2(B)}^2+\|H\|_{L^2([0,T];L^2(B))}^2+\|\calH\|_{L^\infty([0,T];L^2(B))}^2).
\end{equation}
If $\partial_\bary\calH^a\in L^1([0,T];L^2(B))$, then $\Lambda$ satisfies \footnote{Here $\nabla\Lambda$ on the boundary $\partial B$ is in the weak sense: Suppose $\Lambda\in H^{1}(B)$ is a solution to \eqref{eq:Lambdaep}. Then its weak normal derivative $\nabla^{w}_{n}\Lambda$ on $[0,T]\times\partial B$ is defined such that for any $\varphi\in C^{\infty}([0,T]\times\overline{B})$, we have 
\begin{align*}
	\int_{0}^{T}\int_{\partial B}\bangles{\nabla_{n}^{w}\Lambda}{\varphi}dS\,dt:=\int_{0}^{T}\int_{B}\angles{g^{ab}\partial_{a}\Lambda}{\partial_{b}\varphi}d\bary\,dt+\int_{0}^{T}\int_{B}\left((\Lambda'',\varphi)-\angles{H}{\varphi}-\angles{\calH^{a}}{\partial_{a}\varphi}+\tilde{\calL}_{\sigma}(\Lambda,\varphi)\right)d\bary\,dt.
\end{align*}
Here $\tilde{\calL}_{\sigma}(\Lambda,\varphi):=\calL_{\sigma}(\Lambda,\varphi)-\angles{g^{ab }\partial_{a}\Lambda}{\partial_{b}\varphi}.$}
\begin{align}\label{eq:Lambdanormal1}
\begin{split}
\|\nabla\Lambda\|_{L^2([0,T];L^2(\partial B))}^2\leq c_1\big( \|\lambda_0\|_{H^1(B)}^2+\|\lambda_1\|_{L^2(B)}^{2}+\|H\|_{L^2([0,T];L^2(B))}^2+\|\partial_a\calH^a\|_{L^2([0,T];L^2(B))}^2\big),
\end{split}
\end{align}
for some constant $c_1$ depending only on the first two derivatives of $g$ and on $c_0$. If instead $\Lambda$ can be written as $\partial_t\Gamma$, with $\Gamma\in L^\infty([0,T];H^1(B))$, and $\partial_t\calH\in L^2([0,T];L^2(B))$, $\calH^r\in L^2([0,T];L^2(\partial B))$, then $\|\partial_a\calH^a\|_{L^2([0,T];L^2(B))}^2$ on the right-hand side of \eqref{eq:Lambdanormal1} can be replaced by
\begin{align}\label{eq:Lambdanormal2}
\begin{split}
\|\calH^r\|_{L^2([0,T];L^2(\partial B))}^2+\|\calH\|_{L^2([0,T];L^2(B))}^2+\int_0^T\int_B |\partial_t\calH||\partial^2_{t,\bary}\Gamma|\ud \bary \ud t+\sup_{t\in[0,T]}\int_B|\calH||\partial_{t,\bary}^2\Gamma|\ud\bary.
\end{split}
\end{align}
\end{lemma}
\begin{proof}
This is a standard estimate for the wave equation with Dirichlet boundary conditions. See for instance \cite{Lions-Magenes-book1,Wloka-book,Evans-book} and Lemma~\ref{lem:benergy} above. We sketch the proof for completeness. Note that by \eqref{eq:g-assumption1} the metric $g$ is at least $C^3$. We approximate $(\lambda_0,\lambda_1)$, $H$, and $H^a$ by regular (say $C^3$) functions $(\lambda_0^\ep,\lambda_1^\ep)$, $H^\ep$, and $H^{\ep,a}$ such that 
\begin{align*}
\begin{split}
(\lambda_0^\ep,\lambda_1^\ep)\to (\lambda_0,\lambda_1) \qquad &\mathrm{in}\quad H^1_0(B)\times L^2(B),\\
H^\ep\to H\qquad &\mathrm{in}\quad L^2([0,T];L^2(B)),\\
\partial_aH^{\ep,a}\to \partial_aH^{a}\qquad &\mathrm{in}\quad L^2([0,T];L^2(B)),
\end{split}
\end{align*}
and satisfying appropriate compatibility conditions. Let $\Lambda^\ep$ be the solution of the corresponding wave equation
\begin{align*}
\begin{split}
\Box_g\Lambda^\ep = h^\ep:=H^\ep-\partial_a\calH^{\ep,a},\qquad (\Lambda^\ep(0),\partial_t\Lambda^\ep(0))=(\lambda_0^\ep,\lambda_1^\ep).
\end{split}
\end{align*}
Multiplying by a general multiplier $Q\Lambda^\ep:=q^\gamma\partial_\gamma\Lambda^\ep$ we get
\begin{align}
(\Box_g\Lambda^\ep)(Q\Lambda^\ep)&=\partial_\alpha\big(g^{\alpha\beta}q^\gamma(\partial_\gamma\Lambda^\ep)(\partial_\beta\Lambda^\ep)+\frac{1}{2}q^{\alpha}g^{\gamma\beta}(\partial_\beta\Lambda^\ep)(\partial_\gamma\Lambda^\ep)\big)\label{eq:genmultLag}\\
&\quad+\frac{1}{2}(\partial_\alpha\log |g|)q^\gamma g^{\alpha\beta}(\partial_\beta\Lambda^\ep)(\partial_\gamma\Lambda^\ep)-g^{\alpha\beta}(\partial_\alpha q^\gamma)(\partial_\beta\Lambda^\ep)(\partial_\gamma\Lambda^\ep)+\frac{1}{2}\partial_\gamma(g^{\alpha\beta}q^\gamma)(\partial_\alpha\Lambda^\ep)(\partial_\beta\Lambda^\ep).\nonumber
\end{align}
We now take $q^{\alpha}=g^{\alpha r}$. Integrating \eqref{eq:genmultLag} over $[0,T]\times B$ we see that the contribution on the timelike boundary $[0,T]\times\partial B$ is
\begin{align*}
\begin{split}
(g^{r\alpha}\partial_\alpha\Lambda^\ep)^2-\frac{1}{2}g^{rr}g^{\alpha\beta}(\partial_\alpha\Lambda^\ep)(\partial_\beta\Lambda^\ep).
\end{split}
\end{align*}
decomposing into polar coordinates, and noting that the tangential derivatives (that is, $\partial_t$ and the angular derivatives tangential to $\partial B$) of $\Lambda^\ep$ are zero (because $\Lambda^\ep$ is constant on $[0,T]\times \partial B$) this expression simplifies to
\begin{align*}
\begin{split}
\frac{1}{2}(g^{rr}\partial_r\Lambda^\ep)^2.
\end{split}
\end{align*}
Since $g^{rr}$ is bounded away from zero and the tangential derivatives of $\Lambda^\ep$ are zero on $[0,T]\times \partial B$ this controls $|\partial_{t,\bary} \Lambda^\ep|^2$ on $[0,T]\times \partial B$. Therefore integration of \eqref{eq:genmultLag} gives
\begin{align}\label{eq:Lambdanormaltemp1}
\begin{split}
\int_0^T\int_{\partial B}|\partial_{t,\bary}\Lambda^\ep|^2\ud S\,\ud t\lesssim \sup_{t\in[0,T]}\int_B|\partial_{t,\bary}\Lambda^\ep|^2\ud \bary+ \int_0^T\int_B|\partial_{t,\bary}\Lambda_\ep|^2\ud\bary\,\ud t+\Big|\int_0^T\int_B (\Box \Lambda^\ep)(Q\Lambda^\ep)\ud \bary\,\ud t\Big|.
\end{split}
\end{align}
Estimate \eqref{eq:Lambdanormal1} for $\Lambda^\ep$ follows after adding a suitably large multiple of \eqref{eq:energystandardtemp1}, and the corresponding estimate for $\Lambda$ follows by taking the limit $\ep\to0$. For \eqref{eq:Lambdanormal2} we simply integrate by parts in the last term in \eqref{eq:Lambdanormaltemp1} (here $\Gamma^{\epsilon}$ is defined such that $\partial_{t}\Gamma^{\epsilon}=\Lambda^{\epsilon}$): 
\begin{align*}
\begin{split}
\int_0^T\int_B(\partial_a\calH^{\ep,a})Q\Lambda^\ep\ud\bary \ud t&= \int_0^T\int_{\partial B}\calH^{\ep,r} Q\Lambda^\ep \ud S\ud t-\int_0^T\int_B\calH^{\ep,a}(\partial_aq^\alpha)\partial_\alpha\Lambda^\ep\ud \bary\ud t\\
&\quad-\int_0^T\int_B\calH^{\ep,a}q^\alpha\partial_t\partial^2_{a\alpha}\Gamma^\ep\ud \bary\ud t\\
&=\int_0^T\int_{\partial B}\calH^{\ep,r} Q\Lambda^\ep \ud S\ud t-\int_0^T\int_B\calH^{\ep,a}(\partial_aq^\alpha)\partial_\alpha\Lambda^\ep\ud \bary\ud t\\
&\quad+\int_0^T\int_B (\partial_t\calH^{\ep,a})q^\alpha\partial^2_{a\alpha}\Gamma^\ep\ud \bary \ud t-\int_B\calH^{\ep,a}q^\alpha\partial^2_{a\alpha}\Gamma^\ep\ud\bary\Big\vert^T_0.
\end{split}
\end{align*}
This gives the desired estimate for $\Lambda^\ep$ and the corresponding estimate for $\Lambda$ follows by taking limits.
\end{proof}

We can now state the analogue of Proposition~\ref{prop:weak2} for $\Lambda$.

\begin{proposition}\label{prop:weaksigma}
Suppose \eqref{eq:g-assumption1} holds and that there exist
	\begin{align*}
	\begin{split}
	\lambda_k\in H^1_0(B),\quad \lambda_{k+1}\in L^2(B),  \quad k=0,...,K
	\end{split}
	\end{align*}
	such that
		\begin{align*}
		\begin{split}
		\angles{\lambda_{k+2}}{v}+\calL_\sigma(\lambda_{k},v)
		=\angles{F_{\sigma,k}(0)}{v}+\angles{\calF^{a}_{\sigma,k}(0)}{\partial_av}.
		\end{split}
		\end{align*}
		Here $F_{\sigma,k}$ and $\calF^{a}_{\sigma,k}$ and their initial values are defined as in Proposition~\ref{prop:weak1} using \eqref{eq:Lambdalinsource}.
	Then there exists a unique $\Lambda_{k}$ satisfying \eqref{eq:Lambdaspaces1} and \eqref{eq:Lambdaid}, such that for all $v\in H^1_0(B)$ equation \eqref{eq:Lambdalinweak2} holds for almost every $t\in[0,T]$. The solution satisfies
	\begin{align}\label{eq:lambda-energy}
	\begin{split}
	&\sup_{t\in [0,T]}\big(\|\Lambda'_{k}\|_{L^2(B)}+\|\Lambda_{k}\|_{H^1(B)}\big)+\|\nabla \Lambda_k\|_{L^2([0,T];L^2(\partial B))}^2\\
	&\leq C_1e^{C_2T}\Big(\|\lambda_k\|_{H^1(B)}+\|\lambda_{k+1}\|_{L^2(B)}+\|F_{\sigma,k}\|_{L^2([0,T];L^2(B))}+\|\calF_{\sigma,k}\|_{L^{\infty}([0,T];L^2(B))}\Big).
	\end{split}
	\end{align}
	In these estimates $C_1$, $C_2$, and $C_3$ are constants depending on the various norms of $g$ appearing in \eqref{eq:g-assumption1}.  Moreover, we have $\Lambda'_{k-1}=\Lambda_{k}$ for $k=1,...,k$, and for for some function $P_k$ depending polynomially on its arguments such that for $k\leq K$ and $2a+k\leq K+2$  
	\begin{align}\label{eq:LambdaSobolev-estimate1}
	\begin{split}
	&\|\partial^a_\bary\Lambda_k(\tau)\|_{L^\infty([0,\tau];L^2(B))}\leq P_k\bigg(\sup_{t\leq\tau}\sum_{\ell\leq 2a+k-2}(\|\nabla \Lambda_{\ell}(t)\|_{L^{2}(B)}+\| \Lambda_{\ell+1}(t)\|_{L^2(B)}),\\
	&\phantom{\|\partial^a_\bary\Lambda_k(\tau)\|_{L^\infty([0,\tau];L^2(B))}\leq P_k\bigg(}\|g\|_{L^\infty([0,\tau];H^{\max\{a-2,5\}}(B))},\sum_{\ell\leq k}\|\partial_{t}^{\ell}F_{\sigma}\|_{L^{2}([0,T];H^a(B))}\bigg).
	\end{split}
	\end{align}
\end{proposition}
\begin{proof}
The proof is similar to those of Propositions~\ref{prop:weak1},~\ref{prop:weak2}, and~\ref{prop:Sobolev1}, where for higher derivatives we use Lemma~\ref{lem:Dirichlet} instead of~\ref{lem:Neumann} (see also \cite{Wloka-book,Lions-Magenes-book1}).  The only part that requires separate treatment is the estimate on $\|\nabla\Lambda_k\|_{L^2([0,T];L^2(\partial B))}$ in \eqref{eq:lambda-energy}. For this we may assume that we already have a weak solution satisfying \eqref{eq:lambda-energy} without $\|\nabla\Lambda_k\|_{L^2([0,T];L^2(\partial B))}$ on the left-hand side, and then appeal to Lemma~\ref{lem:Lambdanormal} to finish the proof. Here, note that $\Gamma$ on the right-hand side of \eqref{eq:Lambdanormal2} corresponds to $\Lambda_{k-1}$ so the corresponding contribution can be bounded using Cauchy-Schwarz with a small constant, and the elliptic estimate \eqref{eq:Dirichlet-bound1} applied to \eqref{eq:Lambdalinweak2} with $k$ replaced by $k-1$. See the calculation leading to \eqref{eq:gen1} for a similar estimate. 
\end{proof}

Before moving on to the iteration for the nonlinear system, we need one more estimate for $\Theta$ corresponding to Lemma~\ref{lem:nablaDV1}. In our scheme, this will be necessary to guarantee the second assumption in \eqref{eq:g-assumption1}.

\begin{lemma}\label{lem:voblique}
Under the assumptions of Proposition~\ref{prop:weak2}, for any $\ell \leq K-1$
\begin{align*}
\begin{split}
\|\nabla\Theta_\ell\|_{L^2([0,T];L^2(\partial B))}^2&\lesssim \sum_{j\leq\ell}\Big(\|F_j\|_{L^2([0,T];L^2(B))}^2+\|\calF_j\|_{L^2([0,T];L^2(B))}^2+\|f_j\|_{L^2([0,T];L^2(\partial B))}^2\Big)\\
&\quad+\|\theta_\ell\|_{H^1(B)}^2+\|\theta_{\ell+1}\|_{L^2(B)}^2+\|\tiltheta_{\ell+1}\|_{L^2(\partial B)}^2\\
&\quad+\|(\tr\,\Theta_\ell)'\|_{L^2([0,T];L^2(\partial B))}^2+\|(\tr\,\Theta_\ell)''\|_{L^2([0,T];L^2(\partial B))}^2,
\end{split}
\end{align*}
where the implicit constant depends only on $g$, $\gamma$, and their first three derivatives.
\end{lemma}
\begin{proof}
The proof is essentially the same as that of Lemma~\ref{lem:nablaDV1} adapted to the Lagrangian setting as in the proof of Lemma~\ref{lem:Lambdanormal}. We omit the details.
\end{proof}

\section{The Iteration}\label{sec:iteration}
With the linear theory and energy estimates at hand, the proof of Theorem~\ref{thm:main} is a more or less routine application of Picard iteration.  We continue to work with the renormalization $c=1$ to simplify the notation.

\begin{proof}[Proof of Theorem~\ref{thm:main}]
We will prove the existence for the system on the Lagrangian side \eqref{eq:V2}, \eqref{eq:LambdaLag}, \eqref{eq:SigmaLag}. We will sketch the proof of existence of the solution (convergence of the iteration) in some detail, and uniqueness and persistence of regularity follow from similar arguments as usual (see for instance \cite{Wu97}, Theorem~5.11). It is then a routine calculation to go back to the corresponding Eulerian equations in \eqref{eq:sigma1a-intro} and \eqref{eq:sigma1a-intro},  provided $K$ is sufficiently large. At the end, we will show how to return to the original equation \eqref{eq:Voriginalintro2}.

\underline{\emph{The differentiate equations \eqref{eq:V2}, \eqref{eq:LambdaLag}, \eqref{eq:SigmaLag}:}} We will use $\Thetam$, $\Lambdam$, and $\Sigmam$ to denote the iterates of $V$, $D_V\sigma^2$, and $\sigma^2$ on the Lagrangian side, respectively. The zeroth iterates $\Theta^{(0)}$ and $\Lambda^{(0)}$ are chosen such that they are polynomials in $t$ and when $t=0$ they themselves and their first order time derivatives agree with the corresponding initial data. The zeroth iterate $\Sigma^{(0)}$ is defined to be the solution to the elliptic equation \eqref{eq:SigmaLag} where $\Theta, \Lambda$ in that equation are replaced by $\Theta^{(0)}, \Lambda^{(0)}$ respectively. Given $\Thetam$ we define the $m$th iterate of the (renormalized) Lagrangian map $X$ by 
\begin{align*}
\begin{split}
\frac{\partial (\Xm)^i}{\partial t} = \frac{(\Thetam)^i}{(\Thetam)^0},\quad i=1,2,3.
\end{split}
\end{align*}
The $m$th iterate, $\dgm$, of the metric is then defined as
\begin{align*}
\begin{split}
\dgm=&-\left(1-\sum_{i=1}^3\frac{(\Thetam)^{i})^{2}}{((\Thetam)^{0})^{2}}\right)\ud t^{2}+2\sum_{i,\ell=1}^3\frac{(\Thetam)^{i}}{(\Thetam)^{0}}\frac{\partial (\Xm)^{i}}{\partial y^{\ell}}\ud t\ud y^{\ell}+\sum_{i,k,\ell=1}^3\frac{\partial (\Xm)^{i}}{\partial y^{k}}\frac{\partial (\Xm)^{i}}{\partial y^{\ell}}\ud y^{k}\ud y^{\ell}.
\end{split}
\end{align*}
We denote the components of the metric $\dgm$ and the inverse metric $\gm$ by $\dgm_{\alpha\beta}$ and $\gm^{\alpha\beta}$ respectively. Note that these components, and their first time derivatives, are at the same regularity level as one derivative of $\Thetam$. 
The $m$th iterate of the coefficient $\gamma$ is defined as
\begin{align*}
\begin{split}
\gammam:=\frac{\sqrt{\gm^{\alpha\beta}\partial_\alpha\Sigmam\partial_\beta\Sigmam}}{2((\Thetam)^0)^2},
\end{split}
\end{align*}
and $\Phim:H^1(B)\to (H^1(B))^\ast$ by
\begin{align*}
\begin{split}
(\Phim(u),v):=\angles{u}{v}+\bangles{\gammam^{-1}\tr\,u}{\tr\,v}.
\end{split}
\end{align*}
If the $m$th iterate of $V$, $D_V\sigma^2$, and $\sigma^2$ are given, we define the $(m+1)$st iterates as follows: First $\Thetamp\in L^2([0,T];H^1(B))$ is the weak solution of
\begin{align}\label{eq:Thetamp1}
\begin{split}
&(\Phim(\Thetamp)'',v)+\calLm(\Thetamp,v)=\bangles{\fm}{\tr\,v},\qquad \forall v\in H^1(B),\\
&\Thetamp(0)=\theta_0\quad\mathrm{in~} L^2(B),\\
&(\Phi(\Thetamp)'(0),v)=\angles{\theta_1}{v}+\bangles{\tiltheta_1}{\tr\,v},\qquad \forall v\in H^1(B).
\end{split}
\end{align}
Here $\calLm$ is defined as
\begin{align*}
\begin{split}
\calLm(u,v):=\Bm(u,v)+\Cm(u',v)+\Dm((\tr\,u)',v)+\Em(\tr\,u,v),
\end{split}
\end{align*}
where
\begin{align*}
\begin{split}
&\Bm:H^1(B)\times H^1(B)\to \bbR,\quad \Cm:L^2(B)\times H^1(B)\to\bbR,\quad \Dm,\Em:L^2(\partial B)\times H^1(B)\to \bbR,\\
&\Bm(u,v):=\angles{\gm^{ab}\partial_au}{\partial_bv} -\frac{1}{2}\angles{\partial_au}{ v\gm^{a\alpha}\partial_\alpha\log |\dgm|}-\angles{\partial_au}{v\partial_t\gm^{ta}},\\
&\Cm(u,v):=2\angles{u}{\gm^{ta} \partial_av}-\frac{1}{2}\angles{u}{v \gm^{t\alpha}\partial_\alpha\log |\dgm|}+\angles{u}{v\partial_a \gm^{ta}},\\
&\Dm(u,v):=-\bangles{u}{\gm^{tr}\tr\,v}-2\bangles{u}{(\gammam^{-1})'\tr\,v},\\
&\Em(u,v):=-\bangles{u}{(\gammam^{-1})''\tr\,v}.
\end{split}
\end{align*}
Similarly, $\Lambdamp\in L^2([0,T];H^1_0(B))$ is the weak solution of
\begin{align}\label{eq:Lambdamp1}
\begin{split}
&((\Lambdamp)'',v)+\Bm(\Lambdamp,v)+\Cm((\Lambdamp)',v)+\angles{\Fm_\sigma}{v}=0,\qquad \forall v\in H^1_0(B),\\
&\Lambdamp(0)=\lambda_0,\qquad \quad ((\Lambdamp)'(0),v)=\angles{\lambda_1}{v},\quad\forall v\in H^1_0(B),
\end{split}
\end{align}
where 
\begin{align*}
\begin{split}
\Fm_\sigma:= S(\Thetam,\Sigmam)
\end{split}
\end{align*}
Here $S$ is defined as in \eqref{eq:Lambdam-intro} and $(\Lambdam)''$ should be understood as an element of $H^{-1}(B):=(H^1_0(B))^\ast$, where $\Lambdam$ is identified with an element of $H^{-1}(B)$ through $(\Lambdam,v):=\angles{\Lambdam}{v}$. 
Finally, $\Sigmamp$ is defined through the transport equation $\partial_t\Sigmamp= \frac{1}{(\Thetamp)^0}\Lambdamp$.

{\bf{Boundedness.}} From now on, we will stop writing $\tr \,u$ for the restriction to the boundary and simply write $u$ when there is no risk of confusion. Let\footnote{Here among the components of $\nabla_{t,y}\partial_{t}^{K}\Lambda^{(m)}$, the normal components $\nabla_{n}\partial^{K}\Lambdam$ is defined in the weak sense as in the statement of Lemma \ref{lem:Lambdanormal}}.

\begin{align*}
\begin{split}
\calE_k^m(T)&:=\sup_{0\leq t\leq T}\sum_{\ell\leq k}(\|\nabla_{t,y}\partial_t^\ell\Thetam(t)\|_{L^2(B)}^2+\|\nabla_{t,y}\partial_t^\ell\Lambdam(t)\|_{L^2(B)}^2+\|\partial_t^{\ell+1}\Thetam(t)\|_{L^2(\partial B)}^2)\\
&\quad+\sum_{\ell\leq k}\int_0^T\|\nabla_{t,y}\partial_t^\ell\Lambdam(t)\|_{L^2(\partial B)}^2\ud t.
\end{split}
\end{align*}
We claim that if $T$ is sufficiently small, then there are constants $A_0\leq A_1\leq\dots\leq A_K$ such that for all $m$
\begin{align}\label{eq:boundedness-claim}
\begin{split}
\calE_k^m(T)\leq A_k,\quad k=0,\dots,K.
\end{split}
\end{align}
For $m=0$ this holds trivially for any $T$, so we assume that \eqref{eq:boundedness-claim} holds for some $m$, with constant $A_k$ to be determined, and prove it for $m+1$.  Let us note a few consequences of the induction hypothesis. First, from the Sobolev estimates in Lemmas~\ref{lem:Neumann} and~\ref{lem:Dirichlet} (see also Proposition~\ref{prop:Sobolev1}) it follows that
\begin{align*}
\begin{split}
\sum_{\ell\leq k}\sum_{2p+\ell\leq k+2}(\|\partial_\bary^p\partial_t^\ell\Thetam\|_{L^2(B)}^2+\|\partial_\bary^p\partial_t^\ell\Lambdam\|_{L^2(B)}^2)\leq C_{A_k},
\end{split}
\end{align*}
for some constant depending on $A$. Also note that in in view of the transport equation defining $\Sigmamp$, we can estimate $\partial_t^{k}\Sigmamp$ in terms of $\partial_t^\ell\Lambdamp$ and $\partial_t^\ell\Thetamp$ for $\ell\leq k-1$. 
It then follows from these observations that the coefficients and source terms in the equations \eqref{eq:Thetamp1} for $\Thetamp$ and \eqref{eq:Lambdamp1} for $\Lambdamp$ satisfy the assumptions in Propositions~\ref{prop:weak2} and~\ref{prop:weaksigma} respectively. Since $\calF=0$ in equation \eqref{eq:Thetamp1} the energy estimates in Propositions~\ref{prop:weak2} and~\ref{prop:weaksigma} imply that
\begin{align*}
\begin{split}
\calE_0^{m+1}(T)\leq C_0+TC_{0,A_K},
\end{split}
\end{align*}
where $C_0$ depends only on the initial data. If $A_0$ is sufficiently large and $T$ sufficiently small it follows that
\begin{align*}
\begin{split}
\calE_0^{m+1}(T)<A_0
\end{split}
\end{align*}
as desired. Next, again by Propositions~\ref{prop:weak2} and~\ref{prop:weaksigma}
\begin{align*}
\begin{split}
\calE_1^{m+1}(T)\leq C_1+ C_{1,A_0 }+TC_{1,A_K}
\end{split}
\end{align*}
where $C_1$ depends only on the initial data, and the term $C_{1,A_0}$ comes from the fact that now $\calF\neq0$ after commuting one $\partial_t$ with \eqref{eq:Thetamp1}. If $A_1$ is sufficiently large, relative to $C_1$ and $A_0$, and $T$ sufficiently small, it follows that
\begin{align*}
\begin{split}
\calE_1^{m+1}(T)<A_1.
\end{split}
\end{align*}
We can now continue inductively in this fashion to prove \eqref{eq:boundedness-claim} with $m$ replaced by $m+1$. The only additional detail is that for higher values of $k$, we also need to use the Sobolev estimates to bound $\Thetamp$ and $\Lambdamp$ (and their lower order derivatives) in $L^\infty$ in terms of $\calE_k^{m+1}$. Using Gronwall we can then conclude boundedness of the higher order energies as above. We omit the routine details. Note that after completing the proof of \eqref{eq:boundedness-claim} we can again appeal to the Sobolev estimates from Lemmas~\ref{lem:Neumann} and~\ref{lem:Dirichlet} to conclude that 
\begin{align*}
\begin{split}
\sum_{k\leq K}\sum_{2p+k\leq K+2}(\|\partial_\bary^p\partial_t^k\Thetam\|_{L^2(B)}^2+\|\partial_\bary^p\partial_t^\ell\Lambdam\|_{L^2(B)}^2)\leq C_{A_K}.
\end{split}
\end{align*}

{\bf{Convergence.}} Having established that $\calE_K^m(T)$ is uniformly bounded, we prove the convergence of $\Thetam$, $\Lambdam$, and $\Sigmam$ in some lower order Sobolev norm. The argument is standard using our energy estimates for linear systems and we will be brief. If $K$ is sufficiently large we can assume that equations \eqref{eq:Thetamp1} and \eqref{eq:Lambdamp1} are satisfied in the strong sense. Let
\begin{align*}
\begin{split}
C_m(t)&:=\sup_{0\leq s\leq t}\sum_{\ell\leq 5}(\|\nabla_{t,y}\partial_t^\ell\Thetamp(s)-\Thetam(s)\|_{L^2(B)}^2+\|\partial_t^{\ell+1}\Thetamp(s)-\Thetam(s)\|_{L^2(\partial B)}^2)\\
&\quad+\sup_{0\leq s\leq t}\sum_{\ell\leq 5}\|\nabla_{t,y}\partial_t^\ell\Lambdamp(s)-\Lambdam(s)\|_{L^2(B)}^2+\sum_{\ell\leq 5}\int_0^t\|\nabla_{t,y}\partial_t^\ell\Lambdamp(s)-\Lambdam(s)\|_{L^2(\partial B)}^2\ud s.
\end{split}
\end{align*}
We write the equations satisfied by $\Thetamp$ and $\Lambdamp$ in the schematic forms
\begin{align*}
\begin{cases}
\Box_{\dgm}\Thetamp =0,\qquad &\mathrm{in~}[0,T]\times B\\
\partial_t^2\Thetamp+\gamma^{(m)}\nabla_{n^{(m)}}\Thetamp=f(\Thetam,\Lambdam),\qquad&\mathrm{on~}[0,T]\times \partial B
\end{cases},
\end{align*}
and
\begin{align*}
\begin{cases}
\Box_{\dgm}\Lambdamp =F(\Thetam,\Sigmam),\qquad &\mathrm{in~}[0,T]\times B\\
\Lambdamp\equiv0,\qquad&\mathrm{on~}[0,T]\times \partial B
\end{cases}.
\end{align*}
and recall that $\partial_t\Sigmamp= \frac{1}{(\Thetamp)^0}\Lambdamp$.
Taking differences we see that $\Thetamp-\Thetam$ and $\Lambdamp-\Lambdam$ have zero initial data and satisfy
\begin{align}\label{eq:Thetadif1}
\begin{cases}
\Box_{\dgm}(\Thetamp-\Thetam)= \Box_{\dgm-\dgmm}\Thetam,\qquad &\mathrm{in~}[0,T]\times B\\
(\partial_t^2+\gamma^{(m)}\nabla_{n^{(m)}})(\Thetamp-\Thetam)=f(\Thetam,\Lambdam)-f(\Thetamm,\Lambdamm)\\
\phantom{(\partial_t^2+\gamma^{(m)}\nabla_{n^{(m)}})(\Thetamp-\Thetam)=}-(\gamma^{(m)}\nabla_{n^{(m)}}-\gamma^{(m-1)}\nabla_{n^{(m-1)}})\Thetamm,\qquad&\mathrm{on~}[0,T]\times\partial B
\end{cases},
\end{align}
and
\begin{align}\label{eq:Lambdadif1}
\begin{cases}
\Box_{\dgm}(\Lambdamp-\Lambdam)= F(\Thetam,\Sigmam)-G(\Thetamm,\Sigmamm)\\
\phantom{\Box_{\dgm}(\Lambdamp-\Lambdam)=}-\Box_{\dgm-\dgmm}\Lambdam,\qquad &\mathrm{in~}[0,T]\times B\\
\Lambdamp-\Lambdam\equiv0,\qquad&\mathrm{on~}[0,T]\times\partial B
\end{cases}.
\end{align}
Similarly,
\begin{align}\label{eq:Sigmadif1}
\partial_t(\Sigmamp-\Sigmam)= \frac{1}{(\Thetamp)^0}\Lambdamp-\frac{1}{(\Thetam)^0}\Lambdam.
\end{align}
Since we have already shown that all coefficients, as well as their first few derivatives, are uniformly bounded, we can apply the energy estimates from Propositions~\ref{prop:weak2} and~\ref{prop:weaksigma} to \eqref{eq:Thetadif1} and \eqref{eq:Lambdadif1} to conclude that for some absolute constant $C$
\begin{align*}
\begin{split}
C_m(t)\leq C\int_0^t C_{m-1}(t_1)\ud t_1.
\end{split}
\end{align*}
Iterating this inequality gives
\begin{align*}
\begin{split}
C_m(t)\leq \frac{C^m t^m}{m!}\sup_{t\leq T}C_0(t),
\end{split}
\end{align*}
proving that $(\Thetam)_{m=0}^\infty$ and $(\Lambdam)_{m=0}^\infty$, and hence $(\Sigmam)_{m=0}^\infty$ by and elliptic estimates, are Cauchy sequences, converging to some $\Theta$, $\Lambda$, and $V$ respectively.

\underline{\emph{Going back to \eqref{eq:Voriginalintro2}:}} It is now not difficult to show that the undifferentiated version of the equations, that is, \eqref{eq:Voriginalintro2}, holds. For this we start with our Eulerian solutions $(V,D_V\sigma^2,\sigma^2)$ which satisfy 
\begin{align*}
\begin{split}
&\Box V =0\mathrm{~in~}\Omega,\qquad (D_V^2-\frac{1}{2}(\nabla^\alpha\sigma^2)\nabla_\alpha)V_\mu=-\frac{1}{2}\nabla_\alpha D_V\sigma^2 \mathrm{~on~}\partial\Omega,\qquad V\mathrm{~tangent~ to~}\partial\Omega,\\
& \Box D_V\sigma^2 =  4(\nabla^\mu V^\nu)(\nabla_\mu V^\alpha)(\nabla_\alpha V_\nu)+4(\nabla^\mu V^\nu)\nabla_\mu\nabla_\nu \sigma^2,\qquad D_V\sigma\equiv0\mathrm{~on~}\partial\Omega,\qquad \sigma^2\equiv1\mathrm{~on~}\partial\Omega.
\end{split}
\end{align*}
Let $B:=V^\alpha V_\alpha+\sigma^2$, $X_\alpha := D_VV_\alpha+\frac{1}{2}\partial_\alpha\sigma^2$, $\omega_{\mu\nu}:=\partial_\mu V_\nu-\partial_\nu V_\mu$, and $Y_\mu=(D_V^2-\frac{1}{2}(\nabla^\alpha\sigma^2)\nabla_\alpha)V_\mu+\frac{1}{2}\nabla_\alpha D_V\sigma^2$. We need to show that these quantities are identically zero. For this we use the following equations which can be verified by direct computation:
\begin{align*}
\begin{split}
&D_V\Box B = 4(\nabla^\mu V^\nu)\nabla_\mu X_\nu\quad \mathrm{in~}\Omega,\qquad D_V B =  2V^\alpha X_\alpha,\quad \mathrm{in~}\overline{\Omega},\\
&\Box \omega =0,\quad\mathrm{in~}\Omega,\qquad (D_V^2-\frac{1}{2}(\nabla^\alpha\sigma^2)\partial_\alpha)\omega = f(\omega,D_V\omega,\nabla Y,\nabla X),\quad\mathrm{on~}\partial\Omega,\\
&\Box Y = F(\omega,\nabla\omega,\nabla X,\nabla^{(2)}X),\quad\mathrm{in~}\Omega,\qquad Y\equiv0,\quad \mathrm{on~}\partial\Omega,\\
&D_VX = G(\omega,Y),\quad\mathrm{in~}\overline{\Omega}.
\end{split}
\end{align*}
It follows that $(\omega,Y,X)$ satisfy exactly the same type of equation\footnote{The only difference is that $D_VX$ involves both $\omega$ and $Y$, but in our a priori estimates we already encountered $\nabla^{(2)}D_V^{k-1}V$ after commuting $k$ derivatives. See Lemmas~\ref{lem:Vho},~\ref{lem:boxDVkV}, and~\ref{lem:boxDVk1sigma}.} as $(V,D_V\sigma^2,\sigma^2)$ for which we already proved a priori estimates. Therefore, since these quantities vanish initially, they must vanish on all of $\Omega$. Then the equations for $B$ imply that $B$ is also identically zero.
\end{proof}

\section{Newtonian limit: Proof of Theorem~\ref{thm:limit}}\label{sec:Newtonian}
In this final section we present the proof of Theorem~\ref{thm:limit} which at this point is an almost direct consequence of Theorem~\ref{thm:main} and Proposition~\ref{prop:apriori}.
\begin{proof}[Proof of Theorem~\ref{thm:limit}]
Let $\Omega_t:=\Phi(t,\Omega_0)$ where $\Phi$ denotes the flow of $V$ given by $\frac{\ud\Phi^{i}(t,\cdot)}{\ud t}=\left(\frac{\barV^{i}}{\barV^{0}}\right)(t,\cdot)$ and let $\Psi_c(t',\cdot)=\Phi(ct',\cdot)$ so that $\Omega_{ct'}=\Psi_c(t',\Omega_0)$ and $\frac{\ud \Psi}{\ud t'}(t',x')=c\frac{\ud \Phi}{\ud t}(ct',x')$. We work under the hypotheses of Theorem~\ref{thm:limit}.  By Theorem~\ref{thm:main}, for any $c>0$ we have a local in time solution $(V,\sigma^2)$ to \eqref{eq:Voriginalintro2}  and hence \eqref{eq:main}--\eqref{eq:sigma1a}. By Proposition~\ref{prop:apriori} the solution can be extended to $x^0=cT_1$ with $T_1$ independent of $c$. For each $c$ and $t'\in [0,T_1]$ we define (with $x':=\Psi_{c}(t',x'_{0})$ and  $x'_{0}\in\Omega_{0}$)
\begin{align*}
\begin{split}
f_c(t',x'):=\frac{V^0}{c}(ct',\Psi_{c}(t',x'_{0}))-c,\quad v_c^j(t',x'):=\frac{V^j}{c}(ct',\Psi_{c}(t',x'_{0})),\quad h_c(t',x'):=\sigma(ct',\Psi_{c}(t',x'_{0}))-c^2.
\end{split}
\end{align*}
We claim that $(f_{c},v^j_c,\partial_{t'}v^j_c,h_c)$ converge as $c\rightarrow\infty$. Indeed, since the higher order energies of $
\barV$ and $\barsigma$, as defined in Proposition~\ref{prop:apriori}, are uniformly bounded (assuming $K$ in the statement of the theorem is large), by the Rellich-Kondrachov compactness theorem, there is an increasing subsequence $c_\ell\nearrow\infty$ such that for each $t'\in[0,\barT]$, $(f_{c_\ell}(c_{\ell}t',\Psi_{c_{\ell}}(t',\cdot)),v^j_{c_\ell}(c_{\ell}t',\Psi_{c_{\ell}}(t',\cdot)),\partial_{t'}v^j_{c_\ell}(c_{\ell}t',\Psi_{c_{\ell}}(t',\cdot)),h_{c_\ell}(c_{\ell}t',\Psi_{c_{\ell}}(t',\cdot)))$ converge in, say, $H^{10}(\Omega_{0})$ to some $(f(t',\Psi(t',\cdot)),v^j(t',\Psi(t',\cdot)),\partial_{t'}v^j(t',\Psi(t',\cdot)),h(t',\Psi(t',\cdot)))$ with $\Psi(t',\cdot)$ given by \eqref{def Psi}. Now we prove that $(v^j(t',\cdot),\partial_{t'}v^j(t',\cdot),h(t',\cdot))$ is a solution to the free boundary problem \eqref{eq:Newtonian-problem} (similar considerations show that $f=h+\frac{1}{2}|v|^2$ but we do not present the details as this is not needed for the proof).
	 Let us start with the following relation:
	\begin{align}\label{recovered pre 1}
	\begin{split}
		c^{-2}\sigma^2-c^2=&c^{-2}\left((V^0)^2-\sum_{i=1}^{3}(V^i)^{2}-c^{4}\right)=\left(c^{-1}V^0-c\right)^{2}-\sum_{i=1}^{3}\left(c^{-1}V^{i}\right)^{2}+2c(c^{-1}V^{0}-c).
		\end{split}
	\end{align}
	Since $\barsigma^2_{c}, \barV^{i}_{c}$ and $\barV^{0}-c$ remain bounded as $c\rightarrow\infty$, the last term on the right above is bounded and hence. Therefore we have
	\begin{align}\label{B0 recovered}
		c^{-1}V^0-c=O(c^{-1}),\quad \textrm{as}\quad c\rightarrow\infty.
	\end{align}
	Differentiating equation \eqref{recovered pre 1}, we similarly obtain that $c^{-1}\partial_{0}V^{0}=O(c^{-1})$ as $c\rightarrow\infty$. On the other hand, based on \eqref{B0 recovered}, we have, with $x'=\Psi_c(t',x_0')$, 
\begin{align}\label{velocity convergence}
\begin{split}
\lim_{\ell\to\infty} \left(c^{-2}_{\ell}V^{0}(ct',x')\frac{\partial}{\partial t'}+\sum_{i=1}^{3}c_{\ell}^{-1}V^{i}(ct',x')\frac{\partial }{\partial x'}\right) \to \frac{\partial}{\partial t'}+\sum_{i=1}^{3}v^{i}(t',x')\frac{\partial}{\partial x'}.
\end{split}
\end{align}
Therefore, taking the limit in the second equation in \eqref{eq:Voriginalintro2} and the first equation in \eqref{eq:Voriginalintro2} we obtain the first two lines of \eqref{eq:Newtonian-problem}. The boundary condition $h=0$ also follows from taking the limit of the boundary condition $\sigma-c^{2}=0$, and the tangency condition $(1,v)\in\calT(\cup_{t'}(t',\partial\calD_{t'}))$ follows from \eqref{velocity convergence}.


	 Finally, since the solution to \eqref{eq:Newtonian-problem} is unique, any convergent subsequence of $(f_c,v_c,\partial_tv_c,h_c)$ converges to the same limit. Therefore the entire sequence converges to the same limit which is a solution to \eqref{eq:Newtonian-problem}. This completes the proof of Theorem~\ref{thm:limit}.
\end{proof}

\appendix

\section{Non-vanishing vorticity and general sound speed}\label{sec:app general}
Here we discuss how the proof of Theorem~\ref{thm:main} can be adapted with minimal changes to treat the case of barotropic fluids \eqref{isentropic-barotropic1}-\eqref{isentropic-barotropic2}-\eqref{boundary1}-\eqref{boundary2} with non-zero vorticity and general sound speeds. By a slight abuse of notation, we will use $\sigma$ to denote $\|V\|$. The main equation is now
\begin{align}\label{main eqs vor}
	\begin{split}
	\nabla_\mu(G V^\mu) =0 ,\qquad D_VV^\mu+\frac{1}{2}\nabla^\mu\left(\sigma^2\right)=0.
	\end{split}
	\end{align}
 Differentiating the first equation above, a direct computation shows that with the notation $h^{\mu\nu}:=Gm^{\mu\nu}-2G' V^\mu V^\nu$, $V$ satisfies the following \emph{acoustical} wave equation (all indices in this appendix are raised and lowered with respect to $m$):
\begin{align}\label{eq:Vgenint}
\begin{split}
\partial_\mu \big(h^{\mu\nu}\partial_\nu V_\alpha\big) + \partial_\mu (G m^{\mu\nu}\omega_{\alpha\nu})=0\qquad \mathrm{in~}\Omega.
\end{split}
\end{align}
Here we used the fact $i_{V}\omega=0$ which can be derived as follows. Let $\calL_{V}$ be the Lie derivative along the vectorfield $V$ and the $1$-form $\beta$ be $\beta_{\mu}:=m_{\mu\nu}V^{\nu}$. Then the second equation in \eqref{main eqs vor} implies:
	\begin{align*}
	\calL_{V}\beta=-d\,\sigma^{2}
	\end{align*}
	Therefore we have 
	\begin{align*}
	i_{V}\beta=-\sigma^{2},\quad \Rightarrow\quad i_{V}\omega=i_{V}d\beta=\calL_{V}\beta-d\,i_{V}\beta=\calL_{V}\beta+d\,\sigma^{2}=0.
	\end{align*}
It is well-known that the wave equation above for $V$ can be written as the wave operator of a metric conformal to $h$ applied to $V$, but we will not need this formulation. On the boundary (here $D_V:=V^\mu \partial_\mu$)
\begin{align}\label{eq:Vgenb}
\begin{split}
(D_V^2-\frac{1}{2}(\nabla^\mu\sigma^2)\partial_\mu)V_\alpha-\frac{1}{2}(\nabla^\mu\sigma^2)\omega_{\alpha\mu}=-\frac{1}{2}\nabla_\alpha D_V\sigma^2.
\end{split}
\end{align}
Note that since $V$ is tangent to $\partial\Omega$ and $\nabla\sigma^2$ normal, $h^{\mu\nu}\partial_\mu\sigma^2=Gm^{\mu\nu}\partial_\mu\sigma^2$ on $\partial\Omega$. The fact $i_{V}\omega=0$ implies that the vorticity $\omega$ satisfies the transport equation $\calL_V\omega=0$. In coordinates this can be rewritten as
\begin{align}\label{eq:omegagen}
\begin{split}
D_V\omega_{\mu\nu}+(\nabla_\mu V^\lambda)\omega_{\lambda\nu}+(\nabla_\nu V^\lambda)\omega_{\mu\lambda}=0.
\end{split}
\end{align}
For $\sigma^2$ and $D_V\sigma^2$ which are constant on the boundary, we can derive the following interior acoustical wave equations:
\begin{align}\label{eq:sigmagen}
\begin{split}
\partial_\mu(h^{\mu\nu}\partial_\nu \sigma^2)=2G(\nabla^\mu V^\nu)\omega_{\mu\nu}-2h^{\mu\nu}(\nabla_\mu V^\lambda)(\nabla_\nu V_\lambda),
\end{split}
\end{align}
and
\begin{align}\label{eq:DVsigmagen}
\begin{split}
\partial_\mu(h^{\mu\nu}\partial_\nu D_V\sigma^2) = F(\omega,\sigma^2,D_V\sigma^2,\nabla\sigma^2,\nabla D_V\sigma^2,\nabla V,\nabla\omega, \nabla^2{\sigma^2}),
\end{split}
\end{align}
where the right-hand side is given by
\begin{align*}
\begin{split}
F&= -2 D_V\big((\nabla_\mu V^\lambda)(h^{\mu\nu}\nabla_\nu V_\lambda+G m^{\mu\nu}\omega_{\lambda\nu})\big)+(\nabla_\mu V^\lambda)(2h^{\mu\nu}\nabla_\lambda\nabla_\nu\sigma^2+(\nabla_\lambda h^{\mu\nu})\nabla_\nu\sigma^2)\\
&\quad-\nabla_\mu\big((D_Vh^{\mu\nu})\nabla_\nu\sigma^2\big)+(\nabla^\lambda\sigma^2)\nabla_\mu(Gm^{\mu\nu}\omega_{\lambda\nu}).
\end{split}
\end{align*}
This term can be further simplified, but the exact structure is not important for our purposes, except that using the relation $D_VV=-\frac{1}{2}\nabla\sigma^2$, the dependencies of $F$ on the unknowns is as stated in \eqref{eq:DVsigmagen}. At this point we can already see that our proof of a priori estimates for \eqref{eq:main} and \eqref{eq:sigma1a} can be applied to \eqref{eq:Vgenint}, \eqref{eq:Vgenb}, and \eqref{eq:DVsigmagen} with minimal modifications. Indeed, note that since $V$ and $\nabla \sigma^2$ are respectively tangential and normal (with respect to $m$) to $\partial\Omega$, we have $h^{\mu\nu}\partial_\mu\sigma^2 = Gm^{\mu\nu}\partial_\mu\sigma^2= G\nabla^\nu\sigma^2$. It follows that multiplying \eqref{eq:Vgenint} by $D_VV_\alpha$ the resulting boundary flux on $\partial\Omega$ is\footnote{Remarkably, we can also integrate by parts in the expression $(D_VV_\alpha) \partial_\mu(Gm^{\mu\nu}\omega_{\alpha\nu})$ to cancel out the boundary term $\frac{1}{2}(\nabla^\mu\sigma^2)\omega_{\alpha\mu}$ in \eqref{eq:Vgenb}. More precisely, multiplying the equation \eqref{eq:Vgenint} by $D_{V}V_{\alpha}$ ($\alpha$ is not summed) and then integrating in $\Omega$, we obtain
	\begin{align*}
	0=&\int_{\Omega}\partial_{\mu}\left(\left((h^{\mu\nu}\partial_{\nu}V_{\alpha})+\partial_{\mu}(Gm^{\mu\nu}\omega_{\alpha\nu})\right)D_{V}V_{\alpha}\right)dx\,dt\\
	&-\int_{\Omega}\left(\left(h^{\mu\nu}\partial_{\nu}V_{\alpha}+Gm^{\mu\nu}\omega_{\alpha\nu}\right)(\partial_{\mu}V^{\kappa})(\partial_{\kappa}V_{\alpha})\right)dx\,dt\\
	&-\int_{\Omega}\left(D_{V}\left(\frac{h^{\mu\nu}}{2}(\partial_{\nu}V_{\alpha})(\partial_{\mu}V_{\alpha})+Gm^{\mu\nu}\omega_{\alpha\nu}\partial_{\mu}V_{\alpha}\right)\right)dx\,dt\\
	&+\int_{\Omega}\left(\frac12(D_{V}h^{\mu\nu})(\partial_{\nu}V_{\alpha})(\partial_{\mu}V_{\alpha})+(D_{V}G)m^{\mu\nu}\omega_{\alpha\nu}\partial_{\mu}V_{\alpha}+Gm^{\mu\nu}(D_{V}\omega_{\alpha\nu})\partial_{\mu}V_{\alpha}\right)dx\,dt.
	\end{align*}
	Except the first term on the right-hand side above, all the other integrals in $\Omega$ can be treated as lower order terms, similar as in the irrotational hard phase case. The first term on the right-hand side above gives a boundary integral 
	\begin{align*}
	\int_{\partial\Omega}\left((D_{V}V_{\alpha})(\partial^{\nu}\sigma^{2})(\partial_{\nu}V_{\alpha}+\omega_{\alpha\nu})\right)dS\,dt,
	\end{align*}
	which cancels exactly the non-coercive contributions from the left-hand side of the equation \eqref{eq:Vgenb}.}
\begin{align*}
\begin{split}
\int_0^T\int_{\partial\Omega_t}G((\nabla^\mu\sigma^2)\partial_\mu V_\alpha)(D_V V_\alpha)\ud S \ud t,
\end{split}
\end{align*} 
which can be combined with \eqref{eq:Vgenb} to give us control of 
\begin{align*}
\begin{split}
\sup_{t\in[0,T]}\int_{\partial\Omega_t}|D_VV|^2\ud S.
\end{split}
\end{align*}
Similarly, the right-hand sides of equations \eqref{eq:Vgenb} and \eqref{eq:DVsigmagen} have the same regularity structure as those of the corresponding equations in \eqref{eq:main} and \eqref{eq:sigma1a}. Indeed, the contribution of $\sigma^2$ can be treated exactly as before using equation \eqref{eq:sigmagen} (which amounts to the fact that it is one order lower than $D_V\sigma^2$ in terms of our energies). Moreover, the transport equation \eqref{eq:omegagen} shows that in general $D_V^k\omega$ is of the order $\nabla D_V^{k-1}V$. But as observed in Lemmas~\ref{lem:Vho},~\ref{lem:boxDVkV}, and~\ref{lem:boxDVk1sigma}, we already encountered $\nabla D_V^{k-1}V$ in the right-hand side of the boundary equation for $D_V^kV$, and encountered $\nabla^{(2)}D_V^{k-1}V$ in the right-hand side of the interior equations for $D_V^kV$ and $D_V^{k+1}\sigma^2$. Finally, it remains to check the commutator structure between $D_V$ and the acoustical operator $\partial_\mu(h^{\mu\nu}\partial_\nu)$. But again a direct computation using the definition of $h$ shows that for any $\Theta$,
\begin{align*}
\begin{split}
[D_V,\partial_\mu(h^{\mu\nu}\partial_\nu)]\Theta = -(\nabla_\mu V^\lambda)\big((\nabla_\lambda h^{\mu\nu})\nabla_\nu \Theta+2h^{\mu\nu}\nabla_\mu\nabla_\nu\Theta\big)+\nabla_\mu\big((D_Vh^{\mu\nu})\nabla_\nu\Theta\big)+(\nabla_\mu(Gm^{\mu\nu}\omega_{\lambda\nu}))\nabla^\lambda\Theta.
\end{split}
\end{align*}
Comparing with the commutator identity for $[D_V,\Box]$ from equation \eqref{eq:comtemp3}, we see that the right-hand side above has exactly the same regularity as the case we already treated in our a priori estimates. Indeed, the only difference is the appearance of second order derivatives of $V$, but these always come with lower orders of $D_V$ and as mentioned above were already encountered in Lemmas~\ref{lem:boxDVkV} and~\ref{lem:boxDVk1sigma}. 

\bibliographystyle{plain}
\bibliography{ref}

\bigskip

\centerline{\scshape Shuang Miao}
\smallskip
{\footnotesize
 \centerline{School of Mathematics and Statistics, Wuhan University}
\centerline{Wuhan, Hubei, 430072, China}
\centerline{\email{shuang.m@whu.edu.cn}}
} 

 \medskip

\centerline{\scshape Sohrab Shahshahani}
\medskip
{\footnotesize
 \centerline{Department of Mathematics and Statistics, University of Massachusetts}
\centerline{Lederle Graduate Research Tower, 710 N. Pleasant Street,
Amherst, MA 01003-9305, U.S.A.}
\centerline{\email{sohrab@math.umass.edu}}
}

\medskip

\centerline{\scshape Sijue Wu}
\medskip
{\footnotesize
 \centerline{Department of Mathematics, University of Michigan}
\centerline{East Hall, 530 Church Street,
Ann Arbor, MI 48109-1043, U.S.A.}
\centerline{\email{sijue@umich.edu}}
}

\end{document}